\documentclass{article}
\usepackage[T1]{fontenc}
\usepackage[utf8]{inputenc}
\usepackage[english]{babel}
\usepackage{lmodern}
\usepackage[margin=3cm]{geometry}

\usepackage[backend=biber,style=alphabetic,sorting=nty,doi=false,isbn=false,url=false,minalphanames=2,maxbibnames=4,maxcitenames=4]{biblatex}
\AtEveryBibitem{\clearlist{language}}

\usepackage{empheq}
\usepackage{csquotes}
\usepackage{authblk}
\usepackage[hidelinks, colorlinks=false]{hyperref}
\usepackage{enumitem}
\usepackage{tikz-cd} 	
\usepackage{alltt,bigints,bbm,slashed,mathtools,amssymb,amsmath,amsfonts,amsthm} 
\usepackage{tikz-cd} 	
\usepackage[nottoc]{tocbibind}
\usepackage{graphicx}
\usepackage[font={small,it}]{caption}
\usepackage{subcaption}
\usepackage{aliascnt}
\usepackage{cases}
\usepackage{comment}	
\usepackage[capitalize,nameinlink]{cleveref}
\usepackage{tikz}
\usetikzlibrary{tikzmark}
\usepackage{slashed}

\crefname{figure}{Fig.}{Fig.}
\Crefname{figure}{Fig.}{Fig.}
\crefname{subfigure}{Fig.}{Fig.}
\Crefname{subfigure}{Fig.}{Fig.}

\usepackage{xargs}   

\usepackage{tocloft}
\crefformat{equation}{(#2#1#3)}
\numberwithin{equation}{section}

\theoremstyle{definition}

\newtheorem*{theorem*}{Theorem}

\theoremstyle{plain}
\newtheorem{theorem}{Theorem}[section]
\newtheorem{corollary}{Corollary}[section]
\newtheorem{lemma}{Lemma}[section]
\newtheorem{prop}{Proposition}[section]
\newtheorem{conjecture}{Conjecture}[section]
\newtheorem*{claim}{Claim}
\newtheorem{definition}{Definition}[section]

\theoremstyle{remark}
\newtheorem{remark}{Remark}[section]
\newtheorem{notation}{Notation}[section]

\crefname{lemma}{Lemma}{Lemmas}
\crefname{prop}{Proposition}{Proposition}
\crefname{conjecture}{Conjecture}{Conjecture}
\crefname{corollary}{Corollary}{Corollary}
\crefname{remark}{Remark}{Remark}
\crefname{defi}{Definition}{Definition}
\crefname{equation}{}{}
\crefname{enumi}{}{}
\crefname{appendix}{}{}

\newcommand{\dd}{\mathop{}\!\mathrm{d}}
\newcommand{\Ve}{\mathcal{V}_{\mathrm{e}}}
\newcommand{\Vb}{\mathcal{V}_{\mathrm{b}}}
\newcommand{\M}{\mathcal{M}}
\newcommand{\A}{\mathcal{A}_{\mathrm{b}}}
\newcommand{\Hb}{{H}_{\mathrm{b}}}
\newcommand{\Diffb}{\mathrm{Diff}_{\mathrm{b}}}

\graphicspath{{figures}}

\newenvironment{nalign}{
	\begin{equation}
		\begin{aligned}
		}{
		\end{aligned}
	\end{equation}
	\ignorespacesafterend
}

\newcommand{\C}{\mathcal{C}}
\newcommand{\N}{\mathbb{N}}

\newcommand{\R}{\mathbb{R}}
\newcommand{\T}{\mathbb{T}}
\newcommand{\Z}{\mathbb{Z}}

\newcommand{\abs}[1]{\left\lvert #1\right\rvert}

\newcommand{\jpns}[1]{\langle #1 \rangle}
\newcommand{\norm}[1]{\left\lVert #1\right\rVert}

\newcommand{\D}{\mathcal{D}}

\newcommand{\Diff}{\mathrm{Diff}}

\DeclareMathOperator{\supp}{supp}



\title{Smooth finite time singularity formation without quantization}
\author[1]{Istvan Kadar\thanks{istvan.kadar@math.ethz.ch}}
\affil[1]{\small \textit{Department of Mathematics, ETH Zurich, R\"amistrasse 101, 8006 Zurich, Switzerland}}

\usepackage[makeroom]{cancel}
\graphicspath{{figures}}

\renewcommand{\M}{\mathcal{M}}

\newcommand{\Mcomp}{\overline{\M}}
\newcommand{\Mcompg}{\overline{\mathbf{M}}}
\newcommand{\blambda}{\boldsymbol{\lambda}}
\renewcommand{\P}{\mathcal{P}}
\renewcommand{\A}{\mathcal{A}_{\mathrm{b}}}

\newcommand{\Aext}{\bar{\mathcal{A}}_{\mathrm{b}}}

\renewcommand{\b}{\mathrm{b}}
\renewcommand{\Hb}{H_{\mathrm{b}}}

\renewcommand{\Vb}{\mathcal{V}_{\mathrm{b}} }

\renewcommand{\Ve}{\mathcal{V}_{\mathrm{t}} }

\newcommand{\Rcompd}[1]{\overline{\R^{#1}}}
\newcommand{\ext}{\mathrm{ext}}
\renewcommand{\Diffb}{\Diff_{\mathrm{b}}}
\newcommand{\interval}{[0,1)}
\newcommand{\K}{\mathcal{K}}
\newcommand{\normal}[1]{A(D_{\phi_0}P,#1)}
\newcommand{\sphere}{\mathbb{S}^2}

\newcommand{\Tlin}{\stackrel{\scalebox{.6}{\mbox{\tiny (1)}}}{\T}}
\renewcommand{\S}{\mathcal{S}}
\renewcommand{\div}{\mathrm{div}}

\renewcommand{\div}{\mathrm{div}}

\begin{document}
	\maketitle
	\begin{abstract}
		We revisit the finite time singularity formation of Krieger-Schlag-Tataru \cite{krieger_slow_2009} for the focusing energy critical wave equation in $\mathbb{R}^{3+1}$ from a geometric singular-analytic point of view, following  Hintz \cite{hintz_lectures_2023}.
		We construct $C^{\nu/2-}$ regular approximate solutions that settle down to multiple solitons, shrinking at a rate $t^{\nu}$ with $\nu>1$, and approaching the origin on different geodesics $\{x=zt\}$.
		By fine tuning the velocities, sizes and signs of the solitons, we are able to construct \emph{smooth} ansätze with any $\nu>8$.
		Using robust energy estimates, the ansätze are corrected to exact solutions.
	\end{abstract}
	
	\setcounter{tocdepth}{1}
	\tableofcontents
	
	\section{Introduction}\label{sec:intro}
	In this paper, we study finite time singularity formation\footnote{here, we avoid the usage of the equivalent term blow-up, so as to avoid confusion with the geometric notion} for the energy critical wave equation
	\begin{equation}\label{in:eq:critical}
		P[\phi]:=\Box\phi+\phi^5:=(-\partial_t^2+\Delta)\phi+\phi^5=0
	\end{equation}
	in $\M:=\{t\in(0,1),\abs{x}\leq t\}\subset \R^{3+1}$.
	It is well known that \cref{in:eq:critical} admits a stationary solution, called ground state soliton $W:=\sqrt{3}(1+3\abs{x}^2)^{-1/2}\sim r^{-1}$ as $r\to\infty$.
	Moreover, the Poincaré and scaling symmetries of \cref{in:eq:critical} allow for the construction of the family of solutions for $z\in\mathring{B}:=\{x\in\R^3:\abs{x}<1\}$ and $\blambda\in\R_{>0}$
	\begin{equation}\label{in:eq:soliton_family}
		 W^{\blambda,z}:=\blambda^{1/2}W(\blambda x_z),\quad x_z:=x+(\gamma_z-1)(x\cdot\hat{z})\hat{z}-\gamma_ztz,\quad \gamma_z:=(1-\abs{z}^2)^{1/2},\quad \hat{z}=\frac{z}{\abs{z}}.
	\end{equation}
	
	The construction of solutions here mimics the pioneering work of Krieger-Schlag-Tataru in \cite{krieger_slow_2009}, from a geometric singular-analytic point of view\footnote{we already point out we \emph{do not} use microlocal propagation estimates and the proof is only based on the so called vectorfield method}, as discussed, for instance, in \cite{hintz_lectures_2023}.
	We state a preliminary version of the main theorem; a pictorial representation is given in \cref{fig:singularity_picture}.
	\begin{theorem}\label{in:thm:main_rough}
			Given a finite collection of velocities $\{z\in A\}=A\subset \mathring{B}:=\{x\in\R^3:\abs{x}<1\}$, scales $\blambda_z\in\R_{\neq0}$, signs $\sigma_z\in\{\pm1\}$ and singularity speed $\nu\in\R_{>1}$, there exists a solution, $\phi\in C^{\nu/2-}$, of \cref{in:eq:critical} in $\{t\in(0,T),\abs{x}<t\}$ for $T\ll1$ such that
			\begin{equation}
				\norm{\phi-\sum_z \sigma_zW^{\blambda_z\cdot t^{-\nu},z}}_{\dot{H}^{1}(\abs{x}<t)}\to0.
			\end{equation}
			Moreover, for $\nu>8$ and a well chosen soliton data $(A,\blambda,\sigma)$ depending on $\nu$, we can construct solutions that are  \emph{smooth} up to and including the lightcone $\phi\in C^\infty(\{t\in(0,T);\abs{x}\leq t\})$.
	\end{theorem}
	We also prove the 3 dimensional analogue of \cite{hillairet_smooth_2012}, and show smooth singularity formation by adiabatic shrinking to a single soliton for quantised values of $\nu\in2\N_{\geq2}$.
	
	Finally, our results on the solution in the lightcone are sufficient --contain enough regularity-- to apply the extension theorem of \cite{kadar_note_2026} and conclude existence of solutions all the way to the Cauchy horizon:
	\begin{corollary}\label{in:cor:Cauchy}
		Fix $\nu>1$ and a $C^{\nu/2-}$ (or $C^\infty$) solution $\phi$ from \cref{in:thm:main_rough} inside $\{t\in(0,T);\abs{x}\leq t\}$.
		Let $(\phi_0,\phi_1):\R^3\to\R$ be a $C^{\nu/2-}$ (or $C^\infty$) compactly supported extension of the Cauchy data on the hypersurfaces $\{t=1\}$.
		Then, the maximal hyperbolic development of \cref{in:eq:critical} with $(\phi,T\phi)|_{\{t=1\}}=(\phi_0,\phi_1)$ includes $\{t+\abs{x}>0, 0<\abs{x}-t\ll1 \}$.
		
		For $\nu>10$, there exists an extension such that the maximal hyperbolic development is $\{t+\abs{x}>0, \abs{x}-t<\infty \}$.
	\end{corollary}

	\emph{Overview of the introduction:}
	We first discuss previous results on singularity formation for \cref{in:eq:critical} in \cref{in:sec:singularity}.
	We focus on the two main approaches, and compare their differences.
	Then, in \cref{in:sec:geometry}, we discuss \cite{krieger_slow_2009} in more detail and introduce some extra rudimentary geometric language to explain part of their construction.
	Finally, we state the precise results in \cref{in:sec:main} and give the mains ideas of the proofs in \cref{in:sec:proof}.

	\subsection{Singularity formation: scattering versus evolution}\label{in:sec:singularity}
	
	The energy-critical wave equation belongs to a broad class of dispersive--more generally evolutionary--equations for which soliton-resolution conjecture is an important guiding principle regarding the dynamical behaviour of their solutions \cite{duyckaerts_classification_2013,duyckaerts_soliton_2022,duyckaerts_soliton_2023,jendrej_soliton_2025}.
	The conjecture colloquially says that, generic maximally extended solutions of \cref{in:eq:critical} decouple into a sum of moving dynamically scaled solitons and radiation, including the possibility of global and finite time solutions.
	Given the breadth of work related to this conjecture, we do not aim for a comprehensive review. Instead, we highlight a selection of results most relevant to \cref{in:thm:main_rough}, and refer the interested reader to \cite{duyckaerts_soliton_2022,jendrej_soliton_2025} for a more detailed historical overview and additional references.
	\cref{in:thm:main_rough}  is about \emph{constructing} special solutions that fall into this conjecture, without \emph{any} stability or classification considerations.
	
	Regarding the construction of solutions forming singularity -at a faster than self similar rate\footnote{see \cite{shatah_weak_1988,biernat_hyperboloidal_2021,merle_determination_2003} and references therein for self-similar singularities}- in line with the conjecture there are two main approaches.
	In both cases, the starting point is an ansatz $\phi_A$ that solves the equation with sufficiently fast decaying error. 
	In the first case, exemplified by \cite{merle_universality_2004,rodnianski_formation_2010,raphael_stable_2012}, one solves a Cauchy problem with initial data chosen close to the ansatz and shows convergence to a dynamically modified family of approximate solutions.
	On \cref{fig:singularity_picture}, this involves posing initial data on $\Sigma$ and evolving in the negative time direction.
	This method propagates \emph{towards} the singularity, and therefore also establishes finite codimension stability at the same time.
	We call this the \emph{evolutionary} approach, which has been wildly successful as indicated for instance by \cite{hillairet_smooth_2012,merle_blow_2011-2,raphael_quantized_2014,collot_type_2018,merle_type_2015}.

	Let us also mention the work \cite{jendrej_dynamics_2022}, where singularity forming solutions to the energy critical wave maps were constructed by prescribing the leading order behaviour of $\phi|_{t=0}$ as $r\to0$.
	This prescription is equivalent choosing $\phi|_{\C}$.
	In a sense, our approach to prove \cref{in:thm:main_rough} is similar, but we show how to choose these radiative terms in order to obtain smooth solutions.
	
	\begin{figure}[htbp]
		\centering
		\includegraphics[width=260pt]{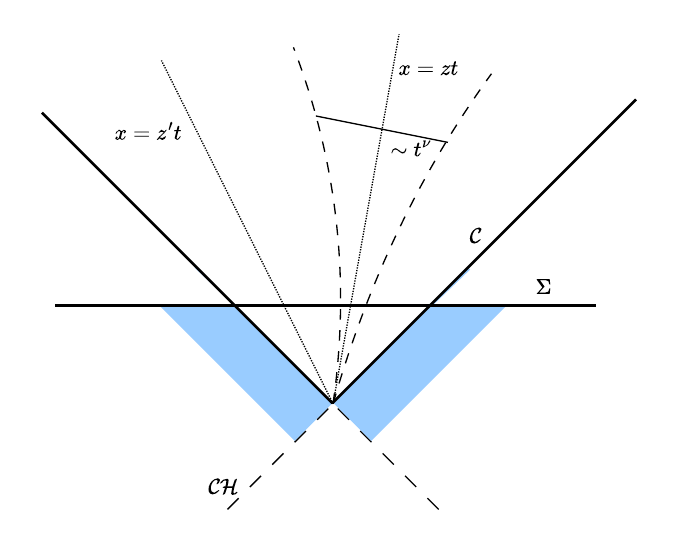}
		\caption{We construct the solution in within the lightcone $\C$.
			The solitons move along the $\{x_z=0\},\{x_{z'}=0\}$ curves, and we also indicated the size of the $\{x_z=0\}$ soliton by dashed line.
			In the exterior region, we indicated where the solution can be extended using \cite{kadar_note_2026}.
			The solution can be regarded as part of the maximal development of some data prescribed on the Cauchy hypersurface $\Sigma$.}
		\label{fig:singularity_picture}
	\end{figure}
	
	Alternatively, one can construct the ansatz $\phi_A$ to a much higher precision, as pioneered in \cite{krieger_slow_2009,krieger_renormalization_2008}, and solve for the error term \emph{from} the singularity.
	By finite speed of propagation it suffices to solve a singular characteristic initial value problem with data specified on $\C$.
	Once a solution is constructed within the lightcone, $\M$, the corresponding Cauchy data along $\Sigma$ may be extended arbitrarily to obtain--via finite speed of propagation--a solution forming the singularity at the tip of the cone.\footnote{While this does not cover the global behaviour and it is unclear what happens in the maximal hyperbolic development, for special choices of extension we can guarantee that $\{t=x=0\}$ is the only singularity, see \cite{kadar_note_2026}.}
	We call this the \emph{scattering} approach, which was also used in many settings including
	\cite{donninger_exotic_2012,krieger_full_2014,donninger_nonscattering_2013,jendrej_concentric_2025}.
	
	Let us compare these two approaches, emphasising the points important to the present paper:
	\begin{itemize}
		\item 
		The solutions constructed via scattering give a more explicit description of $\phi$.
		Since $\phi_A$ can be constructed to arbitrary precision (i.e. $P[\phi_A]=\mathcal{O}(t^N)$ for any $N$) and the remainder, $\phi-\phi_A=\mathcal{O}(t^N)$, decays fast, the solutions are computable to an unrestricted order.
		\item 
		The ansatz $\phi_A$, as well as $\phi$, is smooth in $t>0$ for the evolutionary approach.
		Indeed, since $\phi$ is constructed from smooth initial data posed on $\Sigma$, the regularity is propagated from $\Sigma$.
		This also yields finite codimension stability at the same time.
		\item
		The scattering approach can yield a continuous family of singularity speed $\lambda\sim t^{-\nu}$ for $\nu\in\R_{>1}$ or even more exotic function as shown in \cite{donninger_exotic_2012}.
		On the contrary, in the evolutionary approach, singularity formation (to a single soliton)  is expected to be quantised, see \cite{raphael_quantized_2014,kim_rigidity_2025,jeong_classification_2026} for related discussion.\footnote{We mention that in \cite{martel_blow_2015}, a continuum of $\nu$ are constructed for the $L^2$ critical gKdV equation, but there the singularity formation is influenced by the \emph{tail} of the data, an effect clearly absent for wave equation due to the finite speed of propagation. See also the multi-soliton cases below for further discussion speeds.}
		\item The evolutionary approach usually requires some fine structure of the nonlinear equation, such as an appropriate factorisation, as used in \cite{rodnianski_formation_2010,raphael_stable_2012}.
		In contrast, the scattering approach can utilise that the ansatz $\phi_A$ is already given precisely, and in some cases that $P[\phi_A]=\mathcal{O}(t^N)$ for $N\gg1$.
	\end{itemize}
	
	In this paper, we construct singularities via scattering using the framework of singular geometric-analysis as presented for instance in \cite{hintz_lectures_2023}.
	
	\paragraph{Many body problems.}
	Although this paper considers finite time singularity formation, the other emphasis in \cref{in:thm:main_rough} is the multi body dynamics.
	In this respect, the problem closely resembles the global-in-time constructions in  \cite{martel_construction_2016,martel_inelasticity_2018,jendrej_construction_2020,yuan_multi-solitons_2019}, where solutions are constructed with multiple solitons moving away from one another,\footnote{this is true with respect to the scale of the solitons, for instance in \cite{jendrej_construction_2020}, the solitons are a finite distance away with respect to the Minkowski distance, but their sizes are shrinking} and only interacting via long range effects.
	The interaction strength between solitons is governed by the slow decay of their tails. In this context, lower spatial dimensions correspond to stronger interactions. Notably, both in \cite{kadar_construction_2024} and the present paper the $W(r)\sim r^{-1}$ decay rate plays a crucial role.
	Indeed this work mostly follows the same approach as \cite{kadar_construction_2024}, with a different proof strategy --see \cref{in:sec:proof}-- to give a more streamlined proof of \cref{in:thm:main_rough}.

	Finite time singularity formation for dispersive equations of many body problems have already appeared  for instance in \cite{cote_construction_2013,martel_strongly_2018} (see \cite{collot_singularity_2024,kim_rigidity_2026} and references therein for parabolic equations).
	In \cite{cote_construction_2013}  a superposition of \emph{self-similar} singularities is studied and there is no decoupling of scales, whereas the latter two are more similar to the present setting.
	Although \cref{fig:singularity_picture} shows that the soliton distances with respect to the Minkowski metric $(\eta)$ shrink at a self similar speed ($t$), their sizes shrink at a faster ($t^\nu$) rate, implying that the appropriate interpretation is one of separating solitons rather than colliding ones.
	We also note that in \cite{martel_strongly_2018,collot_singularity_2024,kim_rigidity_2026} the non-radial solutions yields \emph{new} singularity rates compared to a single one, but whether a quantization of rates happens outside symmetry is unclear.
	Here, in \cref{in:thm:main_rough}, we show that such a statement is false for \cref{in:eq:critical}, but see \cref{in:conj:classification} for more discussion.
	
	\subsection{Geometric setup}\label{in:sec:geometry}
	In this section, we recast the construction of the ansatz $\phi_A$ of \cite{krieger_slow_2009} within a geometric framework. 
	Fix a singularity speed $\lambda(t)=t^{-\nu}$.
	
	We study \cref{in:eq:critical} on a compact manifold $\Mcomp$ that satisfies $\mathring{\Mcomp}=\{\abs{x}<t,t\in(0,1)\}$.
	It is defined, by extending the coordinates
	$\rho_-=\jpns{x\lambda(t)}^{-\frac{1}{\nu-1}}$ and  $\rho_\K=t\jpns{x\lambda(t)}^{\frac{1}{\nu-1}}$ to $0$ smoothly.\footnote{In the case of multiple singularity directions $A\subset B=\{\abs{z}<1\}$, we also extend $\rho_{\K_z}=t\jpns{x_z\lambda(t)}^{\frac{1}{\nu-1}}$ to 0, but we stick to $A=\{0\}$ for simplicity.}
	We call the corresponding zero sets $\K=\{\rho_\K=0\}$ and $i_-=\{\rho_-=0\}$.
	The choice of these coordinates is motivated by the observation $t=\mathcal{O}(\rho_-\rho_\K)$, but see \cref{not:rem:coordinate_choice} for further discussion.
	See \cref{fig:comp} for a visual depiction.
	
	As already observed in \cite{krieger_slow_2009}, the linearised operator
	\begin{equation}
		P_W=\Box+5 \lambda^2W^4(\lambda x)
	\end{equation}
	has different leading order behaviour near $\K$ and $i_-$ when acting on function that are conormal, i.e.~for $f:\Mcomp_0\to\R$ which are arbitrarily regular with respect to vectorfields becoming tangential near the boundary $\Vb(\Mcomp_0)=\{\rho_-\partial_{\rho_-},\rho_\K\partial_{\rho_\K},x_i\partial_{x_j}-x_j\partial_{x_i}\}$.
	In particular, \cite{krieger_slow_2009} show that for $f\in\mathcal{O}(\rho_-^{a_-}\rho_\K^{a_\K})$ it holds that
	\begin{nalign}
		P_Wf&\in\mathcal{O}(\rho_-^{a_--2}\rho_\K^{a_\K-2\nu}),\\
		P_W f-\Box f&\in\mathcal{O}(\rho_-^{a_-+4(\nu/2-1)}\rho_\K^{a_\K-2\nu}),\\
		P_W f-\big(\lambda^2\Delta_{\lambda x}+5\lambda^4W^4(\lambda x)\big)f&\in \mathcal{O}(\rho_-^{a_--2}\rho_\K^{a_\K-\nu-1}).
	\end{nalign}
	
	Next, \cite{krieger_slow_2009} solve the model problems associated to $\Box$ and $(\Delta+5W^4)$ in the regions $i_-$ and $\K$ iteratively.
	Both of these operators yield elliptic equations, and so appropriate boundary conditions must be specified.
	Here, we diverge slightly from the discussion of \cite{krieger_slow_2009} to provide a more geometric explanation of the boundary data.
	
	Near the boundary $i_-$, we want to invert $\Box$ on functions that have the form $\mathcal{O}(\rho_-^{a_-}\rho_\K^{a_\K})$.
	This problem only yields a unique solution in the future lightcone if in addition, we also pick $\phi|_{\C}\in\mathcal{O}(\rho_-^{a_-+2})$.
	This extra choice is our freedom in the scattering construction to determine what radiation we inject into the spacetime.
	For generic inhomogeneity and characteristic data with this decay, however, the solution will not be smooth across $\C$; instead,
	on the slice $\{t=1\}$, it will be regular up to a conormal term of order $\mathcal{O}((\abs{x}-1)^{a_-+3})$, i.e. will only belong to $C^{a_-+3-}$.
	In \cref{in:sec:proof}, we discuss the choice of radiation and explain how to use it to modulate the size of the solitons.

	Near $\K$, the relevant operator $(\Delta+5W^4)$ is not invertible due to the explicit resonance and eigenfunction $\Lambda W:=(1/2+x\partial_x)W$, $\partial_x W$.
	Therefore, we must ensure, that the right hand side is orthogonal to both $\Lambda W$ and $\partial_x W$.
	Once this modulation is achieved, a unique inverse may be selected provided that we choose a complementary space to $\mathrm{span}\{\Lambda W,\partial_x W\}$.
	The choices are discussed further in \cref{in:sec:proof}.

	\subsection{Main result}\label{in:sec:main}
	As we already mentioned, \cref{in:thm:main_rough} is constructed in two steps.
	Correspondingly, the main result of the paper has a part regarding an approximate solution and one with the correction to an exact solution.
	
	We construct 3 different types of ansätze.
	First, we prove a generalisation of the approximate solution constructed in \cite{krieger_slow_2009}, by allowing the solution to form a singularity along multiple trajectories.
	\begin{prop}[Irregular solution, \cref{an:prop:multi}]\label{in:prop:an_irreg}
		Fix singularity speed $\nu>1$,  soliton velocities $A$, scales $\blambda\in\R^{\abs{A}}_{>0}$, and signs $\sigma\in\{\pm1\}^{\abs{A}}$.
		There exists an ansatz $\phi_A\in C^{\nu/2-}$, solving $P[\phi_A]=\mathcal{O}(\rho_\K^\infty\rho_-^\infty)$ and satisfying the following conditions
		\begin{subequations}\label{in:eq:pointwise_convergences}
			\begin{align}
				\lim_{\substack{t\to0\\
						t_z^{-\nu}x_z=y}} t_z^{\nu/2}\phi_A-\blambda_z^{1/2} \sigma_zW(\blambda_z y)&=0, &&\forall z\in A, y\in\R^3;\\
				\lim_{\substack{t\to0\\
						x=t\bar{z}} }t^{-\nu/2+1+\epsilon}\phi_A&=0, &&\forall z\notin A, \epsilon>0.
			\end{align}
		\end{subequations}
		More precisely $\phi_A$ is conormal with $t^{\nu/2-1}$ and  $t^{-\nu/2}$ leading order behaviour towards $i_-$ and $\K_z$ respectively (written as $\phi_A=\mathcal{O}(\rho_-^{\nu/2-1}\rho_\K^{-\nu/2})$ ) and $t_z^{\nu/2}\phi_A|_{\K_z}=\blambda_z^{1/2}\sigma_zW(\blambda_zy)$.
		As a corollary, $\phi_A$ also converges in $\dot{H}^1\times L^2$ to the sum of scaled solitons within the lightcone.
	\end{prop}
	
	Some remarks are in order:
	\begin{remark}[Singularity speed]
		Although, we stated the result when all the solitons have the same speed of singularity formation, we expect the methods of the paper to be applicable with only slight notational modification when these speeds are allowed to vary for each, so that $\nu\in\R_{>1}^{\abs{A}}$.
		Similarly, following \cite{donninger_exotic_2012}, more general conormal velocities may be allowed, i.e. $\lambda(t)=\mathcal{O}(t^{-\nu})$ with infinite regularity with respect to $t\partial_t$ and the requirement that $t\dot{\lambda}/\lambda$ is uniformly bounded away from 0.
		One only need to improve the Minkowskian \cref{mod:lem:rad_multi} result to treat the leading order modulation.
		We only show one extension along these lines, following \cite{jendrej_concentric_2025}, with  $\lambda$ forming a singularity exponentially, see \cref{an:prop:exponential}.
	\end{remark} 
	
	\begin{remark}[Outgoing radiation]
		The construction also allows us to observe that the approximate solution along the cone $\C$ satisfies $\phi|_\C=\mathcal{O}(r^{\nu/2-1})$.
		This in turn allows us to study the solution in the exterior of the lightcone via \cref{in:cor:Cauchy}.
	\end{remark}
	
	\begin{remark}[Perturbed equations]
		Naturally, one can add lower order terms to the equation, as long as the model operators as explained in \cref{in:sec:geometry} are not effected.
		For instance, we expect the methods to apply to $\Box\phi=-\phi^5+c\phi^3$ for $c\in\R$.
	\end{remark}
	
	\begin{remark}[Expansion]
		Although we construct $\phi_A$ in an iterative procedure, we do not keep track of fine structure of $\phi_A$, other than the decay rate and regularity.
		In particular, no polyhomogeneous expansion, similar to that in \cite{krieger_slow_2009}, or outside of symmetry in \cite{kadar_construction_2024}, is proved.
		Nonetheless, the techniques of \cite{kadar_construction_2024} should yield such an expansion.
		Instead of such an explicit ansatz, we construct $\phi_A$ in function spaces without explicit structure following \cite{angelopoulos_matching_2026}.
		We believe that this has applicability in even space dimensions as well, where no polyhomogeneous expansion is expected, see \cite{raphael_stable_2012,hillairet_smooth_2012,raphael_quantized_2014}.
	\end{remark}

	The next two results concerns smooth singularity formation, i.e. $\phi_A\in C^\infty$.
	In accordance with the expectations from the evolutionary approach, discussed in \cref{in:sec:singularity}, we construct approximate multi-soliton solutions with quantised rates of singularity speeds.
	\begin{prop}[Smooth quantised, \cref{an:prop:smooth_quantized} ]\label{in:prop:an_quant}
		Fix $A$ soliton velocities with $\abs{A}\leq3$, and arbitrary scales and signs.
		Then, for singularity speed $\nu$ satisfying $\nu/2-1\in\N_{\geq1}$, there exists a smooth ansatz,  $\phi_A\in C^\infty(\{t\in(0,T);\abs{x}\leq t\})$ for some $T>0$ with decay rate $\phi_A=\mathcal{O}(\rho_-^{\nu/2-1}\rho_\K^{-\nu/2})$ satisfying \cref{in:eq:pointwise_convergences} and $P[\phi_A]=\mathcal{O}(\rho_\K^\infty\rho_-^\infty)$.
		When $\abs{A}=1$, $\nu=2$ is also allowed.
	\end{prop}

	\begin{remark}[Restriction on the number of solitons]
		As will be explained in \cref{in:sec:proof}, the limitation in \cref{in:prop:an_quant} on the number of solitons allowed boils down to the cancellation of leading order modulation parameters by a radiative term $\phi_{\mathrm{rad}}$ solving $\Box\phi_{\mathrm{rad}}=0$ with decay $\phi_{\mathrm{rad}}|_{r=0}\sim t^{\nu/2-1}$.
		We require $\phi_{\mathrm{rad}}$ to be smooth across $\C$ and to attain certain limiting behaviour along the geodesics $\{x_z=0\}$.
		This requirement becomes increasingly difficult with larger $\abs{A}$, and we expect that \cref{in:prop:an_quant} will only be possible if the minimum value of $\nu$ increases with $\abs{A}$.
		See \cref{an:sec:rad} for more on this issue.
	\end{remark}
	
	\begin{remark}[Codimension stability for smooth solutions]
		Although we do not address the issue of stability in the present work, we allow ourselves to provide some heuristics for our expectations in the smooth case.
		Let's first restrict to the case $\abs{A}=1$, in the spherically symmetric setting. For any $\nu'\leq \nu$ such that $\nu'/2-1\in\N_{\geq0}$, there is a radiative smooth perturbation\footnote{explicitly given as $r\phi_\mathrm{rad}=(t+r)^{\nu/2}-(t-r)^{\nu/2}$} $\Box\phi_{\mathrm{rad}}=0$ with  decay rate $t^{\nu'/2-1}$ exciting at the linear level a more violent, or equally violet when $\nu=\nu'$, change to the modulation parameter than the one used in the construction.
		To be more explicit, the $\nu\in2\N_{\geq1}$ solution should be codimension $\nu/2$ unstable due to this effect.
		These and the time translation symmetry should account for all unstable perturbations.
		Outside of spherical symmetry, we additionally need to include the space translations.
		For $\abs{A}\geq2$, the different spacetime translation of each soliton yields an instability together with all the perturbations
		\begin{equation}
			\{\phi_{\mathrm{rad}}\in C^\infty:\Box\phi_{\mathrm{rad}}=0 ,\exists\nu'\leq \nu \text{ such that } \lim_{t\to0}\phi_{\mathrm{rad}}|_{x_z=0}t^{1-\nu'/2}\neq0\}
		\end{equation}
		exciting the modulation parameters of the solitons.
	\end{remark}
	
	Finally, we can use not only radiation as explained above to tune the leading modulation equation, but also $A,\sigma,\lambda$.
	We show that the speed of singularity can take a continuum of values even for smooth solutions.
	\begin{prop}[Smooth, \cref{an:prop:smooth_multi}]\label{in:prop:an_new}
		Let $\nu>8$. 
		There exist parameters $A,\sigma,\blambda$, depending on $\nu$, and a corresponding smooth ansatz $\phi\in C^\infty(\{t\in(0,T);\abs{x}\leq t\})$ satisfying the rate $\phi_A=\mathcal{O}(\rho_-^{\nu/2-1}\rho_\K^{-\nu/2})$ together with  \cref{in:eq:pointwise_convergences} and $P[\phi_A]=\mathcal{O}(\rho_\K^\infty\rho_-^\infty)$.
		In fact, we always take $\abs{A}\in\{2,3\}$.
	\end{prop}
	
	\begin{remark}[Classification]
		As noted earlier, most results on smooth singularity formation identify a discrete collection of possible speeds $\nu$.
		Indeed, for the case of a single soliton even classification results exists (see \cite{jeong_classification_2026} for a dispersive equation and references therein).
		Allowing multiple solitons introduces additional mechanisms that lead to new discrete admissible speeds, as explained in \cref{in:sec:singularity}.
		Although under certain assumptions--as in \cite{kim_rigidity_2026} for the nonlinear heat equation--it is still possible to derive a sharp classification yielding a discrete set of speeds, our result shows that such a discrete classification of $\nu$ cannot be achieved for equation \cref{in:eq:critical}. Instead, at best, the admissible parameter tuples $(A, \sigma, \blambda, \nu)$ form a discrete set, see \cref{in:conj:classification}.
	\end{remark}
	
	\begin{remark}[Multiple speeds]
		We believe that the methods developed in this paper are equally reliable to construct solutions with multiple different singularity speeds.
		For this, we only need to guarantee that a leading order orthogonality condition is satisfied, see \cref{in:sec:proof} for more details.
		As this would introduce much extra necessary notation we decided not to pursue such directions in this preliminary work.
	\end{remark}
	
	\begin{remark}[Smoothness]
		We explain in \cref{in:sec:proof} why smooth solutions arise, but we already remark that the smoothness is \emph{not} a consequence of the irregular parts of \cref{in:prop:an_irreg} cancelling.
	\end{remark}
	
	All previously constructed approximate solutions fit into the energy estimate framework that we develop to construct exact solutions.
	\begin{theorem}\label{in:thm:correction}
			For an approximate solution $\phi_A$ as in \cref{in:prop:an_irreg,in:prop:an_quant,in:prop:an_new} there exists $\bar{\phi}=\phi-\phi_A\in\mathcal{O}(\rho_-^\infty\rho_\K^\infty)$ such that $\phi$ solves \cref{in:eq:critical}.
			Moreover, when $\phi_A$ is smooth, $\phi$ can be extended to a smooth solution in the region $\{t+r\in(0,v),t-r>-\delta\}$ for some $\delta,v>0$, and when\footnote{We already mention that the restriction $\nu>10$ is not expected to be necessary for the global result, \cite{kadar_note_2026} only uses this for convenience.} $\nu>10$ we may take $\delta=\infty$.
	\end{theorem}
	The extension to the exterior immediately follows from \cite{kadar_note_2026}.

	\subsection{Ideas of the proof}\label{in:sec:proof}
	The main ideas of the paper can be split into 3 different observations.
	Let us discuss each in turn, focusing mostly on a single soliton with parameters $(z=0,\lambda_0=1,\sigma_0=1)$, and only highlighting the main difference for the case of multiple solitons.
	Throughout the discussion here, we ignore the modulation of the position and the corresponding issues regarding $\partial_x W\in\ker (\Delta+5W^4)$, as they are similar and simpler than those for $\Lambda W$.

	\emph{Correct time dependence of scaling.}
	As already noted in \cite{kadar_scattering_2024}, the slow decay of the resonance $\Lambda W=(1/2+x\partial_x)W\sim r^{-1}$ causes difficulties when constructing an ansatz near $\K$, as this naively introduces a term near $i_-$ with very slow decay.
	To shortcut the difficulties--and reduce the complexities of \cite{kadar_construction_2024}--it is convenient to let $\lambda$ depend on the asymptotically null coordinate
	\begin{equation}
		v=\begin{cases}
			t& y<10,\\
			t-\abs{x} & y>20,
		\end{cases}
	\end{equation}
	where $y=\lambda(v) x$.
	With this choice, one quickly observes the improvement of the naiv ansatz
	\begin{equation}\label{in:eq:1st_ansatz}
		P[W^{\lambda(t)}]=\mathcal{O}(\rho_-^{\nu/2-3}\rho_\K^{-\nu/2-2}) \text{ replaced with } P[W^{\lambda(v)}]=\mathcal{O}(\rho_-^{5(\nu/2-1)}\rho_\K^{-3\nu/2-1}).
	\end{equation}
	Note that the improvement only happens near $i_-$, and the behaviour near $\K$ is worse.
	However, we will argue that this seemingly worse behaviour is actually the optimal that one should observe.\footnote{For those accustomed to the global problem, formally $\nu=0$, we note that using $v$ dependence already gives that for a resonance, the time decay is one power slower than the inhomogeneity, as already seen from \cite{jensen_spectral_1979}.}
	
	\emph{Modulation with radiation and solitons.}
	Using the improved behaviour in \cref{in:eq:1st_ansatz}, one can compute the projection of the error term onto the kernel $\Lambda W(y)$
	\begin{equation}\label{in:eq:Lambda_first_proj}
		\P_{\Lambda}P[W^{\lambda(v)}](\tau)=\int_{t=\tau} \Lambda W(y) P[W^{\lambda(v)}]\sim -\pi\nu\blambda_0^{3/2}\tau^{-3\nu/2-1},
	\end{equation}
	where we included the dependence on $\lambda_0$, otherwise normalised to 1.
	See \cref{an:eq:leading_Kernel} for precise formula.
	After some computation one notices that a modulation term $\lambda(\tau)=\tau^{-\nu}+\bar{\lambda}$ with $\bar{\lambda}\in\mathcal{O}(\tau^{-\nu'})$, where $\nu'<\nu$, changes \cref{in:eq:Lambda_first_proj} by $\tau^{-\nu'/2-\nu-1}$, thus we \emph{cannot} remove the leading term in \cref{in:eq:Lambda_first_proj} through modulating $\lambda$ at a subleading  rate.
	
	On the other hand, adding a multiple of the explicit solution of the wave equation $\phi_{\mathrm{rad}}=((t+r)^{\nu/2}-(t-r)^{\nu/2} )/r\in\mathcal{O}(\rho_+^{\nu/2-1}\rho_K^{\nu/2-1})$ to $W^{\lambda(v)}$  yields\footnote{In the evolutionary approach of \cite{raphael_stable_2012} we believe that such terms correspond to the tail behaviour, see for instance \cite[Eq. (1.30), (1.31)]{raphael_stable_2012}.}
	\begin{equation}
		P[W^{\lambda(v)}+\beta\phi_{\mathrm{rad}}]=P[W^{\lambda(v)}]+\beta\phi_{\mathrm{rad}}(t,0) \lambda^{2}(t)W^4(y)+\mathcal{O}(\rho_-^{5(\nu/2-1)}\rho_\K^{-\nu/2-2}).
	\end{equation}
	Observing that $\lambda^2(t)\phi_{\mathrm{rad}}(t,0)= \nu t^{-3\nu/2-1}$, we can choose $\beta$ such that this radiation modified ansatz has a vanishing kernel projection.
	The quantisation for the singularity rates for smooth solutions in \cref{in:prop:an_quant} is simply a consequence of when $\phi_{\mathrm{rad}}$ is smooth across $\C$.
	
	After the precise leading order term has been corrected for with radiation, a straightforward iteration scheme yields improving solutions, provided that one also modulates $\lambda$ (and the location so far ignored here).
	
	Provided that we consider multiple solitons, $\abs{A}\geq2$, the leading order projection in \cref{in:eq:Lambda_first_proj}, up to an overall factor of $\tau^{-3\nu/2-1}$,  changes to\footnote{While we do not use the setup of the evolutionary approach--such as \cite{raphael_stable_2012}--we expect that the equation here signals that the ODEs driving the finite dimensional dynamical system for the different solitons couple at leading order.} 
	\begin{equation}\label{in:eq:Lambda_multi}
		-\pi\nu\blambda_0^{3/2}+\blambda_0^2\sum_{\bar{z}\in A\setminus \{0\}}2\pi\sigma_{\bar{z}}\blambda^{-1/2}_{\bar{z}}\frac{(1+\abs{\bar{z}})^{\nu/2}}{\abs{\bar{z}}}\gamma_{\bar{z}}^{\nu/2-1},
\end{equation}
	The second term is very similar to that appearing in \cite[Eq. (3.24)]{kadar_construction_2024}, and expresses the fact that the soliton scales are coupled at leading order.
	Provided that one is able to choose the free parameters $(\lambda,\sigma,A)$ such that \cref{in:eq:Lambda_multi} vanishes, the previously mentioned iteration scheme yields the approximate solutions.
	Indeed, we believe that this is a necessary requirement:
	\begin{conjecture}\label{in:conj:classification}
		Fix soliton parameters $(A,\blambda,\nu,\sigma)$.
		The following two are equivalent:
		\begin{itemize}
			\item There exists smooth singularity forming solution to \cref{in:eq:critical} satisfying \cref{in:eq:pointwise_convergences}.
			\item There exists a smooth solution to $\Box\phi_{\mathrm{rad}}=0$ with $\norm{\phi t^{\nu/2-1}}_{L^\infty(\M)}<\infty$ and $t_z^{\nu/2-1}\phi_{\mathrm{rad}}|_{x_z=0}=\mathfrak{r}_{z}$ such that
			\begin{equation}
				-\pi\nu\blambda_0^{3/2}+\blambda_0^2\sum_{\bar{z}\in A\setminus \{0\}}2\pi\sigma_{\bar{z}}\blambda^{-1/2}_{\bar{z}}\frac{(1+\abs{{\bar{z}}})^{\nu/2}}{\abs{{\bar{z}}}}\gamma_{\bar{z}}^{\nu/2-1}+2\pi\mathfrak{r}_0=0,
			\end{equation}
			and similar Lorentz boosted version holds for all $z\in A$.
		\end{itemize}
	\end{conjecture}
	In this paper we prove that the second implies the first.
	
	\emph{Linear energy estimate and non-linear solution.}
	Let us comment on the strategy to obtain an exact solution and prove \cref{in:thm:correction}.
	We study the linearised problem
	\begin{equation}
		D_{\phi_A}P\phi:=(\Box+5\phi_A^4)\phi=f,\qquad \phi|_{\C}=0.
	\end{equation}
	He use finite regularity weighted $L^2$ based Sobolev spaces denoted by $\Hb^{s;\vec{a}}$ (and introduced in \cref{not:sec:energy-spaces}) capturing decay towards $i_-$ and $\K_z$.
	A trivial $\partial_t$ energy estimate yields for some weights $\vec{a},\vec{a}^f\in\R^3$ and a collection of vectorfields $\Ve$
	\begin{equation}\label{lin:eq:T_energy}
		\norm{\Ve\phi}_{\Hb^{0;\vec{a}}}\lesssim \norm{f}_{\Hb^{0;\vec{a}^f}}+\norm{\phi}_{\Hb^{0;\vec{a}}}.
	\end{equation}
	The right hand side contains a $\phi$ dependent term due to the fact that the linearised operator $(\Delta+5W^4)$ is not negative definite.
	The two obstructions to this come from the kernel elements $\ker(\Delta+5W^4)=\{\Lambda W,\partial W\}$ and the non-zero eigenvalue $\kappa>0$ with eigenfunction $\kappa^2Y=(\Delta+5W^4)Y$.
	Essentially, we want to use the well known coercivity estimate
	\begin{equation}\label{in:eq:coercivity}
		(-\phi,(\Delta+V)\phi)+C_Y(Y,\phi)^2+\sum_\nu(f_\nu,\phi)^2\geq C \norm{\phi}_{\dot{H}^1(\R^3)},
	\end{equation}
	where $f_\nu\in \dot{H}^{-1}$ are some projections onto the kernel elements.
	We treat the two obstructions separately.
	
	For the kernel elements, we revisit the ideas of \cite{kadar_construction_2024} and construct 1 forms, $J$, the integrals of which over well chosen hypersurfaces give a leading order contribution when evaluated on the kernel elements.
	Thereby, we control $(f_\nu,\phi)$ of \cref{in:eq:coercivity} in term of bulk error integrals.
	
	For $(Y,\phi)$, we split it into a forward and a backward stable mode $\alpha^\pm$.
	For the forward stable mode, taking $t_2>t_1$, we can bound  $\alpha^+(t_2)$ by $\alpha^+(t_1)$ plus a bulk error term. For $\phi$ decaying sufficiently fast, we obtain $\alpha^+(t_1)\to0$, and thus we only need to bound the error terms.
	For the backward stable mode, we have the exact opposite relation and we can bound $\alpha^-(t_1)$ in terms of $\alpha^-(t_2)$.
	Therefore, we must have an a priori control on $\phi$ in some bounded time $t\in[t_*,2t_*]$ for $t_*>0$ to bound $\alpha^-(t_2)$.
	
	In conclusion, we can improve \cref{lin:eq:T_energy} to
	\begin{equation}
		\norm{\Ve\phi}_{\Hb^{0;\vec{a}}}\lesssim \norm{f}_{\Hb^{0;\vec{a}^f}}+\norm{ \phi}_{\Hb^{1;}(\{t\in[t_*,2t_*]\})}.
	\end{equation}
	Next, instead of keeping the above estimate essentially sharp-- as in \cite{kadar_construction_2024}-- and commuting with $\Vb$ vectorfields, we only commute with unit coordinate vectorfields $\partial\in\{\partial_t,\partial_x\}$.
	These being ignorant about the geometry means that the weights $\vec{a},\vec{a}^f$ will not be respected and the higher order estimates will have a regularity dependent weight
	\begin{equation}\label{in:eq:semi-Fred}
		\norm{\Ve\phi}_{\Hb^{s;\vec{a}_s}}\lesssim \norm{f}_{\Hb^{s;\vec{a}^f_s}}+\norm{ \phi}_{\Hb^{1;}(\{t\in[t_*,2t_*]\})},
	\end{equation}
	where we left out $\vec{a}$ on the last term, as it does not touch any of the boundaries of $\Mcomp$.
	We can afford this loss as we already have arbitrary fast decaying errors of the approximate solutions from the previous step.
	These simple commutations significantly reduces the complexity of the proof.
	This is the major difference compared to \cite{kadar_construction_2024} and gives a considerably easier proof.
	See \cref{sec:lin} for more details.
	
	The term on the right hand side of \cref{in:eq:semi-Fred} is a compact error.
	At this point, we simply observe that \cref{in:eq:semi-Fred} and the well posedness of the wave equation in low regularity spaces yields that $D_{\phi_A}P$ is a surjective operator.
	
	Using the strong weights in the decay estimates allows us to close the nonlinear problem by contraction mapping theorem.

	\paragraph{Overview:}
	We start in \cref{sec:setup} by discussing the geometric and functional framework in more detail.
	Then, we record some standard mapping properties and invertibility statements in \cref{sec:mapping,sec:pre} respectively.
	We use these to construct the approximate solutions in \cref{sec:an}.
	Then, we develop linear scattering estimates and construct a right inverse of the linearisation $D_{\phi_A}P$ in \cref{sec:lin}, which in turn is used in \cref{sec:nonlin} to prove \cref{in:thm:correction} within the lightcone.

	\paragraph{Acknowledgement:}
	I thank the Princeton Gravity Initiative for its hospitality where most of the mathematical work was conducted. I am grateful to Igor Rodnianski for many helpful discussions and support.

	\section{Geometric and analytic setup}\label{sec:setup}
	Throughout the paper, we will always use convenient coordinates to perform analytic computations.
	For this purpose, whenever partial derivative appear, it is to be interpretted that the remaining coordinates are kept fixed.
	For instance given a cooridnate change $(x,t)\mapsto (y,\tau)$, we will write
	\begin{equation}
		\partial_x=\partial_x|_t=\frac{\partial y}{\partial x}\Big|_{t}\partial_y|_\tau+\frac{\partial\tau}{\partial x}\Big|_t\partial_\tau|_y=\Big(\frac{\partial y}{\partial x}\Big|_{t}\Big)\partial_y+\Big(\frac{\partial\tau}{\partial x}\Big|_t\Big)\partial_\tau
	\end{equation} 
	
	We also use multiindex notation.
	In particular, for a finite collection of differential operators $\mathcal{V}=\{\Gamma_1,...,\Gamma_m\}$, $k\in\N$ and a norm $\norm{\cdot}_{X}$ on smooth functions, we write
	\begin{equation}
		\norm{\mathcal{V}^k f}_X=\sum_{\abs{\alpha}=k}\norm{\big(\prod_i\Gamma^{\alpha_i}_i\big)f}_X.
	\end{equation}
	Similarly, we will also use the shorthand in formulae $X\mathcal{V}f$ to be
	\begin{equation}
		X\mathcal{V}f=\sum_{i=1}^m f_i \Gamma_i f,\text{for some }f_i\in X.
	\end{equation}

	\subsection{Conormal functions}
	In this section, we introduce the basic analytic framework of the paper, which is geometric singular analysis following \cite{hintz_lectures_2023}.
	
	\begin{definition}[Conormal operators]
		For a manifold with corners $X$, we write $\Diffb(X)$ for smooth vectorfields on $X$ that become tangential at $\partial X$.
		We write $\Vb(X)$ for a finite subset of $\Diffb(X)$ such that $\Vb(X)$ spans $\Diffb(X)$ over $C^\infty(X)$.
		Finally, we write $\Diffb^k(X)$ for $k$ fold composition of $\Diffb(X)$.
		
		Fix $Y\subset X$ a subset of $X$ with a finite number of boundary hypersurfaces removed, $X\setminus Y\subset\partial X$.
		We write $\Diffb(Y)$ for smooth vectorfields on $X$ that become tangent on $\partial Y$.
		We also call $\partial X\setminus\partial Y$ artificial boundaries of $Y$.
	\end{definition}

	\begin{definition}[Conormal function]
		On a manifold with corners $X$,	we define the \emph{conoromal} space of functions
		\begin{equation}
			\A^s(X):=\{f\in L^\infty(X):  \Diffb^s(X)f\in L^\infty(X)\}.
		\end{equation}
		Denoting by $\rho_i$ the defining function of the $i$-th boundary component of $\partial X$, we also introduce weighted conormal spaces
		\begin{equation}
			\A^{s;\vec{a}}(X):=\{w f\in L^\infty(X):  w\Diffb^s(X)f\in L^\infty(X)\},
		\end{equation}
		where $w=\prod_i \rho_i^{-a_i}$.
		The associated norm $\norm{\cdot}_{\A(X)}$ on $\A(X)$ can be defined by a choice of $\Vb(X)$ spanning $\Diffb$.
		
		We also write $\A^{;\vec{a}}:=\cap_{s}\A^{s;\vec{a}}$, and $\A^{\vec{a}-}:=\cup_{\epsilon>0} \A^{\vec{a}-\epsilon}$, $\A^{\vec{a}+}:=\cap_{\epsilon>0} \A^{\vec{a}+\epsilon}$.
	\end{definition}
	
	\begin{notation}
		Although $\A^{s;\vec{a}-}$ has no associated norm, we will write
		\begin{equation}
			\norm{\phi}_{\A^{s;\vec{a}-}}\lesssim A \iff \forall\epsilon>0\, \exists C_\epsilon>0 \text{ such that }\norm{\phi}_{\A^{\vec{a}-\epsilon}}\leq C_\epsilon A.
		\end{equation}
	\end{notation}
	
	Some of the canonical example that we will use are 
	\begin{itemize}
		\item $\interval_t$, where $t$ is used as a boundary defining function of $0$;
		\item $\A(\interval;\R^n)$, where we use $\R^n$ with the trivial inner product as a vector space;
		\item $\overline{\R^3}=\{\abs{x}< 2\}\cup_{x\mapsto (\abs{x}^{-1},\hat{x})}\{(R,\omega)\in\left([0,1),\sphere\right)\}$, where $R$ is the boundary defining function of $\partial\overline{\R^3}=\{R=0\}\times \sphere$.
	\end{itemize}
	
		\subsection{Unit size coordinate functions}
	We will mostly be concerned with solutions to \cref{in:eq:critical} in the region $\M:=\{t\in(0,1),\abs{x}\leq t\}$.
	It is useful to introduce the self similar variable $\rho=x/t$ to parametrise parts of $\M$.
	
	We use $\eta=-\dd t^2+\dd x^2$ to denote the Minkowski metric, with corresponding null coordinates $\mathring{u}=\frac{t-r}{2},\mathring{v}=\frac{t+r}{2}$.
	Let us introduce some useful coordinates around the individual solitons.
	Fix $h\in C^\infty(\R)$ with $h(r)|_{r<R_2-3}=1-h'|_{r>R_2}=0$ and $h'\in[0,1]$ for some $R_2$ to be used in \cref{sec:lin}.
	Fix $\lambda\in\R_{>0}$ and $\zeta\in C^\infty(\R)$ and define
	\begin{equation}\label{not:eq:v_def}
		\bar{v}_\lambda=t+h(\abs{\bar x_\lambda})/\lambda,\quad \bar{x}_\lambda=\lambda(x-\zeta(t)).
	\end{equation}
	For $\lambda=1$, we drop the subscript.
	In $\{t,x\},\{t,\bar{x}\}$ and $\{\bar{v},\bar{x}\}$ coordinates, the coordinate derivatives and the wave equation takes the form
	\begin{nalign}
		&\partial_x|_t=\lambda\partial_{\bar{x}}|_t=\lambda\partial_{\bar{x}}|_{\bar{v}}+h'\partial_{\bar{v}}|_{\bar{x}}\\
		&\partial_t|_x=\partial_t|_{\bar{x}}-\lambda\dot{\zeta}\cdot\partial_{\bar{x}}|_{t}=\partial_v|_{\bar{x}}-\lambda\dot\zeta\cdot\partial_{\bar{x}}|_v-h'\hat{\bar{x}}\cdot\dot{\zeta}\partial_{\bar{v}}|_{\bar{x}}
	\end{nalign}
	\begin{nalign}\label{not:eq:frozen_Box}
		\Box:=\Box_\eta&=-\partial_t^2+\Delta_x\\
		&=-\partial_t^2+\lambda\ddot{\zeta}\cdot\partial_{\bar{x}}-\lambda^2(\dot\zeta\cdot\partial_{\bar{x}})^2+2\lambda\dot\zeta\cdot\partial_{\bar{x}}\partial_t+\lambda^2\Delta_{\bar{x}}
		\\&\begin{multlined}
			=(-1+h'(\bar{x})^2)\partial_{\bar{v}}^2+\lambda^2\Delta_{\bar{x}}+\lambda\underbrace{\Big(\frac{2h'(\bar{x})}{\abs{\bar{x}}}+2h'(\bar{x})\partial_{\abs{\bar{x}}}+h''(\bar{x})\Big)}_{=:L_1}\partial_{\bar{v}}\\
			+\lambda\ddot{\zeta}\cdot\partial_{\bar{x}}-\big(\dot\zeta\cdot(\lambda\partial_{\bar{x}}+h'\hat{\bar{x}}\partial_{\bar{v}})\big)^2+2\lambda\dot\zeta\cdot\partial_{\bar{x}}\partial_{\bar{v}}+\ddot{\zeta}\cdot\hat{\bar{x}}\partial_{\bar{v}}+2h' \dot{\zeta}\cdot\hat{\bar{x}}\partial_{\bar{v}}^2
		\end{multlined}
	\end{nalign}
	where $\hat{a}=a/\abs{a}$.
	Let us already note the important computation for $f=c/\jpns{r}+\mathcal{O}(r^{-2})$
	\begin{equation}
		\int_{\R^3}\dd x^3 fL_1f=\int_{\R^3} r^{-2}h'\partial_r (rf)^2+h'' f^2\dd x^3=c^2.
	\end{equation}
	We also introduce the Killing vectorfield $T=\partial_t|_x$.
	
	Next, we recall the boosted coordinates
	\begin{equation}
		x_z=x+(\gamma_z-1)(x\cdot\hat{z})\hat{z}-\gamma_ztz,\quad t_z=\gamma_z(t-z\cdot x),\quad \gamma_z=(1-\abs{z}^2)^{1/2},
	\end{equation}
	which in the case of a boost in the direction $z=(z^1,0,0)$ simplifies to 
	\begin{equation}
		(x^1_z,x^2_z,x^3_z)=(\gamma_z(x-z^1t),x^2,x^3),\quad t_z=\gamma_z(t-z^1x^1)
	\end{equation}
	and the corresponding time dilation equation $t|_{t^{-\nu}_zx_z=c}=\gamma_z^{-1} t_z+\mathcal{O}(t_z^{\nu-1})$.
	Importantly, the Minkowski metric is invariant under the boosts $\eta:=-\dd t^2+\dd x^2=-\dd t_z^2+\dd x_z^2$.
	We introduce $\rho_z,\bar{v}_z,h_z,\zeta_z,\bar{x}_z$ for the boosted versions of the functions defined above.
	For instance $v_z(t,x):=v(t_z(t,x),x_z(t,x))$.
	
	We also recall the relativistic velocity change formula
	\begin{equation}\label{not:eq:velocity_change}
		z_{\bar{z}}=\frac{1}{1-z\cdot \bar{z}}\Big(z-\gamma_{z}^{-1}\bar z-\frac{\gamma_{z}}{1+\gamma_{z}}(z\cdot\bar z)z\Big),\quad \gamma_{z,\bar{z}}=\gamma_{\bar{z},z}:=(1-\abs{z_{\bar{z}}}^2)^{-1/2}.
	\end{equation}
	Via \cref{not:eq:velocity_change}, we can express $\{x_z,t_z\}$ in terms of $\{x_{\bar{z}},t_{\bar{z}}\}$
	\begin{equation}
			x_z=x_{\bar{z}}+(\gamma_{\bar{z},z}-1)(x_{\bar{z}}\cdot\widehat{z_{\bar{z}}})\widehat{z_{\bar{z}}}-\gamma_{\bar{z},z}t_{\bar{z}}z_{\bar{z}},\quad t_z=\gamma_{z,\bar{z}}(t_{\bar{z}}-z_{\bar{z}}\cdot x_{\bar{z}}).
	\end{equation}
	We note that in the case that $z,\bar{z}$ are collinear \cref{not:eq:velocity_change} simplifies to 
	\begin{equation}\label{not:eq:velocity_paralell}
		z_{\bar{z}}=\frac{z-\bar{z}}{1-z\cdot\bar{z}}.
	\end{equation}

	\subsection{Compactification}\label{not:sec:comp}
	Fix a non-empty finite subset $A\subset \mathring{B}:=\{x\in\R^3:\abs{x}<1\}$, corresponding to soliton velocities, and without loss of generality assume that $0\in A$.
	Fix $\nu>1$ singularity rate.
	
	Let us introduce $\lambda_z(s)=\blambda_z s^{-\nu}$ together with $b_z=\dot{\lambda}_z/\lambda_z$  and $\zeta_z\in\A^{\nu}(\interval;\R^3)$.
	We parametrise the solutions close to the solitons using the following set of coordinates
	\begin{equation}\label{not:eq:vy_def}
		v_z:=t_z+h(y_z)/\lambda_z(v_z),\quad y_z:=\lambda_z(v_z)\big(x_z-\zeta_z(t_z)\big).
	\end{equation}
	\begin{lemma}\label{not:lemma:diffeo}
		The map $(t_z,x_z)\mapsto(v_z,y_z)$ is a diffeomorphism for $t$ sufficiently small.
	\end{lemma}
	\begin{proof}
		It suffices to prove this for $z=0$, and we drop the $0$ subscript.
		Along $\{x=\zeta\}$ it holds that $t=v$.
		Similarly, we easily obtain for $\abs{x}<t/2$ that $v\in(t,2t)$.
		We may compute that $\partial_t v=1-hb\lambda^{-1}\partial_tv+b\lambda^{-1}y\cdot\nabla h \partial_t v-\dot{\zeta}\cdot\nabla h$, and so $\partial_t v=(1-\dot{\zeta})/(1-b\lambda^{-1}(h-y\cdot\nabla h))$, and so $v$ is a monotone function of $t$.
		Bijection follows trivially.
		We may also compute the Jakobian
		\begin{equation}
			\det\begin{pmatrix}
				\partial_v t&\partial_y t\\
				\partial_v x& \partial_y x
			\end{pmatrix}=
			\lambda^{-1}\det\begin{pmatrix}
				1+\frac{hb}{\lambda} t& -\frac{h'}{\lambda}\\
				-{yb}+\dot{\zeta}(\lambda+hb)& 1-h'\dot{\zeta}
			\end{pmatrix}
			=\lambda^{-1}\Big(1+\frac{hb-h'yb}{\lambda}\Big).
		\end{equation}
		Noting that $h-h'y\in\A^{0,0}(\Rcompd{3}_y)$, implies invertiblity and since all the functions $h,\lambda,\zeta$ are smooth, we also get smoothness.
	\end{proof}
	
	The individual solitons will be positioned along $\{\abs{x_z}/t=0\}$ and are on the scale $y_z$.
	Let's define the coordinates
	\begin{equation}\label{not:eq:boundary_def_functions}
		\rho_{z}=t\jpns{y_z}^{\frac{1}{\nu-1}},\quad \rho_-=t\prod_{z\in A} \rho_z^{-1},\quad\rho_\C=(\abs{x}-t)/t.
	\end{equation}
	
	\begin{definition}\label{not:def:Mcomp}
		We define the manifold with corners $\Mcomp_A$ to be the compactification of $\M$, where $\rho_z,\rho_-$ are extended smoothly to 0.
		We also introduce $\Mcomp_z=\Mcomp_{\{z\}}$ for $z\in A$.
		We also introduce $\Mcomp_\emptyset$ for the compactification where $t$ and $\rho_\C=(\abs{x}-t)/t$ are extended smoothly to 0.
		See \cref{fig:comp}.
		
		For $\vec{a}=(a_-,a_{z_1},...,a_{z_{\abs{A}} } )$, we define
		\begin{equation}
			\Aext^{k;\vec{a}}(\Mcomp)=\{f\in\A^{0,\vec{a}}(\Mcomp):\{1,t\partial_x,t\partial_t\}^kf\in \A^{0,\vec{a}}(\Mcomp)\}=\A^{k;\vec{a}}(\Mcomp\setminus\C).
		\end{equation}
	\end{definition}
	We write $\Mcomp=\Mcomp_A$ whenever $A$ is fixed and keep the dependence on $A$ and $\nu$ implicit.
	
	The zero sets correspond to 
	\begin{itemize}
		\item $\K_z:=\{\rho_z=0\}$ the $z$-th soliton face;
		\item $i_-:=\{\rho_-=0\}$ self similar face;
		\item $\C:=\{\rho_\C=0\}$ future null cone of the singularity.
	\end{itemize}
	
	\begin{notation}
		We label the weights from the far end to the near end on $\Mcomp$.
		For the weights $a_\C,a_-\in\R$ and  $\vec{a}_\K=(a_1,a_2,...,a_{\abs{A}})\in\R^{\abs{A}}$, we write $\A^{a_\C,a_-,\vec{a}_\K}(\Mcomp)$ for the corresponding conormal function space.
		We also introduce $\rho_\K:=\prod_{z\in A} \rho_z$, and write $\A^{a_\C,a_-,a_\K}(\Mcomp)$ when $a_z=a_\K$ for all $z\in A$. 
	\end{notation}
	
	\begin{remark}[Choice of coordinates]\label{not:rem:coordinate_choice}
		Let us note that the powers of the boundary defining functions could have been picked differently, for instance picking $\tau=\rho_z^{\nu-1}$ instead $\rho_z$ corresponds to the common choice $\frac{\dd\tau}{\dd t}=\lambda (t)$.
		Our choice is motivated by the sharp inclusion $t\in\A^{0,1,1}(\Mcomp)$.
		This unfortunately implies  that other geometric quantities get extra factors of $\nu$ appear in their expansion, such as $y\in\A^{0,1-\nu,0}(\Mcomp_0)$, and see \cref{not:eq:sizes} for others.
		Normalising instead $y=\rho_-^{-1}$ and $t^{\nu-1}=\rho_-\rho_\K$ would have potentially reduced the factors of $\nu$ appearing.
	\end{remark}

		\begin{figure}[htbp]
		\centering
		\includegraphics[width=0.7\textwidth]{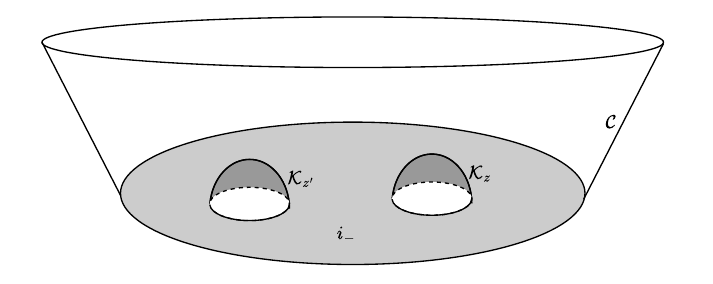}
		\caption{Depicted is the compactification of $\Mcomp$, with boundary components $\C,i_-,\K_z$ and local respective local coordinates $\{t,x/\abs{x}\},\{x/t\},\{y_z\}$. }
		\label{fig:comp}
	\end{figure}
	
	Fix $\delta\ll1$ such that $\{\abs{x_z}/t_z<100\delta\}$ and $\{1-\abs{x}/t<100\delta\}$ are all disjoint.
	We introduce the following spacetime regions 
	\begin{itemize}
		\item $\D^{\delta'}_z=\{\rho_z<\delta'\}$ $\D_z=\D^\delta_z$;
		\item $\D^{\delta'}_{\mathrm{ext}}=\{\rho>1-\delta'\}$, $\D_{\mathrm{ext}}=\D^{\delta}_{\mathrm{ext}}$;
		\item $\D_{\mathrm{s}}=\Mcomp\setminus (\D^{\delta/2}_z\cup\D^{\delta/2}_{\mathrm{ext}})$
	\end{itemize}
	
	It will be helpful for the energy estimates\footnote{the coordinates $\mathbf{t},\mathbf{x}$ are the usual self-similar coordinates as used for instance in \cite{rodnianski_formation_2010,raphael_stable_2012}} to also introduce coordinate $\mathbf{t}=t t^{-\nu}/(1-\nu)$, $\mathbf{x}=(x-\zeta(t))t^{-\nu}$  in the region $\mathbf{u}:=\mathbf{t}-\mathbf{\abs{x}}>1$.
	We will only use this in the region $\D_0$.
	\begin{definition}\label{not:def:comp_global}
		Let's define $\Mcompg$ as the compactification of $\D_0$, where $\boldsymbol{\rho}_{\K}:=\jpns{\mathbf{x}}/\mathbf{t}$ and $\boldsymbol{\rho}_{+}:=1/\jpns{\mathbf{x}}$ are smoothly extended to 0.
	\end{definition}
	Let us record the difference between $\Mcompg$ and $\Mcomp_0$ on $\D_0$.
	\begin{equation}\label{not:eq:comp_comparison}
		\A^{k;a,b}(\Mcompg\cap\D_0)= \A^{k;a(\nu-1),b(\nu-1)}(\Mcomp_0\cap\D_0).
	\end{equation}
	
	Finally, we record the form of $\Box$ in the above introduced coordinates
	\begin{lemma}[Wave operator]\label{not:lemma:Box}
		Fix $\lambda_0=v_0^{-\nu}+\A^{;-1}(\interval)$, $\zeta_0\in\A^{;\nu}(\interval;\R^3)$ and $b=\dot{\lambda}/\lambda$.
		Then in $v,y$ coordinates, setting $\Gamma=(\partial_{v}+by\cdot\partial_y)$ and $L_y=\Big(\frac{2h'(y)}{\abs{y}}+2h'(y)\partial_{\abs{y}}+h''(y)\Big)$
		\begin{nalign}
			&\begin{multlined}
				\Box=(-1+h'(y)^2)\Gamma^2+\lambda^2\Delta_{y}-\lambda L_y\Gamma
				+\lambda\ddot{\zeta}\cdot\partial_{y}\\
				-\big(\dot\zeta\cdot(\lambda\partial_{y}+h'\hat{y}\Gamma)\big)^2+2\lambda\dot\zeta\cdot\partial_{y}\Gamma+\ddot{\zeta}\cdot\hat{y}\Gamma+2h' \dot{\zeta}\cdot\hat{y}\Gamma^2
			\end{multlined}\\
			&
				=(-1+h'(y)^2)\Gamma^2+\lambda^2\Delta_{y}+\lambda L_y\Gamma+\lambda\ddot{\zeta}\cdot\partial_y-(\lambda\dot{\zeta}\cdot\partial_y)^2
				+2\lambda\dot{\zeta}\cdot\partial_y\Gamma+\A^{;2\nu-4,\nu-3}(\Mcomp)\Diffb^2.
		\end{nalign}
		
		With respect to $\mathbf{t}$ and $\mathbf{x}$ we can also write the  metric and the wave operator as
		\begin{nalign}\label{not:eq:box_nu}
			t^{2\nu}\Box&=-\partial_{\mathbf{t}}^2+\Delta_{\mathbf{x}}+\A^{;0,0}(\Mcompg)(\jpns{\mathbf{x}}^{-1}\mathbf{t}^{-1}\Ve^2(\Mcompg)\cap \mathbf{t}^{-2}\Vb^2(\Mcompg))\\
			t^{-2\nu}\eta&=-\dd \mathbf{t}^2+\dd\mathbf{x}^2+Q\Big(\dd \mathbf{x},\frac{\jpns{\mathbf{x}}}{\mathbf{t}}\dd \mathbf{t}\Big)
		\end{nalign}
		where the vectorfields are $\Ve(\Mcompg)=\{1,\jpns{\mathbf{x}}\partial_{\mathbf{x}},\jpns{\mathbf{x}}\partial_{\mathbf{t}}\}$, $\Vb(\Mcompg)=\{1,\jpns{\mathbf{x}}\partial_{\mathbf{x}},\jpns{\mathbf{t}}\partial_{\mathbf{t}}\}$ and $Q$ is a quadratic form in its arguments without $\dd \mathbf{x}^2$ terms and $\A^{;0,0}(\Mcompg)$ coefficients.
	\end{lemma}
	
	\begin{proof}
		This is an explicit computation and easily follows using \cref{not:eq:frozen_Box} and letting $\lambda$ be $v$ dependent.
		
		For the second part, we compute 
		\begin{nalign}\label{not:eq:partial_nu}
			\partial_{t}|_{x}&=t^{-\nu}\partial_{\mathbf{t}}|_{\mathbf{x}}-(\dot{\zeta}t^{-\nu}-\nu t^{-1}\mathbf{x})\cdot\partial_{\mathbf{x}}|_{\mathbf{t}}\\
			\partial_{x}|_{t}&=t^{-\nu}\partial_{\mathbf{x}}|_{\mathbf{t}}\\
			\implies &
			\begin{multlined}
				t^{2\nu}\Box=\partial_{\mathbf{x}}^2-\partial_{\mathbf{t}}^2+2(t^{\nu-1}\nu \mathbf{x}+\dot{\zeta})\partial_{\mathbf{t}}\partial_{\mathbf{x}}-\left((\dot{\zeta}+\nu t^{\nu-1}\mathbf{x})\cdot\partial_{\mathbf{x}}\right)^2\\
				+\nu t^{\nu-1}\partial_{\mathbf{t}}-(\mathbf{x} t^{2\nu-2}\nu(\nu+1)+2\nu t^{\nu-1}\dot{\zeta}-t^\nu \ddot{\zeta})\partial_{\mathbf{x}}.
			\end{multlined}
		\end{nalign}
		Substituting the definition of $\mathbf{t}$ yields the result.
	\end{proof}

	\subsection{Energy spaces}\label{not:sec:energy-spaces}
	We use $L^2$ based analogues of the $\A$ spaces.
	We write $\mu=\dd x^3\dd t$ for the Minkowski volume form and $\mu_\b=\rho_{\C}^{1}\rho_-^{4}\rho_{\K}^{3\nu+1}\mu$ for a rescaled version.
	We define for $q_\C,q_-,q_\K\in\R$ and $s\in\N$ the norm
	\begin{equation}
		\norm{f}_{\Hb^{s;q_\C,q_-,q_\K}(\Mcomp)}:=\norm{\rho_{\C}^{-q_{\C}}\rho_{\K}^{-q_{\K}}\rho_{-}^{-q_{-}}\Vb^s f}_{L^2(\Mcomp;\mu_{\b})}=\norm{\rho_-^{-q_\C-1/2}\rho_-^{-q_--2}\rho_\K^{-q_\K-\frac{3\nu+1}{2}}\Vb^s f}_{L^2(\Mcomp;\mu)}.
	\end{equation}
	The weights in the above definition are motivated by the Sobolev embedding, which yields the following inclusions
	\begin{equation}\label{not:eq:Sobolev}
		 \Hb^{s+3;\vec{q}}(\Mcomp)\subset \A^{s;\vec{q}}(\Mcomp)\subset \Hb^{s;\vec{q}-}(\Mcomp).
	\end{equation}
	
	It will be also convenient to introduce the notion of regularity that is tied to the operator $\Box$ on $\Mcomp$.
	We study 3 different regions.
	Away from $\C$ and $\K$, we can express the metric as $\eta=-(1-\rho^2)\dd t^2+2t\dd\rho\dd t+t^2 \dd \rho^2$, where $\rho=x/t$.
	Up to a rescaling by $\rho_-^2$, this is a nondegenerate Lorentzian metric all the way to $t=0$ in the bundle $t^{-1}\dd t,\dd\rho$.
	Therefore, the corresponding regularity is measured in the dual basis $\{t\partial_t,\partial_\rho\}$.
	The regularity here is the same as $\Vb$.
	
	In a neighbourhood of a single soliton face $\K$, we use \cref{not:eq:box_nu} to observe that $\eta$, up to a rescaling by $\rho_\K^{2\nu}$, is a nondegenerate Lorentzian metric in the bundle $\{\frac{\dd y}{\jpns{y}},\frac{\dd t}{\jpns{y}\lambda},\omega\}$, where $\omega$ are 1-forms on $\sphere$.
	Note that using the time coordinate $\mathbf{t}$ is convenient to replace the second part with $\frac{\dd \mathbf{t}}{\jpns{y}}$.
	The corresponding regularity is measured with respect to $\{\jpns{y}\partial_y,\jpns{y}\partial_{\mathbf{t}}\}$.
	
	Finally, near $\C$, we can write $\eta$, up to a factor $(t-r)(\mathring{v})$ as a nondegenerate quadratic form in $\{\frac{\dd (t-r)}{t-r},\frac{\dd (\mathring{v})}{\mathring{v}}, \omega\rho_\C^{-1/2}\}$.
	The corresponding regularity is measured with respect to $\{(t-r)\partial_{t-r},(\mathring{v})\partial_{\mathring{v}},\rho_\C^{1/2}\slashed{\nabla}\}$.
	
	Correspondingly, we introduce
	\begin{equation}\label{not:eq:Vt}
		\mathring{\Ve}:=\{\rho_\C\rho_\K^{\nu}\rho_-\partial,\chi_{>1/2} \rho_\C^{1/2}\slashed{\nabla},\chi_{\mathcal{D}_{\mathrm{ext}}}\partial_{\mathring{v}}\},\qquad\Ve=\mathring{\Ve}\cup\{1\}.
	\end{equation}
	Note that near the different regions $\Ve$ is exactly given by the previously discussed regularity. 
	
	\begin{remark}[Top order regularity]
		The sharp top order regularity expressed by $\Ve$ is not necessary for the proofs, but we decide to keep it--as opposed to the $\Vb$ regularity in \cref{sec:lin}-- as it emphasises an important geometric point useful for generalisations.
	\end{remark}
	
	We will also use norms on constant $v_z$ hypersufaces $\Sigma_{z,\tau}:=\{v_z=\tau\}$.
	For some $\tau\in\R^+$ we  may consider the frozen coordinate function $\mathring{v}_z=t_z+h(y_z)/\lambda_z(\tau)$ and corresponding $y_z$ on $\Sigma_{z,\tau}$.
	For $\zeta_z=0$ we use $\dd y_z^3$ or $\dd x_z^3=\lambda^{-1}_z(\tau)\dd y_z^3$ as the measure on $\Sigma_{z,\tau}$.
	
	\subsection{Extra notation and computations}\label{not:sec:extra}
	We will use a lot of cutoff functions throughout the paper, which are going to be denote by $\chi_{<x_1}(x)$, meaning that $\chi|_{x<x_1}=1-\chi_{x>2x_1}=1$.
	
	\begin{notation}
		Throughout the paper, when we prove an estimate for an operator $S:\Hb^{s_1;\vec{a}_1}\to\Hb^{s_2;\vec{a}_2}$, we will always work with a dense subspace of $X\subset \Hb^{\vec{a}_1;s_1}$, say compactly supported functions and extend the estimate by density.
		Such density arguments will not be spelled out, and taken implicitly.
		Furthermore, when $S$ is a linear operator, as in most of \cref{sec:mapping,sec:pre,sec:an,sec:lin}, we say
		\emph{for $S\phi\in \Hb^{s_1;\vec{a}_1}$, it holds that $\phi\in\Hb^{s_1;\vec{a}_1}$}  to mean
		there exists a universal constant $C_{S,\vec{a}_i,s_i}$ such that, for all $\phi\in X$ it holds that $\norm{\phi}_{\Hb^{s_1;\vec{a}_1}}\leq C\norm{S\phi}_{s_2;\vec{a}_2}$ and therefore by density it also holds for all $\phi\in\Hb^{s_1;\vec{a}_1}$.
	\end{notation}

	For the convenience of the reader, we include a list for the sizes of the most important functions appearing in the paper, some of which are only introduced later
	\begin{nalign}\label{not:eq:sizes}
		y\in\A^{;0,1-\nu,0}(\Mcomp_0),\quad W_0\in\A^{;0,\nu/2-1,-\nu/2}(\Mcomp_0),\quad \Box W^{\lambda,0,0}+(W^{\lambda,0,0})^5\in\A^{;0,5\nu/2-5,-1-3\nu/2}(\Mcomp_0),\\
		\rho=x/t\in\A^{;0,0,\nu-1}(\Mcomp_0),\quad W_0^4\in\A^{;0,2\nu-4,-2\nu}(\Mcomp_0),\quad \chi_{<\delta}(\rho)W_zW_0^4\in\A^{;0,5\nu/2-5,-3\nu/2-1}(\Mcomp_0)
	\end{nalign}
	We also include some expressions derived from $W$
	\begin{nalign}\label{not:eq:integrals}
		W(r)=\frac{\chi_{r>1}}{r}+\A^{3}(\Rcompd{3}),\quad \Lambda W=-\frac{\chi_{r>1}}{2r}+\A^{3}(\Rcompd{3}),\\
		\int \Lambda W L_1 \Lambda W=\pi,\quad \int \Lambda W 5W^4=-2\pi,\quad \int \abs{\partial W}^2=\frac{3\sqrt{3}\pi^2}{4}.
	\end{nalign}
	\section{Mapping properties and model operators}\label{sec:mapping}
	In this section, we record the mapping properties of the operator $P$ and its linearisation around multi soliton solutions on conormal function spaces.
	Let us first define the class of approximate solutions we  study.
	\begin{definition}\label{map:def:approximate}
		Fix $\vec{\blambda}\in\R_{\neq 0}^{\abs{A}},\vec{\sigma}\in\{\pm1\}^{\abs{A}}$ soliton sizes and signs.
		We call $\mathcal{S}:=(\phi_0,\phi_1,\lambda_z,\zeta_z)$ an \emph{approximate solution} with rate $\nu\in\R$, and soliton velocities $A\subset\mathring{B}$ and error $\A^{a_\C,a_-,a_\K}(\Mcomp)$ if:
		\begin{itemize}
			\item $\lambda_z(v_z)=\blambda_zv_z^{-\nu}+\A^{\nu'}(\interval_{v_z})$ for $\nu'\geq-1$ and 
			$\zeta_z(v_z)=\A^{\nu}(\interval_{v_z},\R^3)$;
			\item $\phi_0:=\sum_z W_z$, where $W_z:=W^{\lambda_z,z,\zeta_z}:=\lambda_z^{1/2}\sigma_z W\big(\lambda_z(x_z-\zeta_z(t_z))\big)$;
			\item $\phi_1=\sum_zv_z^{\nu/2-1}\A^{2}(\K_z)\chi_{<\delta}(\rho_z)+\Aext^{a'_-,a'_\K}(\Mcomp)+\A^{\nu/2-,a'_-,a'_\K}(\Mcomp)$ with $a'_\K\geq 3\nu/2-2$, $a'_-\geq3\nu/2-2$, with the first term spherically symmetric with respect to the coordinate $y_z$;
			\item $P[\phi_0+\phi_1]\in\A^{a_\C,{a}_-,{a}_\K}(\Mcomp)$.
		\end{itemize}
		We also write $\phi_A=\phi_0+\phi_1$ for the approximate solution.
		In case the above holds with $\A(\Mcomp)$ replaced with $\Aext(\Mcomp)$, we say that $\mathcal{S}$ is a \emph{smooth approximate solution}.
		
		We call $\S$ a \emph{radiative approximate solution}, if in addition $\phi_1$ can have a contribution in $\Aext^{\nu/2-1}(\Mcomp_\emptyset)$ or $\Aext^{\nu/2-1}(\Mcomp_\emptyset)+\A^{\nu/2,\nu/2-1}(\Mcomp_\emptyset)$ in the smooth and non-smooth cases respectively.
	\end{definition}
	
	\begin{notation}[Rescaling]
		From now on, without loss of generality, we will assume that $0\in A$ and $\blambda_0=\sigma_0=1$.
		Moreover, whenever we work around the soliton at the origin, we drop the subscript $0$, ie.~ write $x=x_0, y=y_0, \zeta=\zeta_0$ and so on.
	\end{notation}
	
	\begin{lemma}[Linearised equation]\label{map:lem:lin}
		Let $\S$ be an approximate solution.
		For the linearised operator around $\phi_A$, $D_{\phi_A}P:=\Box+5\phi_A^4$, we have
		\begin{nalign}\label{map:eq:mapping}
				D_{\phi_A}P:&\A^{;a_\C,a_-,a_\K}(\Mcomp)\to\A^{;a_\C-1,a_--2,a_\K-2\nu}(\Mcomp),\\
				D_{\phi_A}P-A(D_{\phi_A}P,i_-):&\A^{;a_\C,a_-,a_\K}(\Mcomp)\to\A^{;a_\C,a_-+2\nu-4,a_\K-2\nu}(\Mcomp),\\
				D_{\phi_A}P-\sum_z\normal{\K_z} :&\A^{;a_\C,a_-,a_\K}(\Mcomp)\to\A^{;a_\C-1,a_--2,a_\K-\nu-1}(\Mcomp),
		\end{nalign}
		where $A(D_{\phi_A}P,i_-)=\Box$ and $\normal{\K_z}=t_a^{-2\nu}\Delta_{y_a}+5W_z^4$.
		The same holds in $\Aext(\Mcomp)$, without the $a_\C$ weights.
		More precisely
		\begin{equation}
			\chi_0 D_{\phi_A}P-t^{-2\nu}\Delta_{\mathbf{x}}+5W_z^4+\tilde{V}:\A^{;a_-,a_\K}(\Mcomp)\to\A^{;a_--2,a_\K-2}(\D_0),
		\end{equation}
		where $\tilde{V}\in \A^{;-2+2(\nu-1),-\nu-1}(\D_0)$ is spherically symmetric with respect to $\mathbf{t},\mathbf{x}$ coordinates.
	\end{lemma}
	\begin{proof}
		This is an explicit computation.
	\end{proof}
	
	Let us introduce the nonlinear operator
	\begin{equation}
		\mathcal{N}_{\phi_A}[\phi]=P[\phi_A+\phi]-P[\phi_A]-D_{\phi_A}P\phi,
	\end{equation}
	where $D_{\phi_A}P=\Box+5\phi_A^4$ is the linearisation of $P$.

	\begin{lemma}[Nonlinear terms]\label{map:lem:nonlin}
		Let $\S$ be an approximate solution, and let $\phi_2\in\A^{;a_C,a_-,a_\K}(\Mcomp)$ with $a_\K\geq \nu/2-1$ and $a_-\geq 3\nu/2-1$.
		Then
		\begin{subequations}
			\begin{align}
				&\mathcal{N}_{\phi_A}[\phi_2]\in\A^{;2a_{\C},2a_-+3(\nu/2-1),2a_{\K}-3\nu/2}(\Mcomp)\\
				&P[\phi_0+\phi_1+\phi_2]-P[\phi_0+\phi_1]-P_{\phi_0}\phi_2\in \A^{;a_{\C},a_--4(\nu/2-1),a_\K-\nu-1}(\Mcomp),
			\end{align}
		\end{subequations}
		where $a'_\K=a_\K-\min(\nu/2+2, \nu+1,2)$ and $a_-'=a_-+2\nu-4$.
		The same holds for smooth approximate solution with $\A(\Mcomp)$ replaced by $\Aext(\Mcomp)$.
	\end{lemma}
	\begin{proof}
		This is an explicit computation, and the leading error term arises from $\phi_1\phi_0^3\phi_2$ term in the nonlinearity.
	\end{proof}
	
	We have an $L^2$ analogue of the above
	\begin{lemma}\label{map:lem:nonlin_L2}
		Let $\phi_0,\phi_1,a_-,a_\K$ be as in \cref{map:lem:nonlin}, and $\phi_2,\phi_3\in \Hb^{a_\C,a_-,a_\K;s}(\Mcomp)$ for $s\geq10$, $a_\C>0$, $a_->\nu/2-1$, $a_\K>-\nu/2$.
		Assume moreover that $\norm{\phi_i}_{\Hb^{a_\C,a_-,a_\K;10}(\Mcomp)}\leq1$ for $i\in\{2,3\}$.
		Write $\phi_A=\phi_0+\phi_1$.
		Then
		\begin{subequations}
			\begin{align}
				\norm{\mathcal{N}_{\phi_A}[\phi_2]}_ {\Hb^{s;a'_C,a'_-,a'_\K}(\Mcomp)}\lesssim_{\phi_A}\norm{\phi_2}_{\Hb^{a_\C,a_-,a_\K;10}(\Mcomp)}\norm{\phi_2}_{\Hb^{a_\C,a_-,a_\K;s}(\Mcomp)}\label{map:eq:Hb_bound}\\
				\norm{\mathcal{N}_{\phi_A}[\phi_2]-\mathcal{N}_{\phi_A}[\phi_2+\phi_3]}_{\Hb^{10;a'_C,a'_-,a'_\K}(\Mcomp)}\lesssim_{\phi_A}\norm{\phi_2}_{\Hb^{a_\C,a_-,a_\K;10}(\Mcomp)}\norm{\phi_3}_{\Hb^{a_\C,a_-,a_\K;10}(\Mcomp)}\label{map:eq:Hb_contraction}
			\end{align}
		\end{subequations}
		where $a'_\K=2a_\K-3\nu/2$ and $a_-'=2a_-+3(\nu/2-1)$.
	\end{lemma}
	\begin{proof}
		Follows using Sobolev embedding, \cref{not:eq:Sobolev}, and the computation in \cref{map:lem:nonlin}.
	\end{proof}

	\section{Preliminaries for model operators}\label{sec:pre}
	In this section, we study the invertibility of the model operators from \cref{sec:mapping}, and then use these for the linear operator $D_{\phi_0}P$.
	We first of all recall some standard divergence and commutation computations that will be used later.
	First, introduce a smooth time function $\mathring{u}_*\in C^\infty$ where $\mathring{u}_*|_{x/t<1-\delta}=t, \mathring{u}_*|_{x/t<1-\delta/2}=t-\abs{x}$ and $\eta(\dd \mathring{u}_*,\dd \mathring{u}_*)\leq0$. See \cite[Lemma 2.2]{kadar_scattering_2024} for a construction.
	
		\begin{lemma}[Divergence computation]\label{pre:lemma:divergence}
		a) For $N>1,\epsilon>0$ and $C(N,\epsilon)$ sufficiently large fix $V_T=C\mathring{u}_*^{-N}t^{-\epsilon}T+\chi_{>1/2}(x/t)\mathring{u}^{-N-1}t^{1-\epsilon}\partial_{\mathring{v}}|_{t-r}$.
		The current
		\begin{equation}
			J^T_N=V_T\cdot\T,\quad \T=\dd\phi\otimes\dd\phi-\frac{\eta}{2}\eta^{-1}(\dd\phi,\dd\phi)
		\end{equation}
		satisfies
		\begin{equation}\label{pre:eq:Tdiv_computation}
			\div J_N^T+\abs{V_T\phi \Box\phi}\gtrsim_{N,\epsilon}t^{-1-\epsilon}\mathring{u}_*^{-N}\big(\abs{\partial\phi}^2+t^{-1}\mathring{u}_*^{-1}\abs{\slashed{\nabla}\phi}^2+t^2\mathring{u}^{-2}\abs{\partial_{v}\phi}^2\big).
		\end{equation}

		b) Let $X=\partial_{\abs{x}}|_{t}$ and $\mathbf{x}=t^{-\nu}x$.
		Let's define the Morawetz current $J^X_N=t_*^{-N-\epsilon}\chi_{<\delta}(\abs{x}/t)(f(\mathbf{x})X\cdot\T-\phi^2\dd\mathfrak{f}+\mathfrak{f} \dd \phi^2)$ where $f=1-\jpns{\mathbf{x} }^{-1}$ and $\mathfrak{f}=\jpns{\mathbf{x}}^{-1}t^{-\nu}f$.
		Then, for $N>1$ there exists $c>0$ sufficiently small (independent of $N$) such that $J_N=J^X_N+J^T_N$ satisfies
		\begin{equation}\label{pre:eq:MasterDiv}
			\div J_N +C_2t^{-\epsilon}\mathring{u}_*^{-N}\frac{\abs{\Box\phi}^2}{w}+\abs{\phi}^2c\rho_-^{-3-\epsilon-N}\rho_\K^{-\epsilon-1-N-2\nu}\gtrsim t^{-\epsilon}u^{-N}_*w\rho_\C^{-2}\rho_-^{-2}\rho_\K^{-2\nu}\abs{\Ve\phi}^2,
		\end{equation}
		where the weight function is  $w=\rho_-^{-1}(\rho_\K^{-1}+c\rho_\K^{-\nu})$ and the regularity gain is given in \cref{not:eq:Vt}.
	\end{lemma}	
	
	\begin{remark}[Morawetz current]
		Let us note that applying Cauchy-Schwartz to \cref{pre:eq:Tdiv_computation} we already obtain
		\begin{equation}
			\div J^T_N+\mathring{u}_*^{-N}t^{1-\epsilon}\abs{\Box\phi}^2\gtrsim \mathring{u}_*^{-N}t^{1-\epsilon}\rho_{\C}^{-2}\rho_-^{-4}\rho_\K^{-2(1+\nu)}\abs{\mathring{\Ve}\phi}^2,
		\end{equation}
		which already yields a sharp--with respect to weights-- inverse of \cref{map:eq:mapping} near $i_-,\C$.
		The role of the Morawetz current is to improve on the weights near the compact region $\K$.
		As will be clear by comparing the proof of \cref{model:lem:i} and \cref{model:lem:i_smooth} the Morawetz estimate is not strictly necessary for our work, but we included as it showcases the precise invertability of \cref{map:eq:mapping} in weighted Sobolev spaces.
	\end{remark}

	\begin{proof}
		a)
		We start by writing 
		\begin{equation}
			\div J^T_N-V\phi\Box\phi=C\T\big[T,(\dd \mathring{u}_*^{-N}t^{-\epsilon})^\#\big]+\T\big[\partial_{\mathring{v}}|_{\mathring{u}},(\dd u^{-N-1}t^{1-\epsilon} )^\# \big]+u^{-N-1}t^{1-\epsilon}\pi^{\partial_{\mathring{v}}|_{\mathring{u}}}\cdot\T.
		\end{equation}
		This first term is clearly positive, since $\T$ satisfies the weak energy condition, and we may compute 
		\begin{equation}
			\T\big[T,(\dd \mathring{u}_*^{-N}t^{-\epsilon})^\#\big]\gtrsim t^{-1-\epsilon}\mathring{u}_*^{-N}\big(\epsilon\abs{\partial\phi}^2+Nt\mathring{u}_*^{-1}(\abs{\partial_v\phi}^2+\abs{\slashed{\nabla}\phi}^2)\big).
		\end{equation}
		This already yields optimal control except for the $\partial_v$ derivative.
		It is easily computed that for $C$ large enough on $\{x>t/2\}$
		\begin{equation}
			\T\big[\partial_{\mathring{v}}|_{\mathring{u}},(\dd u^{-N-1}t^{1-\epsilon} )^\# \big]+Ct^{-1-\epsilon}\mathring{u}_*^{-N}\big(\epsilon\abs{\partial\phi}^2+t\mathring{u}_*^{-1}(\abs{\partial_v\phi}^2+\abs{\slashed{\nabla}\phi}^2)\big)\gtrsim t^{1-\epsilon}\mathring{u}^{-N-2}\abs{\partial_v\phi}^2,
		\end{equation}
		and
		\begin{equation}
			\pi^{\partial_{\mathring{v}}|_{\mathring{u}}}\cdot\T\lesssim t^{-1-\epsilon}\mathring{u}_*^{-N}\big(\epsilon\abs{\partial\phi}^2+t\mathring{u}_*^{-1}(\abs{\partial_v\phi}^2+\abs{\slashed{\nabla}\phi}^2)\big).
		\end{equation}

		b)
		Next, we show that $J_N$ is timelike. 
		It is clear that $\tilde{T}+cX$ is timelike for sufficiently small $c$, so we already conclude the same for the $\T$ part of $J_N$.
		In particular, we have a term of the form $\phi^2\jpns{\mathbf{x}}^{-2}t_*^{-2\nu-N-1}\dd t_*$ in $\tilde{\T}$.
		Therefore, we can also absorb the $\phi^2\dd\mathfrak{f}$ term, to keep the current timelike.
		Similarly for the $\mathfrak{f}\dd\phi^2$ term.
		
		Let us note that $\div J^X_N=t^{-N-\epsilon}\div J^X_{-\epsilon}+(N+\epsilon)J^X_{N}\cdot \dd \mathring{u}_*$.
		Note that in the second term, the contribution form $f(\mathbf{x})X\cdot\T$ in $J^X_N$ can be absorbed into $\div J^T_N$, as it only contains differentiated terms.
		The rest of the terms contain undifferentiated terms, and these we must keep on the left hand side of \cref{pre:eq:MasterDiv}.
		
		Next, we compute in the region $\abs{x}/t<\delta$ the current for $N=0$
		\begin{equation}\label{pre:eq:proof_mor1}
			\div J^{X}_{\nu}=\pi^{fX}\cdot\T_0-\phi^2\Box\mathfrak{f}+\phi\mathfrak{f}\Box\phi+\mathfrak{f}\partial\phi\cdot\partial\phi,
		\end{equation}
		where
		\begin{multline}
			\pi^{fX}=\dd f^{\#}\otimes X+f\mathcal{L}_{\partial_{\abs{x}}}\eta^{-1}=t^{-\nu}f'\partial_{\abs{x}}^2+f\abs{x}\slashed{g}+\nu\A^{\nu-2,-1}(\Mcomp;\{\partial_t,\partial_x\}^2)\\
			\implies \pi^{f\partial_{\abs{y}}}\cdot\T_0\cdot\T_0+f\partial\phi\cdot\partial\phi=t^{-\nu}f'\big((\partial_{\abs{x}}\phi)^2+(\partial_t\phi)^2\big)-t^{-\nu}(f/z^2)'\abs{\slashed{\nabla}\phi}^2+\abs{\partial\phi}^2 \nu\A^{\nu-2,-1}(\Mcomp).
		\end{multline}
		and 
		\begin{equation}
			\Box \mathfrak{f}=\frac{t^{-3\nu}f''}{\abs{z}}+\A^{\nu-3,2}(\Mcomp_0).
		\end{equation}
		Hence, provided that $f',-(f/z^2)',-f''$ all have the same sign, \cref{pre:eq:proof_mor1} yields coercive control near $\K_0$ that is
		\begin{equation}\label{pre:eq:proof_div1}
			\pi^{f\partial_{\abs{y}}}\cdot\T_0-\phi^2\Box\mathfrak{f}+\mathfrak{f}\partial\phi\cdot\partial\phi +(\abs{\partial\phi}^2+\abs{w_A\phi}^2)\rho_-^{\nu-2}\rho_\K^{-1}\gtrsim(\abs{\partial\phi}^2+\abs{w_A\phi}^2)\rho_-^{\nu-2}\rho_\K^{-\nu}.
		\end{equation}
		
		In the transition region $\supp\chi'$, we need to replace the $\rho_-^{\nu-2}$ with $\rho_-^{-1}$ on the left hand side of \cref{pre:eq:proof_div1}, however this is exactly controlled by the $J^T$ current.
		Increasing the power of $t_*^{-N}$, only adds terms of the form 
		\begin{equation}
			\div J^X_N=Nt_*^{-N-1}J^X(\dd t_*)+t_*^{-N}\div J^X_0.
		\end{equation}
		The first term is controlled by the timelike property of $J_N$.
	\end{proof}
	
	In order to recover the 0th order term, we apply Hardy inequality:
	\begin{lemma}[Hardy]\label{pre:lemma:Hardy}
		Fix $a_\C,a_-,a_\K\in\R$ satisfying $a_->-1$ and $a_--a_\K>-3(\nu-1)$.
		For $\phi$ compactly supported in $\mathring{\Mcomp}$, it holds that
		\begin{equation}
			\norm{\phi}_{\Hb^{0;a_\C,a_-,a_\K}(\Mcomp)}\lesssim\norm{\{\Ve\setminus\{1\}\}\phi}_{\Hb^{0;a_\C,a_-,a_\K}(\Mcomp)}.
		\end{equation}
	\end{lemma}
	\begin{proof}
		Away from the soliton faces, in $D=\cap_{z\in A}\{\abs{x_z}/t>\delta\}$, we already control $\rho_-\partial_{\rho_-}\phi\in\Hb^{a_-}(D)$, and thus the standard Hardy inequality 
		\begin{equation}
			\int_D \phi^2 t^{a_-}\dd t\dd x\lesssim_{a_-} \int_D  (\rho_-\partial_{\rho_-}\phi)^2 t^{a_-}\dd t\dd x,\qquad \text{provided }a_-\neq-1.
		\end{equation}
		also yields control of $\phi\in\Hb^{a_-}(D)$.
		
		Near the soliton $\K_0$, we may use the version of Hardy on $\R_y^3$ 
		\begin{equation}
			\int_{\abs{y}<r_2} \phi^2 \jpns{y}^{a}\dd y\lesssim_{a} \int_{\abs{y}<2r_2} (\jpns{y}\partial_{\abs{y}}\phi)^2 \jpns{y}^{a}\dd y+\int_{\abs{y}\in(r_2/2,2r_2)}\phi^2\jpns{y}^{a}\dd y \qquad\text{provided } a=\frac{a_--a_\K}{\nu-1}\neq-3
		\end{equation}
		to propagate the control to $\K_0$. Importantly, the implicit constant is independent of $r_2$, and can be used for $r_2=t\delta$ parametrically in $t$ and then integrating in $t$.
		
	\end{proof}
	
	\begin{lemma}[Commutators]\label{pre:lemma:commutation}
		For rotation, scaling and conformal killing vectorfields $\Omega_{ij}=\{x_j\partial_{x_i}-x_i\partial_{x_j}\},S=t\partial_t+x\partial_x,K=(\mathring{u})^2\partial_{\mathring{u}}+(\mathring{v})^2\partial_{t+r}$,
		the following commutations hold
		\begin{equation}
			[\Box,\partial_x]=[\Box,\partial_t]=[t^2\Box,S]=[\Box,\Omega_{ij}]=[u^2v^2r\Box r^{-1},K]=0.
		\end{equation}
	\end{lemma}
	\begin{proof}
		These are direct computations
	\end{proof}

	\subsection{$L^2$ inversions}
	In this section, we use the previous energy estimates to derive bounds for the wave equation near the regions $i_-$ and $\K$ respectively.
	The $i_-$ region is a straightforward application of the divergence theorem.
	Near $\K$, we will use the same currents and the spectral information about $(\Delta+5W^4)$ as given in \cite[Proposition 5.5]{duyckaerts_dynamics_2008}.

	\paragraph{Self similar region:}
	We consider the inverse of $\Box$, the model operator at $i_-$.
	In this case, we have two inverses depending on what is the radiation that we set on $\C$.
	For prescribed $\phi|_{\C}$, the solutions in general only have limited regularity as shown in \cref{model:lem:i}.
	However, it is also possible to construct, non-unique inverses that are smooth across $\C$, as shown in \cref{model:lem:i_smooth}.
	\begin{lemma}[Inverse of $\normal{i_-}$]\label{model:lem:i}
		Fix $a>-1/2$ and $0<a_{\C}<a+1$.
		
		a)
		Let $f\in\Hb^{s;a_\C-1,a-2,a-2\nu}(\Mcomp)$ with $a_\C<a$.
		Then, there exists a unique $\Ve\phi\in\Hb^{s;a_\C,a-,a}(\Mcomp)$ such that $\Box\phi=f$ and $\phi|_\C=0$.
		
		b) Let $f\in\A^{;a_\C-1,a-2,a_\K}(\Mcomp)$ with $a_\C<a$ and $a_\K\in[a-4\nu,a-2\nu]$.
		Then, there exists a unique $\phi\in\A^{;(a_\C,a,a)-}(\Mcomp)$ such that $\Box\phi=f$ and $\phi|_\C=0$.
		
	\end{lemma}
	\begin{proof}
		a)
		We start with the limited regularity statement.
		
		Let us write $f=f_-+\sum_z f_z$ with $\supp f_z\subset\{x_z/t_z<\delta\}$ and $f=f_z$ on $\{x_z/t_z<\delta/2\}$.
		Thus, it suffices to prove the result on $\Mcomp_0$.
		
		Indeed, notice that using \cref{pre:lemma:divergence}, in particular using the current $J_N$, with $N=2a>0$ and $\epsilon>0$ for any $c>0$ we get
		\begin{nalign}
			&\begin{multlined}
				\int_{\Mcomp}\mu t^{-\epsilon}u^{-N}_*w_N\rho_\C^{-2}\rho_-^{-2}\rho_\K^{-2\nu}\abs{\Ve\phi}^2=\int_{\Mcomp}\mu_b t^{-\epsilon}u^{-N}_*w_N\rho_\C^{1-2}\rho_-^{4-2}\rho_\K^{3\nu+1-2\nu}\abs{\Ve\phi}^2\\
				\sim \norm{\Ve\phi}^2_{\Hb^{0;a+1/2,a-1/2+\epsilon,a-\nu/2+\epsilon}(\Mcomp)}+c\norm{\Ve\phi}^2_{\Hb^{0;a+1/2,a-1/2+\epsilon,a-1/2+\epsilon}(\Mcomp)},
			\end{multlined}\\
			&\int_{\Mcomp}\mu t^{-\epsilon}\mathring{u}_*^{-N}\frac{\abs{f}^2}{w}\lesssim \norm{f}_{\Hb^{0;a-1/2,a-5/2+\epsilon,a-2\nu-1/2+\epsilon}(\Mcomp)}
		\end{nalign}

		implying
		\begin{equation}
			\norm{\Ve\phi}_{\Hb^{0;a+1/2,a-1/2+\epsilon-,a-1/2+\epsilon}(\Mcomp)}\lesssim_{a,\epsilon}c^{-1}\norm{f}_{\Hb^{0;a-1/2,a-5/2+\epsilon,a-2\nu+\epsilon}(\Mcomp)}+c\norm{\phi}_{\Hb^{0;a,a+3/2+\epsilon,a-\nu/2+\epsilon}(\Mcomp)}.
		\end{equation} 
		We use \cref{pre:lemma:Hardy} to recover the 0th order terms and absorb it to the left hand side.
		This proves the lemma for $s=0$.
		
		We proceed by induction.
		Assume the lemma holds for some $s$.
		Let us already note that $S,\Omega, t^{-1}K$ span the $\Diffb(\Mcomp)$ in $\abs{x}/t<1/2$, and together with the Lorentz boosts, they span $\Diffb(\Mcomp)$ in the remaining region.
		Using the commutations from \cref{pre:lemma:commutation}, we obtain by induction
		\begin{equation}
			\{S,\Omega\}\phi\in\Hb^{0;a+1/2,a-1/2+\epsilon-,a-1/2+\epsilon}(\Mcomp).
		\end{equation}
		Using that $K\in \rho_-\rho_\K\Diffb$, we also get
		\begin{equation}
			t^{-1}K\phi\in\Hb^{0;a+1/2,a-1/2+\epsilon-,a-1/2+\epsilon}(\Mcomp).
		\end{equation}
		These already enough to conclude in the region $\abs{x}/t<1/2$.
		In the exterior, we can commute with the Lorentz boosts with a cutoff $\chi_{>1/2}(x/t)(t\partial_{x_i}+x_i\partial_t)$, to conclude the result.
		
		b) We use \cref{not:eq:box_nu} and invert $\Delta_{\mathbf{x}}$ parametrically in $t$ to improve the inhomogeneity to $f\in \A^{;a_\C-1,a-2,a-2\nu}(\Mcomp)$.
		Then we use the first part and Sobolev embedding \cref{not:eq:Sobolev} to deduce the result.

	\end{proof}
	
	\begin{lemma}[Smooth inverse of $\normal{i_-}$]\label{model:lem:i_smooth}
		Let $f\in\Aext^{;a_--2,a_\K-2\nu}(\Mcomp)$.
		Then, there exists (non-unique) $\phi\in\Aext^{;a_--,\min(a_-,a_\K)-}(\Mcomp)$ such that $\Box\phi=f$.
	\end{lemma}
	\begin{proof}
		This follows from \cite{angelopoulos_matching_2026}, after a conformal change of coordinates, but we give a proof here.
		
		As in \cref{model:lem:i_smooth}, we can split $f=\sum_z f_z$ with $f_z\in \A(\Mcomp_z)$.
		We show the result for $f_0$ the rest following similarly.
		
		First, we show that it suffices to consider $f\in\Aext^{a_--2}(\Mcomp_\emptyset)$.
		Indeed, we can iteratively solve $\Delta\phi_n=f_n\chi_{x/t<1/2}$ parametrically in $t$, where $f_0=f$ and $f_{n+1}=f_n-\Box(\phi_n\chi_{x/t<1/2})$.
		It follows that $f_n\in\Aext^{a_--2,a_\K-2\nu+2n(\nu-1)}$.
		By Borel summation, there exists $\phi_{\K}\in \A^{a_-,a_\K}(\Mcomp)$ such that $\Box\phi_\K-f\in\A^{a_--2}(\Mcomp_\emptyset)$.
		Let us also extend $f$ to $\D'=\{\abs{x}/t<2,t\in[0,1)\}$ smoothly with $\supp f\in\{\abs{x}/t<3/2\}$, and denote it by the same function.
		
		For $a_-<-1/2$, we define
		\begin{equation}
			\Box\phi=f,\qquad \phi|_{t=1}=T\phi|_{t=1}=0.
		\end{equation}
		We may apply the energy estimate with current $t^{-2a_--1}T\cdot\T[\phi]$ to obtain as in \cref{pre:lemma:divergence,model:lem:i} with divergence theorem applied in the reverse direction
		\begin{equation}
			\norm{\{t\partial_t,t\partial_x\}\phi}_{\Hb^{0;a_-}(\D')}\lesssim\norm{f}_{\Hb^{0;a_--2}(\D')}.
		\end{equation}
		Commuting with $\partial_x,\partial_t$ yields higher regularity.
		
		For $a_-\in(-1/2,1/2)$, we define $\Phi$ as the solution of
		\begin{equation}
			\Box \Phi=Tf\in\Aext^{a_--3}(\D'),\qquad \Phi|_{t=1}=T\Phi|_{t=1}=0.
		\end{equation}
		Using the induction step, we get that $\Phi\in\Aext^{a_--1}(\D')$.
		We next define $\Phi^{(1)}=\chi(r/t)\Phi$ with $\chi$ localising to $r<4t/3$ with $\supp\chi'\subset{\{r<3t/2\}}$, for which we still have $\Box\Phi^{(1)}=Tf$ in $\Mcomp$.
		We define $\phi=\int_{\abs{x}/2}^t\dd t' \Phi^{(1)}$.
		Using the support assumption, $\supp f\subset \{\abs{x}/t<3/2\}$ implies that $\Box\phi=f$.
		By integration we have $\phi\in\A^{;a_-}(\D')$.
		The case for general $a_-$ follows by induction.
	\end{proof}

	\paragraph{Compact regions:}
	We recall the invertibility of the stationary operator near the individual solitons.
	In rescaled coordinates, the model operator takes the form $\Delta+5W^4$.
	From the scaling and translation symmetry of \cref{in:eq:critical}, we know that $=\Lambda_x W(x)=(x\cdot\partial_x+1/2)W(x)$ and $\partial_i W$ are in the kernel of $\Delta+5W^4$ and are thus an obstruction for invertibility.
	It is well known that this is the only obstruction:
	\begin{lemma}[Inverse of $\normal{\K}$]\label{pre:lem:K}
		Let $f\in\A^{l+2}(\Rcompd{3})$ with $l>1$ satisfying $(f,\Lambda W)=(f,\partial_i W)=0$.
		Then, there exists unique $\phi\in\A^{\min(l,3)-}(\Rcompd{3})$ solving $(\Delta+5W^4)\phi=f$ with $(\phi,\partial W)=0$.
	\end{lemma}
	\begin{proof}
		This is standard. From \cite[Proposition 2.5]{hintz_underdetermined-elliptic_2025}, it follows that
		\begin{equation}
			(\Delta+5W^4):\{u\in\Hb^{s;l}(\Rcompd{3}):(u,\partial W)=0\}\to\{u\in\Hb^{s;l+2}(\Rcompd{3}):(u,\partial W)=(u,\Lambda W)=0\}
		\end{equation}
		is Fredholm.
		By self adjointness, and the fact that $\ker(\Delta+5W^4)=\text{span}\{\Lambda W,\partial W\}$, it follows that the operator is also invertible.
	\end{proof}
	
	\subsection{Spacetime inversion}
	
	In order to study the spacetime invertebility analogous to \cref{pre:lem:K}, it is useful to introduce spacetime kernel projection operators.
	\begin{definition}\label{inv:def:proj}
		Given a function $f\in\A^{;0,a_{\K}+\alpha(\nu-1),a_\K}(\Mcomp)$ with $\alpha>2$, we define the projection operators
		\begin{nalign}
			\P_{z,\Lambda}f(v_z)=\int_{\abs{y_z}\leq \delta t_z^{-\nu+1}}\dd y_z^3 f(v_z,y_z)\Lambda_{y_z}W_\Lambda(\blambda_zy_z)\in\A^{a_\K}(\interval),\\
			\P_{z,\partial}f(v_z)=\int_{\abs{y_z}\leq \delta t_z^{-\nu+1}}\dd y_z^3 f(v_z,y_z) \partial_{y_z}W(\blambda_zy_z)\in\A^{a_\K}(\interval;\R^3).
		\end{nalign}
	\end{definition}
	
	\begin{corollary}[Inverse at $\K$]\label{model:cor:K}
		Let $\S$ be an approximate solution with error $\Aext^{;a_-,a_\K}(\Mcomp)+\A^{;a_{\C},a_-,a_\K}(\Mcomp)$ where $a_\K\geq-\nu/2-2$ and $a_-\in a_\K+(2(\nu-1),5(\nu-1))$.
		Assume furthermore that for $\bullet\in\{\Lambda,\partial\}$
		\begin{equation}
			\P_{z,\bullet}P[\phi_0+\phi_1]\in\A^{a_\K+(\nu-1)}(\interval).
		\end{equation}
		Then, there exists $\phi_\K\in\Aext^{a_-+2,a_\K+2\nu}(\Mcomp)$ such that
		\begin{equation}
			P[\phi_0+\phi_1+\phi_\K]\in\A^{a_\C,a_-,a_\K+(\nu-1)}(\Mcomp)+\Aext^{;a_{\C},a_-,a_\K}(\Mcomp).
		\end{equation}
	\end{corollary}
	\begin{proof}
		Solving around $\K_z$ for each $z$ parametrically in $t_z$ using the inverse provided in \cref{pre:lem:K} and localising around $\K_z$, we obtain $\phi_\K\in\Aext^{a_-+2,a_\K+2\nu}(\Mcomp)$  with $\supp \phi_{\K}\in\D_z$ satisfying
		\begin{equation}
			D_{\phi_0}P\phi_\K\in \A^{;\infty,a_-,a_\K+\nu-1}(\Mcomp),\quad D_{\phi_0}P \phi_\K-\chi_zP[\phi_0+\phi_1]\in\Aext^{;a_{\C},a_-,a_\K}(\Mcomp).
		\end{equation}
		Then, we use the error term computation from \cref{map:lem:nonlin} to conclude the result.
	\end{proof}
	
	\begin{corollary}[Inverse at $i_-$]\label{mode:cor:i_-}
		Let $\S$ be an approximate solution with error $\Aext^{;a_-,a_\K}(\Mcomp)+\A^{;a_{\C},a_-,a_{\K}}(\Mcomp)$ where $a_{\C}>0$, $a_-\geq5(\nu/2-1)$ and $a_\K<a_--2(\nu-1)$.
		There non-unique $\phi_{-}\in\Aext^{a_-+2,a_-+2}(\Mcomp)+\A^{;a_{\C}+1,a_-+2,a_-+2}(\Mcomp)$  solving 
		\begin{nalign}
			\Box\phi_{1,-}=P[\phi_0+\phi_1].
		\end{nalign}
		Furthermore
		\begin{nalign}
			P[\phi_0+\phi_1+\phi_{-}]\in\Aext^{;a_-+(\nu-1),a_\K}(\Mcomp)+\A^{;a_\C+1+2\nu,a_-+(\nu-1),a_\K}(\Mcomp).
		\end{nalign}
	\end{corollary}
	\begin{proof}
		The existence of $\phi_{-}$ follows from \cref{model:lem:i,model:lem:i_smooth} respectively, and the nonlinear errors are bounded using \cref{map:lem:nonlin}.
	\end{proof}
	\begin{remark}\label{mode:rem:irreg}
		A unique $\phi_-$ may also be defined by solving
		\begin{equation}
			\Box\phi_{0,-}=P[\phi_0+\phi_1],\qquad \phi_-|_{\C}=\phi_\C,
		\end{equation}
		for a prescribed $\phi_\C\in\A^{;a_-+2}(\C)$.
		However, this would in general only have regularity $\phi_-\in\Aext^{;a_-+2,a_-+2}(\Mcomp)+\A^{;\min(a_-+3,a_{\C}+1),a_-+2,a_-+2}(\Mcomp)$.
	\end{remark}
	
	\section{Approximate solution}\label{sec:an}
	In this section, we construct approximate solutions to \cref{in:eq:critical} using the methods developed in \cref{sec:pre} together with modulation theory discussed in \cref{an:sec:modulation}.
	We state the result using the terminology of \cref{map:def:approximate}.
	The main results are the following
	\begin{prop}[Smooth radiative solitons]\label{an:prop:smooth_quantized}
		For $\nu/2-1\in\N_{\geq1}$, and arbitrary soliton velocities, scales and signs $\abs{A}\leq2,\vec{\blambda},\vec{\sigma}$ there exists a radiative smooth approximate solution with error $\Aext^{\infty,\infty}(\Mcomp)$.
		For $\abs{A}=1$, $\nu=2$ is also allowed.
	\end{prop}
	
	\begin{prop}[Smooth multi solitons]\label{an:prop:smooth_multi}
		For $\nu>8$ there exists soliton velocities $A$ together with leading scales $\vec{\blambda}$ and signs $\vec{\sigma}$ such that there exists a non-radiative smooth approximate solution with error $\Aext^{\infty,\infty}(\Mcomp)$.
	\end{prop}

	\begin{prop}[Non-smooth solitons]\label{an:prop:multi}
		For $\nu>1$ soliton velocities, scales and signs $\abs{A},\vec{\blambda},\vec{\sigma}$ there exists a radiative approximate solution with error $\Aext^{,\infty,\infty}(\Mcomp)+\A^{\nu,\infty,\infty}(\Mcomp)$.
	\end{prop}
	\begin{prop}[Exponential singularity formation]\label{an:prop:exponential}
		Let $\nu\in\R^+$, and set $\blambda(t)=e^{\nu/t}$, as well as $v_{\mathrm{e}}=t+h(x\lambda(v))/\lambda(v)$.
		There exist spherically symmetric ansatz $\phi_0$ satisfying\footnote{As the regularity structure is more complicated, and does not fit into the polynomial rate compactification as set up in \cref{sec:setup}, we only use $L^\infty$ bounds.}
		\begin{nalign}
			\phi_0&=\lambda(v)^{1/2}W(\lambda(v) x)+\mathfrak{Err}_1,\\
			P[\phi]&=\mathfrak{Err}_2,
		\end{nalign}
		where $\abs{\mathfrak{Err}_i}\leq_q e^{\frac{\nu-\epsilon}{2t}}$ for all $\epsilon>0$.
	\end{prop}
	
	Let us already note that the approximate solution corresponding to exponential singularity formation is \emph{not} corrected to a true solution in \cref{sec:nonlin}.
	We do not address whether this is possible in this work, and only do the construction to highlight an alternative approach to potentially obtain smooth singularities.
	
	The rest of this section is concerned with the proof of the above statements.
	We first focus on the most difficult one, \cref{an:prop:smooth_multi}.
	We lay the ground works for the proof in \cref{an:sec:first,an:sec:second,an:sec:modulation}, and prove the result by induction in \cref{an:sec:iteration}.
	The proof of \cref{an:prop:smooth_quantized,an:prop:multi} follows similarly, we only need to replace the first iterate in the approximate solution as discussed in \cref{an:sec:rad}.
	Finally, in \cref{an:sec:exp} we prove \cref{an:prop:exponential}, where instead of an infinite induction, a finite number of iterations yield the proof.

	\subsection{First iteration and admissibility}\label{an:sec:first}
	We start with a preliminary computation
	\begin{lemma}[Leading kernel element]\label{an:lemma:first}
		Let $\S=(\phi_0,\phi_1,\lambda,\zeta)$ be an approximate solution with $\phi_1=0$ and $\zeta=0$. 
		Then, it holds that $P[\phi_0]\in\Aext^{5(\frac{\nu}{2}-1),-\frac{3}{2}\nu-1}(\Mcomp)$.
		Moreover, we can compute that 
		\begin{equation}\label{an:eq:leading_Kernel}
			v_z^{3\nu/2+1}\P_{z,\Lambda}P[\phi_0]=-\pi\nu\sigma_z\blambda_z^{3/2}-\blambda_z^2\sum_{\bar{z}\in A\setminus \{z\}}2\pi\sigma_{\bar{z}}\blambda^{-1/2}_{\bar{z}}\frac{(1+\abs{z_{\bar{z}}})^{\nu/2}}{\abs{z_{\bar{z}}}}\gamma_{z,\bar{z}}^{\nu/2-1}+\A^{\nu-1}(\interval).
		\end{equation}
		The leading behaviour \cref{an:eq:leading_Kernel} holds provided $\zeta\in\A^{\nu}(\interval)$ and $\phi_1\in\Aext^{3\nu/2-2,3\nu/2-2}(\Mcomp)$, but in this case $P[\phi_0]\in\Aext^{\frac{3\nu}{2}-4,-\frac{3}{2}\nu-1}(\Mcomp)$.
	\end{lemma}
	\begin{proof}
		We first show the decay rate near $i_-$.
		Observe, $W_0^5\in\Aext^{5(\nu/2-1),-5\nu/2}(\Mcomp)$.
		We also have that in a neighbhourhood $U$ of $i_-$ away from the face $\K_0$, we can write
		\begin{equation}\label{an:eq:W_exp}
			\sigma_0W_0=\lambda(v)^{1/2}W(\lambda(v)(x+\zeta(v)))=\lambda(v)^{-1/2}\abs{x}^{-1}\big(1+\abs{x}^{-2}x\cdot\zeta\big)+\A^{5\nu/2-3}(U).
		\end{equation}
		Acting with $\Box$, we observe that the leading order term vanishes.
		Provided that $\zeta=0$, the subleading $\A^{3\nu/2-2}(U)$ term is also 0.
		
		We perform the computation around the soliton $\K_0$, the others following similarly.
		As always, we drop the $0$ subscript on $\lambda,\sigma$ and so on.
		We compute using \cref{not:lemma:Box}.
		On $\Sigma_\tau=\{v=\tau\}$, using $\bar{y}=\lambda(\tau)(x+\zeta(\tau)),v$ coordinates, we can write
		\begin{equation}
			W_0=\lambda(v)^{1/2}\sigma W\Big(\frac{\lambda(v)}{\lambda(\tau)}\bar{y}\Big).
		\end{equation}
		Therefore, we get from \cref{not:lemma:Box} that
		\begin{equation}\label{an:eq:PW_0}
				\sigma(\Box W_0+W_0^5)=\lambda^{3/2}b\Big(\frac{2h'(y)}{\abs{y}}+2h'(y)\partial_{\abs{y}}+h''(y)\Big)\Lambda_y W(y)\mod \Aext^{5/2\nu-5,-\nu/2-2}(\Mcomp_0),
		\end{equation}
		where we kept the $v$ dependence implicit and used that $b=\dot{\lambda}/\lambda$ and $\Lambda_{{y}}=1/2+{y}\cdot\partial_{{y}}$.
		We can compute the corresponding projection via \cref{not:eq:integrals} to get
		\begin{equation}
			\int_{\abs{\bar{y}}\leq \delta t_z^{-\nu+1}} \Lambda_{\bar{y}} W(\bar{y})\big(2h'(\bar{y})\abs{\bar{y}}^{-1}\partial_{\abs{y}}\abs{\bar{y}}+h''(\bar{y})\big)\Lambda_{\bar{y}}W(\bar{y})=4\pi(y \Lambda_{\bar{y}} W(\bar{y}))^2|_{\abs{\bar{y}}= \delta t_z^{-\nu+1}}=\pi+\A^{4(\nu-1)}(\interval_t).
		\end{equation}
		Inserting the $\blambda$ factors, we conclude that
		\begin{equation}
			\P_{z,\Lambda}\sigma(\Box W_z+5W_z^5)=-\pi\blambda_z^{3/2}v_z^{-3\nu/2-1}+\A^{-\nu/2-2}(\interval).
		\end{equation}
		
		Next, we compute the projection corresponding to the cross terms for $z\neq0$ in $\D_0$
		\begin{multline}\label{an:eq:Wz_expansion}
			\chi_{<\delta}(x/t)W_0^4W_z\equiv\chi_{<\delta}W_0^4\frac{\blambda_z^{-1/2}\abs{v_z}^{\nu/2}}{\gamma_z\abs{x-zt}}\equiv\gamma_z^{\nu/2}\frac{W_0^4\blambda_z^{-1/2}}{\abs{z}}t^{\nu/2-1}(1+\abs{z})^{\nu/2}\Big(1+\frac{z\cdot x}{\abs{z}^2t}-\frac{\nu z\cdot x(1+\abs{z}^{-2})}{2(1+\abs{z})t}\Big)\\
			\equiv\blambda_0^{2}W^4(y)\blambda_z^{-1/2}t^{-3\nu/2-1}\sigma_z\frac{(1+\abs{z})^{\nu/2}}{\abs{z}}\gamma_z^{\nu/2-1}\\
			\equiv\blambda_0^{2}W^4(y)\blambda_z^{-1/2}v_0^{-3\nu/2-1}\sigma_z\frac{(1+\abs{z})^{\nu/2}}{\abs{z}}\gamma_z^{\nu/2-1}\mod\Aext^{5(\nu/2-1),\min(\nu/2-3,-\nu/2-1)}(\D_0)
		\end{multline}
		Using that $\int 5W^4 \Lambda W=-2\pi$ yields
		\begin{equation}
			\P_{0,\Lambda}5W_0^4W_z=-2\pi\blambda_0^{2}\blambda_z^{-1/2}v_0^{-3\nu/2-1}\sigma_z\frac{(1+\abs{z})^{\nu/2}}{\abs{z}}\gamma_z^{\nu/2-1}+\A^{\nu/2-3}(\interval).
		\end{equation}
	\end{proof}
	
	At this point, we can only improve our ansatz if \cref{an:eq:leading_Kernel} is only of size $t^{\nu-1}$.
	We campture this scenario with the following definition.
	\begin{definition}[Admissibility]\label{an:def:admiss}
		We call the soliton velocities,scales and sign $(A,\vec{\lambda}^0,\vec{\sigma})$ admissible, if \cref{an:eq:leading_Kernel} has vanishing leading order term.
	\end{definition}
	
	Provided that our initial configuration is admissible, we immediatly get an improved ansatz.
	\begin{corollary}\label{an:cor:first}
		Let $(A,\vec{\blambda},\vec{\sigma})$ be admissible and set $\vec{\zeta}=0$. 
		Then, there exist $\phi_1\in\Aext^{3(\nu-1)+\nu/2-1-,\nu/2-1}(\Mcomp)$, such that $\phi_0+\phi_1$ is a smooth approximate solution with error $\Aext^{5\nu/2-5,-\nu/2-2}(\Mcomp)$.
	\end{corollary}
	\begin{proof}
		This result is a direct consequence of \cref{an:lemma:first} together with \cref{pre:lem:K}.
	\end{proof}
	
	We finish this section by showing that the set of admissible parameters is non-empty.
	\begin{lemma}\label{an:lemma:admiss}
		For any $\nu\in(8,\infty)\cup(1,2)$, there exists admissible $(A,\vec{\blambda},\vec{\sigma})$.
	\end{lemma}
	\begin{proof}
		Consider $\nu>8$ and the following parameters
		\begin{nalign}
			z_1=(-z,0,0),\quad z_2=(z,0,0),\quad z_0=(0,0,0)\\
			\lambda_{z_1}=1=\sigma_{z_1}=\sigma_{z_2}=\lambda_{z_2}=-\sigma_{0}=1,\quad \lambda_{0}=\lambda^2.
		\end{nalign}
		Using \cref{not:eq:velocity_paralell}, we write $\bar{z}=2z/(1+z^2)$ for the relativistic distance between $z_1,z_2$.
		Writing $f(z)=\frac{(1+\abs{z})^{\nu/2}}{2\abs{z}}\gamma_z^{\nu/2-1}$ and dividing \cref{an:eq:leading_Kernel} by $4\pi$, we observe that admissiblity holds if
		\begin{equation}
			\begin{cases}
				-\frac{\nu}{4}+f(z)\lambda-f(\bar{z})=0\\
				\frac{\nu\lambda}{4}-2f(z)=0
			\end{cases}\iff g(z):=\frac{8\big(f(z)\big)^2}{\nu}-\frac{\nu}{4}-f(\bar{z})=0.
		\end{equation}
		Using the asymptotics
		\begin{equation}
			f(z)\sim\begin{cases}
				2^{-1}\abs{z}^{-1},& z\sim 0\\
				2^{\nu/4-1/2}(1-\abs{z})^{1/2-\nu/4}& z\sim1
			\end{cases},\quad 
			f(\bar{z})\sim\begin{cases}
			\abs{4z}^{-1},& z\sim 0\\
			2^{\nu/2-1}(1-\abs{z})^{1-\nu/2}& z\sim1
			\end{cases}
		\end{equation}
		we obtain
		\begin{equation}
			g(z)\sim\begin{cases}
				(\frac{2}{\nu}-\frac{1}{4})\abs{z}^{-2}& z\sim0\\
				2^{\nu/2-1}\big(8/\nu-1\big)(1-\abs{z})^{1-\nu/2}&z\sim1.
			\end{cases}
		\end{equation}
		Note that for $\nu>8$, the second line has a negative prefactor.
		Since $g$ is a continous function on $(0,1)$, and $\lim_{z\to0} g(z)=-\lim_{z\to1} g(z)=\infty$, $g$ must have a zero, and thus an admissible configuration.
		
		A similar computation using the parameters
		\begin{equation}
			z_1=(-z,0,0),\quad z_0=(0,0,0),\quad \lambda_{z_1}=\lambda_{0}=\sigma_{z_1}=\sigma_0=1,
		\end{equation}
		yields an admissible solution for some $z\in(0,1)$ whenever $\nu\in(1,2)$.
	\end{proof}
	
	\subsection{Second iteration and velocity correction}\label{an:sec:second}
	
	We could already include the step detailed in this section in the next, but we find it instructive to compute the leading correction to $\vec{\zeta}$.
	
	\begin{lemma}\label{an:lemma:second}
		Let $(A,\vec{\blambda},\vec{\sigma},\phi_1)$ be as in \cref{an:cor:first}.
		Then, there exists $\boldsymbol{\zeta}_z\in\R^3$ and $\phi_-\in\Aext^{3\nu/2-2,3\nu/2-2}(\Mcomp)$ such that the following holds:
		set $\zeta_z=\boldsymbol{\zeta}_zv_z^\nu$, then
		$\phi=\phi_0+\phi_1+\phi_-$ satisfies $P[\phi]\in\Aext^{5(\nu/2-1)-,-\nu/2-2}(\Mcomp)$
		and
		\begin{equation}
			\P_{z,\partial}P[\phi]=\A^{\nu/2-3}(\interval).
		\end{equation}
	\end{lemma}
	
	\begin{proof}
		Let us first focus on $\phi_-$.
		As we already computed in \cref{an:eq:W_exp}, we get that
		\begin{equation}
			P[\phi_0+\phi_1]=\sum_{z\in A}\Box \lambda_z^{-1/2}\abs{x_z}^{-3}x_z\cdot\zeta_z\qquad \mod \A^{5\nu/2-5,-3\nu/2-1}(\Mcomp).
		\end{equation}
		We apply \cref{model:lem:i_smooth} to obtain $\phi_-\in\Aext^{3\nu/2-2,3\nu/2-2}(\Mcomp)$ with $\supp\phi_-\subset \cap_z\{y_z>10\}$ solving
		\begin{equation}
			\Box\phi_- =\sum_{z\in A}\Box \lambda_z^{-1/2}\abs{x_z}^{-3}x_z\cdot\zeta_z\qquad \text{in }\cap_z\{y_z>20\}.
		\end{equation}
		We compute the corresponding projection
		\begin{multline}\label{mod:eq:second_proof1}
			\P_{0,\partial}\normal{\K_0}\phi_-=\int_{\Sigma_{0,\tau}}\dd y^3 \tau^{-2\nu}\partial_{y_i} W(y)(\Delta_y+5W^4(y))\phi_-\lesssim \tau^{-2\nu}\big(y^2\abs{\partial_{y_i}W\partial_{\abs{y}}\phi_-}\\
			-y^2\abs{\partial_{\abs{y}}\partial_{y_i}W\phi_-}\big)|_{\partial\Sigma_{0,\tau}}
			\in\A^{\nu/2-3}(\interval).
		\end{multline}
		Using \cref{map:lem:lin} already implies that $	\P_{0,\partial}D_{\phi_0}P\phi_-\in\A^{\nu/2-3}(\interval)$.
		
		Next, we compute the cross terms, for which, we use the higher order terms from \cref{an:eq:Wz_expansion}
		\begin{multline}\label{an:eq:Wz_exp2}
			P_{\ell=1}^{\sphere}\chi_{<\delta}(x/t)W_0^4W_z=\blambda_0\blambda_z^{-1/2}v_0^{-\nu/2-2}\sigma_z\frac{(1+\abs{z})^{\nu/2}}{\abs{z}}\gamma_z^{\nu/2-1}z\cdot yW^4(y)\Big(\abs{z}^{-2}-\frac{\nu(1+\abs{z}^{-2})}{2(1+\abs{z})}\Big)\\
			\mod \Aext^{5(\nu/2-1),\nu/2-3}(\Mcomp_0).
		\end{multline}
		In particular, for some e$c_{z,\blambda}$ constants
		\begin{equation}\label{an:eq:proof_second_cross}
			\P_{0,\partial}\chi_{<\delta}(x/t)W_0^4W_z=c_{z,\lambda}(\blambda,A)v^{-\nu/2-2}+\A^{\nu/2-3}(\interval).
		\end{equation}
		
		We expand \cref{an:eq:PW_0} to higher order, to obtain via \cref{not:lemma:Box}
		\begin{multline}
			\sigma(\Box W_0+W_0^5)=-\lambda^{3/2}\Big(\frac{2h'(y)}{\abs{y}}+2h'(y)\partial_{\abs{y}}+h''(y)\Big)b\Lambda_y W(y)+\lambda^{3/2}\ddot{\zeta}\cdot\partial_yW+2\lambda^{3/2}\dot{\zeta}\cdot\partial_yb\Lambda_y W(y)
			\\ -(\lambda\dot{\zeta}\cdot\partial_y)^2\lambda^{1/2}W \mod \Aext^{5/2\nu-5,\nu/2-3}(\Mcomp_0)
		\end{multline}
		The leading order term has no projection onto $\P_{\ell=1}^{\sphere}$, and is also cancelled by $\phi_1$ in \cref{an:cor:first}.
		We compute the contribution of $\dot{\zeta},\ddot{\zeta}$ terms.
		We note that for some $C>0$
		\begin{equation}\label{mod:eq:second_proof2}
			\int_{\abs{y}\leq\delta v} \partial_{y_i}W(y)\partial_{y_j}W(y)=\delta_{ij}C+\A^{\nu-1}(\interval),\quad \int_{\abs{y}\leq\delta v} \partial_{y_j}W(y)\partial_{y_i}\Lambda W(y)=\A^{\nu-1}(\interval).
		\end{equation}
		Hence
		\begin{equation}
			\P_{0,\delta}\sigma(\Box W_0+W_0^5)=C\lambda^{3/2}\ddot{\zeta}+\A^{\nu/2-3}(\interval).
		\end{equation}
		This is sufficient to absorb the error terms in \cref{an:eq:proof_second_cross} by picking $\rho\in\A^{\nu}(\interval)$.
	\end{proof}
	
	\begin{corollary}\label{an:cor:second}
		Let $(A,\vec{\blambda},\vec{\sigma})$ be admissible. 
		Then, there exist $\phi_1\in\Aext^{\nu/2-1+2(\nu-1)-,\nu/2-1}(\Mcomp)$, such that $\phi_0+\phi_1$ is a smooth approximate solution with error $f\in\Aext^{(5\nu/2-5)-,-\nu/2-2}(\Mcomp)$ and moreover in each $\D_z$ there exists a spherically symmetric $f_z$ such that
		\begin{equation}
			f=f_z(y_z,t_z)+\Aext^{(5\nu/2-5)-,\nu/2-3}(\Mcomp)
		\end{equation}
	\end{corollary}
	\begin{proof}
		We note that $\Aext^{5(\nu/2-1),-\nu/2-2}(\Mcomp)=\Aext^{3(\nu-1)-\nu/2-2,-\nu/2-2}(\Mcomp)$.
		The existence is a consequence of \cref{an:lemma:second} 
	\end{proof}

	\subsection{Modulation}\label{an:sec:modulation}
	In order to iterate the construction of \cref{an:cor:first,an:cor:second}, we need to generalise the modulation for the location done in \cref{an:cor:second} and also introduce one for the scaling. 
	These are done in \cref{sec:modulation:position,sec:modulation:scale} respectively.
	Finally, in \cref{sec:modulation:radiative}, we  show how the same modulation for scaling can be achieved by sending nontrivial radiation in through $\C$, rather than changing $\lambda$.
	A similar modulation for $\zeta$ is outlined in a global context with polyhomogeneous errors in \cite{kadar_scattering_2024}
	
	\subsubsection{Scale modulation}\label{sec:modulation:scale}
	
	\begin{lemma}\label{mod:lemma:scale}
		Let $\S=(\phi_0,\phi_1,\lambda,\zeta)$ be an approximate solution.
		Furthermore, let $\bar{\lambda}\in\A^{\alpha}(\interval;\R^{\abs{A}})$ with $\alpha> -\nu$.
		Let $\bar{\phi}_0+\bar{\phi}_1$ be the function obtained by replacing $\lambda$ in $\phi_0+\phi_1$ by $\lambda+\bar{\lambda}$.
		Then
		\begin{subequations}\label{mod:eq:scale_estimates}
			\begin{align}
				P[\phi_0+\phi_1]-P[\bar\phi_0+\bar\phi_1]\in\Aext^{\alpha+\nu+5(\nu/2-1),\alpha-\nu/2-1}(\Mcomp),\\
				\P_{z,\partial}\big(P[\phi_0+\phi_1]-P[\bar\phi_0+\bar\phi_1]\big)\in\A^{\alpha+\nu/2-2}(\interval)\label{mod:eq:scale_estimates_partial},
			\end{align}
		\end{subequations}
		furthermore, there exists an $\abs{A}\times\abs{A}$ matrix $\mathfrak{M}(\nu,A)$, such that 
		\begin{equation}\label{mod:eq:scale_precise}
			\P_{z,\Lambda}\big(P[\phi_0+\phi_1]-P[\bar{\phi}_0+\bar{\phi}_1]\big)=\sigma_z\lambda^{1/2}_z\dot{\bar{\lambda}}_z+v_z^{-\nu/2-1}(\mathfrak{M}\bar{\lambda})_z+\A^{\alpha+\nu/2-2}(\interval).
		\end{equation}
	\end{lemma}
	\begin{remark}
		Let us note that \cref{mod:eq:scale_estimates_partial} is not sharp, and the error term can be improved by an extra factor of $\nu-1$.
	\end{remark}
	\begin{proof}
		We first note that for any function $\phi_1\in\Aext^{q_-,q_\K}(\Mcomp)$,
		\begin{equation}
			\phi_1-\bar{\phi}_1\in\Aext^{q_-+\nu+\alpha,q_\K+\nu+\alpha}(\Mcomp).
		\end{equation}
		Then, using \cref{not:lemma:Box}, we get that in $\D_0$
		\begin{multline}
			P[\phi_0+\phi_1]-P[\bar\phi_0+\bar\phi_1]=
			\Box(\phi_0+\phi_1)-\Box(\bar{\phi}_0+\bar{\phi}_1)+\phi_0^5-\bar{\phi}_0^5+5(\phi_0^4\phi_1-\bar{\phi}_0^4\bar{\phi}_1)\\
			=\lambda^{\frac{3}{2}}b LW(y)-\tilde{\lambda}^{\frac{3}{2}}\tilde{b} LW(\tilde{y})+5\sum_{z\in A\setminus\{0\}}(W_0^4W_z-\bar{W}_0^4\bar{W}_z)+(W_0^4\phi_1-\bar{W}_0^4\bar\phi_1)\mod \Aext^{\alpha+\nu+5(\nu/2-1),\alpha+\nu/2-2}(\D_0).
		\end{multline}
		This computation already yields \cref{mod:eq:scale_estimates} after taking a projection and using the spherical symmetry of the explicit terms.
		Computing the projection of the second term, we we find
		\begin{nalign}
			\P_{0,\Lambda} \tilde{\lambda}^{\frac{3}{2}}\tilde{b} L \Lambda W(\tilde{y})&= \tilde{\lambda}^{\frac{3}{2}}\tilde{b} \P_{0,\Lambda} L \Lambda W(y)+\tilde{\lambda}^{\frac{3}{2}}\tilde{b} \frac{\bar{\lambda}}{\lambda}\P_{0,\Lambda} y\partial_yL \Lambda W(y)+\A^{3\nu/2-3+\alpha}(\interval)\\
			&=\frac{\tilde{\lambda}^{\frac{3}{2}}\tilde{b}}{4}+\bar{\lambda}\dot{\lambda} c+\A^{3\nu/2-3+\alpha}(\interval)\\
			\implies \P_{0,\Lambda}(\lambda^{\frac{3}{2}} b LW(y)&-\tilde{\lambda}^{\frac{3}{2}}\tilde{b} LW(\tilde{y}))=(\frac{1}{2}+c_1)\lambda^{-1/2}\dot{\lambda}\bar{\lambda}+\lambda^{1/2}\dot{\bar{\lambda}}+\A^{3\nu/2-3+\alpha}(\interval).
		\end{nalign}
		where $c_1=\P_{0,\Lambda} y\partial_yL \Lambda W(y)\in\R$.
		We can also compute the projection of the cross terms
		\begin{nalign}
			\begin{multlined}
				I_z:=W_0^4(W_z-\bar{W}_z)+\bar{W}_z(W_0^4-\bar{W}_0^4)=-\frac{W_0^4\lambda_z^{-1/2}}{\abs{x_z}}\frac{\bar{\lambda}_z}{\lambda_z}-\frac{\lambda_z^{-1/2}}{\abs{x_z}}\frac{\bar{\lambda}}{\lambda}(2+y\cdot\partial_y)W_0^4\\
				\mod\A^{\alpha+\nu+5(\nu/2-1),\alpha+\nu/2-2}(\D_0)
			\end{multlined}\\
			\implies \P_{0,\Lambda}I_z=-\frac{\lambda_0^2\bar{\lambda}_z}{\gamma_z\abs{z}t\lambda_z^{3/2}}-c_2\frac{\lambda_0\bar{\lambda}_0}{\gamma_z\abs{z}t\lambda_z^{1/2}}+\A^{\alpha+\nu/2-2}(\interval),
		\end{nalign}
		where $c_2=\P_{0,\Lambda} (2+y\partial_y)W^4(y)\in\R$.
		The $\phi_1$ terms give a similar contribution.
		Combining all the contributions, we get
		\begin{equation}
			\P_{0,\Lambda}\big(P[\phi_0+\phi_1]-P[\bar\phi_0+\bar\phi_1]\big)=\lambda_0^{1/2}\dot{\bar{\lambda}}_0+\sum_{z\in A}c_{0,z}v_0^{-\nu/2-1}\bar{\lambda}_z+\A^{\alpha+\nu/2-2}(\interval),
		\end{equation}
		which yield the result for $\P_{0,\Lambda}$.
		The $\P_{z,\Lambda}$ computation follows similarly.
	\end{proof}
	
	\begin{corollary}\label{mod:cor:scale}
		Let $\phi_0+\phi_1$ be an approximate solution with error $\Aext^{a_-,a_\K}$ for $a_-\in a_\K+(3(\nu-1),4(\nu-1))$ and $a_\K\geq-\nu/2-2$.
		Then, there exist another approximate solution $\bar{\phi}_0+\bar{\phi}_1=\phi_0+\phi_1+\A^{a_\K+3\nu/2+1,a_\K+\nu/2+1}(\Mcomp)$ with the same error satisfying
		\begin{nalign}
			\P_{z,\Lambda}P[\bar{\phi}_0+\bar{\phi}_1]&\in\A^{a_\K+\nu-1}(\interval)\\
			\P_{z,\partial}\big(P[\bar{\phi}_0+\bar{\phi}_1]-P[{\phi}_0+{\phi}_1]\big)&\in\A^{a_\K+\nu-1}(\interval).
		\end{nalign}
	\end{corollary}
	\begin{proof}
		This is a direct consequence of \cref{mod:lemma:scale}.
	\end{proof}
	
	\subsubsection{Position modulation}\label{sec:modulation:position}
	
	\begin{lemma}\label{mod:lemma:pos}
		Let $\S=(\phi_0,\phi_1,\lambda,\zeta)$ be an approximate solutionwith error $\Aext^{a_-,a_\K}(\Mcomp)$, with $a_\K>-\nu/2-2$ and $a_->a_\K+2(\nu-1)$.
		Fix $\bar{\zeta}\in\A^{\alpha}(\interval;\R^{3\abs{A}})$ with $\alpha\geq\nu+1$, and define $\bar{\phi}_0,\bar\phi_1$ by changing $\zeta$ to $\zeta+\bar{\zeta}$.
		Then, there exists $\phi_-\in\A^{?}(\Mcomp)$, such that $\phi=\bar{\phi}_0+\bar{\phi}_1+\phi_-$ satisfies
		\begin{equation}\label{mod:eq:pos_M_decay}
			P[\phi]-P[\phi_0+\phi_1]\in\Aext^{\alpha+3\nu-7,\alpha-3\nu/2-2}(\Mcomp).
		\end{equation}
		Furthermore for some $C>0$
		\begin{equation}\label{mod:eq:pos_correction}
			\P_{z,\partial}(P[\phi]-P[\phi_0+\phi_1])=C\lambda_z^{3/2}\ddot{\zeta}+\A^{\alpha-\nu/2-3}(\interval).
		\end{equation}
	\end{lemma}
	\begin{proof}
		The proof essentially follows the computations of \cref{an:lemma:second}.
		We note that using \cref{not:lemma:Box}, we already have
		\begin{equation}
			P[\bar{\phi}_0+\bar\phi_1]-P[\phi_0+\phi_1]\in\A^{\alpha+\nu/2-4,\alpha-3\nu/2-2}.
		\end{equation}
		Using \cref{model:lem:i_smooth}, we know that, there exists $\phi_-\in\A^{\alpha+\nu/2-2,\alpha+\nu/2-2}(\Mcomp)$, such that 
		\begin{equation}
			\Box\phi_--\big(P[\bar{\phi}_0+\bar\phi_1]-P[\phi_0+\phi_1]\big)=0 \text{ in }\cap_z{\abs{y_z}>10}.
		\end{equation}
		This already proves \cref{mod:eq:pos_M_decay} with \cref{map:lem:lin,map:lem:nonlin}.
		It is easy to see that $\phi_-$ does not contribute at leading order to the position modulation, and indeed as in \cref{mod:eq:second_proof1}, we have
		\begin{equation}
			\P_{0,\partial}\normal{\K_0}\phi_-\in\A^{\alpha-\nu/2-3}.
		\end{equation}
		
		We also observe that from \cref{not:lemma:Box}, it follows that in $\D_0$
		\begin{equation}
			\Box(\phi_0+\phi_1)-\Box(\bar\phi_0+\bar{\phi}_1)=\lambda^{3/2}\ddot{\bar{\zeta}}\cdot\partial_y W(y)+2\lambda\dot\zeta\partial_y\Lambda W+\A^{\alpha+\nu/2-4,\min(\alpha-\nu/2-3,2\alpha-5\nu/2-2)}(\D_0).
		\end{equation}
		In order to compute the position modulation to leading order, it suffices to consider the first two terms.
		Using the computation from \cref{mod:eq:second_proof2}, we get
		\begin{equation}
			\P_{0,\partial}\big(\lambda^{3/2}\ddot{\bar{\zeta}}\cdot\partial_y W(y)+2\lambda\dot\zeta\partial_y\Lambda W\big)=\lambda^{3/2}C\ddot{\zeta}+\A^{\alpha-\nu/2-3}(\interval).
		\end{equation}
		Combining with the other errors, this yields the result.
	\end{proof}

	\begin{corollary}\label{mod:cor:pos}
		Let $\S=(\phi_0,\phi_1,\lambda,\zeta)$ be an approximate solution with error $\Aext^{a_-,a_\K}$ for $a_-\in a_\K+(3(\nu-1),4(\nu-1))$ and $a_\K\geq-\nu/2-2$.
		Then, there exist another approximate solution $\bar{\phi}_0+\bar{\phi}_1=\phi_0+\phi_1+\A^{a_\K+3\nu/2+1,a_\K+\nu/2+1}(\Mcomp)$ with the same error satisfying
		\begin{nalign}
			\P_{z,\partial}\big(P[\bar{\phi}_0+\bar{\phi}_1]\big)&\in\A^{a_\K+\nu-1}(\interval).
		\end{nalign}
	\end{corollary}
	
	\subsubsection{Radiative modulation}\label{sec:modulation:radiative}
	We begin with the case, when the outgoing radiation, leaving through the backward lightcone modulates the solitons, thereby fixing their singularity rate.
	\begin{lemma}[Radiative single modulation]\label{mod:lem:rad_single}
		Let $\S$ be an approximate solution with $\phi_1=0$ and $A=\{0\}$.
		Let $f\in\A^{a_\K}(\interval)$.
		Then, there exists $\phi_{\mathrm{rad}}\in\Aext^{a_\K+2\nu}(\Mcomp_\emptyset)+\A^{a_\K+2\nu+1,a_\K+2\nu}(\Mcomp_\emptyset)$ solving $\Box\phi_{\mathrm{rad}}=0$ such that
		\begin{equation}\label{mod:eq:rad_single}
			\P_{0,\Lambda}(D_{\phi_0}P\phi_{\mathrm{rad}})=f+\A^{a_\K+(\nu-1)}(\interval).
		\end{equation}
		
		Fix $c\in\R$.
		Provided that $a_\K+2\nu\in\Z_{\neq-1}$, there exists $\phi_{\mathrm{rad}}\in\Aext^{a_\K+2\nu}(\Mcomp_\emptyset)$ with $\Box\phi_{\mathrm{rad}}=0$ such that
		\begin{equation}\label{mod:eq:rad_single_smooth}
			\P_{0,\Lambda}(D_{\phi_0}P\phi_{\mathrm{rad}})=-ct^{a_\K}+\A^{a_\K+(\nu-1)}(\interval).
		\end{equation}
	\end{lemma}
	\begin{proof}
		We find $\phi=\frac{\psi(u)-\psi(v)}{r}$ as a spherically symmetric function, and note that $\phi|_{x=0}=\psi'(t)$.
		Noting that 
		\begin{equation}
			D_{\phi_0}P\phi_{\mathrm{rad}}=W_0^4\phi_{\mathrm{rad}}(0,t)+\A^{a_\K+2\nu+5(\nu/2-1),a_\K+\nu-1}(\Mcomp_0)
		\end{equation}
		 \cref{mod:eq:rad_single} follows by setting $\psi=\int \tau^{2\nu}f$.
		
		To prove \cref{mod:eq:rad_single_smooth}, we use a multiple of the explicit ansatz $\frac{(t+r)^{a_\K+2\nu+1}-(\mathring{u})^{a_\K+2\nu+1}}{r}$, which is smooth across $\C$ for $a_\K+2\nu+1\in\Z_{\neq0 }$.
	\end{proof}
	
	We can generalise this modulation to multiple soliton velocities:
	
	\begin{prop}[Radiative modulation]\label{mod:lem:rad_multi}
		Let $\phi_0+\phi_1$ be an approximate soluiton with $\phi_1=0$.
		Let $c\in\R^{\vec{A}}$ .
		Then, there exists $\phi_{\mathrm{rad}}\in\Aext^{a_\K+2\nu}(\Mcomp_\emptyset)+\A^{a_\K+2\nu+1,a_\K+2\nu}(\Mcomp_\emptyset)$ solving $\Box\phi=0$ such that 
		\begin{equation}\label{mod:eq:scale}
			\P_{z,\Lambda}(D_{\phi_0}P\phi_{\mathrm{rad}})=c_z\tau^{a_\K}+\A^{a_\K+(\nu-1)}(\interval),\qquad \forall z\in A.
		\end{equation}
		
		Fix $\abs{A}\leq 4$ and $\vec{c}\in\R^{\abs{A}}$.
		Provided that $a_\K+2\nu\in\Z\setminus\{-2,-1,0\}$, there exists $\phi_{\mathrm{rad}}\in\Aext^{a_\K+2\nu}(\Mcomp_\emptyset)$ with $\Box\phi_{\mathrm{rad}}=0$ such that \cref{mod:eq:scale} holds.
	\end{prop}
	\begin{proof}
		We start with \cref{mod:eq:scale}.
		We will be brief with the proof as this follows from \cite[Lemma 7.17]{kadar_scattering_2024}.
		As in \cref{mod:lem:rad_single} it suffices to solve $\Box\phi=0$ with $\phi|_{x_z=0}=c_zt^{a_\K+2\nu}$.
		Let us write $\phi=t^{a_\K+2\nu}\psi(x/t)$, so that $\psi$ solves an elliptic equation, conjugated to the shifted Laplacian on hyperbolic space \cite[(7.13)]{kadar_scattering_2024}, with a regular singular point at $x/t=1$ and corresponding indicial roots $a_\K+2\nu+1,0$.
		Therefore, provided that a solution $\phi$ exists, it is  smooth in the interior and $\psi\in C^\infty(B)+\A^{a_\K+2\nu+1}(B)$.

		Next, we set $G_{z,\tau}$ to be the Green's function $\Box G_{z,\tau}=\delta(x-zt)\delta(t-\tau)$ and write \cref{mod:eq:scale} as		
		\begin{equation}
			\jpns{\phi,\Box G_{z,\tau}}=\int_{\C}(\phi\partial_v G_{z,\tau}-G_{z,\tau}\partial_v \phi) \dd u\dd\abs{\slashed{g}}=c_z \tau^{a_\K+2\nu}
			\A^{a_\K+\nu-1}(\interval_\tau).
		\end{equation}
		Thus, it suffices to prove that the integration kernels are linearly independent, i.e. there exists no $\bar{c}_z$ such that
		\begin{equation}
			\sum \bar{c}_z\int_{\C}t_z^{-(a_\K+2\nu)}(\partial_v G_{z,\tau}-G_{z,\tau}) \dd u\dd\abs{\slashed{g}}=0.
		\end{equation}
		However, using scattering theory with the above initial data, we also obtain that
		\begin{equation}
			\sum \bar{c}_z\delta(x-zt)\delta(t-\tau)\tau^{-(a_\K+2\nu)}=0\implies \bar{c}_z=0.
		\end{equation}
		Thus indeed the integration kernels are linearly independent.
		
		For the second part, we observe that for $Y_1^{i}$, with $i\in\{1,2,3\}$ linearly independent eigenfunctions $\slashed{\Delta}Y^{i}_1=-2Y^{i}_1$ and $p\notin\{-1,0,1\}$,
		\begin{equation}
			r^{-2}((1+p)ru^{p-1}+2u^{p+1}+(1+p)rv^{p}-2v^{p+1})Y_1\in\Aext^{p-1,p+\nu-2}(\Mcomp_0).
		\end{equation}
		is a non-vanishing regular solution of $\Box_\eta$.
		Choosing an appropriate multiple of these and the spherically symmetric solution yields the result, we can obtain any leading order value at 4 distinct points.\footnote{Note that in general, due to integer coincidences, a logarithmic singularity might form along the lightcone and so we do not expect the result to be true for arbitrary $\abs{A}$. See \cite[Remark 10.5]{kadar_scattering_2025}}
	\end{proof}

	\subsection{Induction}\label{an:sec:iteration}
	In this section, we prove the main inductive argument on how to improve an approximate solution, and use it to prove \cref{an:prop:smooth_multi}.
	
	\begin{lemma}
		Let $\S=(\phi_0,\phi_1,\lambda,\zeta)$ be an approximate solution with error $\Aext^{a_-,a_\K}(\Mcomp)$ satisfying $a_-\geq 5\nu/2-5$ and $a_\K=a_--2(\nu-1)$.
		Then, there exist another approximate solution with an error $\Aext^{a_-+(\nu-1),a_\K+(\nu-1)}(\Mcomp)$.
	\end{lemma}
	\begin{proof}
		Let us write $f_1=P[\phi_0+\phi_1]\in\Aext^{a_-,a_\K}(\Mcomp)$.
		We first apply \cref{model:lem:i_smooth} to find $\phi_-\in\Aext^{a_-+2,a_-+2}(\Mcomp)$ solving $\Box_\eta\phi_-=f_1$.
		Using \cref{map:lem:lin,map:lem:nonlin}, this already yields
		\begin{equation}
			P[\phi_0+\phi_1+\phi_-]\in\Aext^{a_-+2(\nu-1),a_\K}(\Mcomp).
		\end{equation}
		
		Next, we apply \cref{mod:cor:pos} and then \cref{mod:cor:scale} to obtain first position and then scale corrections $\vec{\bar{\zeta}}$ and $\vec{\bar{\lambda}}$ such that the corresponding modulated ansatz $\bar{\phi}_0,\bar{\phi}_1$ satisfies
		\begin{nalign}
			P[\bar\phi_0+\bar\phi_1+\phi_-]\in\Aext^{a_-+2(\nu-1),a_\K}(\Mcomp)\\
			\P_{z,\Lambda}P[\bar\phi_0+\bar\phi_1+\phi_-],\P_{z,\partial}P[\bar\phi_0+\bar\phi_1+\phi_-]\in\Aext^{a_\K+\nu-1}(\interval).
		\end{nalign}
		
		Next, we apply \cref{pre:lem:K} to conlude that there exists $\phi_\K\in\Aext^{a_\K+2\nu+3(\nu-1),a_\K+2\nu}(\Mcomp)$ satisfying
		\begin{equation}
			P[\bar\phi_0+\bar\phi_1+\phi_-]-\normal{K}\phi_\K\in\Aext^{a_-+2(\nu-1),a_\K+\nu-1}(\Mcomp).
		\end{equation}
		Using \cref{map:lem:lin,map:lem:nonlin} the result follows.
	\end{proof}
	
	Iterating the lemma an arbitrary number of times and using Borel summation yields the result.
	
	\subsection{Radiative ansatz}\label{an:sec:rad}
	Next, let us turn to the proof of \cref{an:prop:smooth_quantized} and \cref{an:prop:multi}.
	\begin{lemma}
		a) Let $(\nu,A,\vec{\blambda},\vec{\sigma})$ be as in \cref{an:prop:smooth_quantized}.
		Then, there exist $\phi_1\in\Aext^{2(\nu-1)+,\nu/2-1}(\Mcomp)+\Aext^{\nu/2-1}(\Mcomp_\emptyset)$, such that $(\phi_0,\phi_1,\lambda,0)$ is a smooth radiative approximate solution with error $\Aext^{5\nu/2-5,\nu/2-3}(\Mcomp)$.
		b) Let $(\nu,A,\vec{\blambda},\vec{\sigma})$ be as in \cref{an:prop:multi}.
		Then, there exist $\phi_1\in\Aext^{\nu/2-1}(\Mcomp_\emptyset)+\A^{\nu/2,\nu/2-1,\nu/2-1}(\Mcomp)$, such that $(\phi_0,\phi_1,\lambda,0)$ is a smooth radiative approximate solution with error $\Aext^{5\nu/2-5,\nu/2-3}(\Mcomp)$.
	\end{lemma}
	\begin{proof}
		a)
		Using \cref{an:lemma:first}, we can already compute the leading order kernel elements for the scaling.
		We apply \cref{mod:lem:rad_multi} to obtain $\phi_-\in\Aext^{\nu/2-1}(\Mcomp_\emptyset)$, such that 
		\begin{equation}
			\P_{z,\Lambda}P[\phi_0+\phi_-]\in\A^{-\nu/2-2}.
		\end{equation}
		
		Next, we apply \cref{model:cor:K}, to obtain $\phi_1\in\Aext^{2(\nu-1)+,\nu/2-1}(\Mcomp)$ such that 
		\begin{equation}
			P[\phi_0+\phi_1+\phi_\K]\in\Aext^{5\nu/2-5,\nu/2-3}(\Mcomp).
		\end{equation}
		
		b) This follows the same way, but the correction $\phi_1$ from \cref{mod:lem:rad_multi} is not smooth across $\C$.
	\end{proof}

	\begin{proof}[Proof of \cref{an:prop:smooth_quantized}]
		The rest is the exact same as for \cref{an:prop:smooth_multi}.
	\end{proof}

	\subsection{Exponential ansatz}\label{an:sec:exp}
	Next, let us consider the case of exponential singularity formation.
	Since, in this case, we need to revise the notation, for simplicity, we work with a single soliton in spherical symmetry, but applying the previous methods, we believe that multisoliton construction is also possible in this case.
	
	In this section, we use the $(v,y)$ coordinates as defined in \cref{not:eq:vy_def} with the replacement $\lambda(s)=e^{\nu/s}$.
	It is clear that these new coordinates still give a diffeomorphism as in \cref{not:lemma:diffeo}.
	We define $\rho_0=t^{-1}e^{-\nu/t}\jpns{y}$, $\rho_-=t/\jpns{y}$, so that $\rho_0|_{\abs{y}=1}\sim e^{-\nu/t}$ and $\rho_-|_{\abs{x}/t=1/2}\sim e^{-\nu/t}$.
	
	\begin{proof}[Proof of \cref{an:prop:multi}]
		We first prove show the ansatz with $\rho_-^{\nu/2}\rho_0^{-\nu/2}L^\infty$ error.
		Let us note the computation
		\begin{equation}
			P[\phi_0]\in \rho_-^{5\nu/2}\rho_0^{-3\nu/2}L^\infty,
		\end{equation}
		which follows similarly as in \cref{an:lemma:first}, using \cref{an:eq:W_exp}.
		
		We mimic the computation \cref{an:eq:leading_Kernel} with the exponential weight to obtain that
		\begin{equation}
			\P_{0,\Lambda}P[\phi_0]=-e^{\frac{3\nu}{2t}}\frac{\nu}{4t^2}+e^{-\nu/(2t)}L^\infty(\interval).
		\end{equation}
		We fix radiation $\phi_1\in \rho_-^{\nu/2}L^\infty$ solving $\Box\phi_1=0$ using \cref{mod:lem:rad_single} so that
		\begin{equation}
			\P_{0,\Lambda}P[\phi_0+\phi_1]=e^{-\nu/(2t)}L^\infty(\interval).
		\end{equation}
		
		Next, we solve the elliptic problem via \cref{model:cor:K}, to obtain $\phi_2\in\rho_-^{7\nu/2}\rho_0^{\nu/2}L^\infty$ that yields
		\begin{equation}
			P[\phi_0+\phi_1+\phi_2]\in \rho_-^{5\nu/2}\rho_0^{-\frac{\nu}{2}}L^\infty.
		\end{equation}
		Repeating the above two steps once more yields the result.
	\end{proof}

	\subsection{Peeling}
	Finally, it will be convenient to improve the approximate solutions constructed in the previous section, by removing their leading order behaviour towards $\C$ as well.
	\begin{prop}\label{an:prop:peeling}
		Let $(\phi_0,\phi_1,\lambda,\zeta)$ be the approximate solution constructed in one of \cref{an:prop:smooth_quantized,an:prop:smooth_multi}.
		Then, there exists $\phi_2\in\Aext^{\infty}(\Mcomp_\emptyset)$, such that 
		\begin{equation}
			P[\phi_0+\phi_1+\phi_2]=(t-\abs{x})^\infty\Aext^{\infty,\infty}(\Mcomp).
		\end{equation}
		
		Similarly, for $(\phi_0,\phi_1,\lambda,\zeta)$ the approximate solution constructed in one of \cref{an:prop:multi} there exists $\phi_2\in\Aext^{\infty}(\Mcomp_\emptyset)+\A^{2\nu+1,\infty}(\Mcomp_\emptyset)$, such that 
		\begin{equation}
			P[\phi_0+\phi_1+\phi_2]=(t-\abs{x})^\infty\Aext^{\infty,\infty}(\Mcomp).
		\end{equation}
	\end{prop}
	\begin{proof}
		We use the following result
		\begin{claim}
			Given $p\geq0$ and $f\in\mathring{u}^{p}\Aext^{\infty}(\Mcomp_\emptyset)$, there exists $\psi\in\mathring{u}^{p+1}\Aext^{\infty}(\Mcomp_\emptyset)$, such that
			\begin{equation}
				\Box\psi-f\in\mathring{u}^{p+1}\Aext^{\infty}(\Mcomp_\emptyset).
			\end{equation}
		\end{claim}
		This claim follows by solving $\partial_{\mathring{u}}\partial_{\mathring{v}}r\psi=rf$ near $\C$ by simply integrating the right hand side first in $\mathring{u}$ then in $\mathring{v}$ and finally localising to $\D_{\ext}$.
		This yields $\psi$ the required space, and we also get that the remaining part of $\Box$ has the required decay property.
		
		Applying the claim iteratively, we can obtain the result with $\mathring{u}^N$ decay and Borel summation yields the superpolynomial result.
		\end{proof}

	\section{Linear estimate}\label{sec:lin}
		In this section, we prove estimates for the linearised problem around the approximate solutions constructed in \cref{sec:an}.
	Let $\phi_A=\phi_0+\phi_1$ be an approximate solution with rate $\nu$ and error $\A^{\infty,\infty,\infty}(\Mcomp)$, and consider 
	\begin{equation}\label{lin:eq:main_cutoffed}
		D_{\phi_A}P \phi=f\in\A^{\infty,\infty,\infty}(\Mcomp),\quad \phi|_{\C}=0,\quad \supp f\cap\partial\Mcomp=\emptyset.
	\end{equation}
	We prove uniform estimates that only depend on the $\Hb^{k;N,N,N}$ norm of $f$, for $N$ sufficiently large depending on $k$.
	
	\begin{prop}\label{lin:prop:energyEstimate}
		Let $\phi$ be a solution of \cref{lin:eq:main_cutoffed} and $\chi$ be a cutoff function localising to $[t_*,2t_*]$ for some $t_*\ll1$.
		Let $s_0(q)=\frac{q+1/2-3\nu}{\nu-1}$.
		There exists $q_0$ sufficiently large such that for all $q>q_0$ and $s\in \N_{0\leq\cdot\leq s_0(q)}$ the following estimate holds
		\begin{equation}\label{lin:eq:energyEstimate}
			\norm{\Ve \{1,\rho^\nu\partial\}^s\phi}_{\Hb^{0;q,q,q}(\Mcomp)}\lesssim_s \norm{\{1,\rho^{\nu}\partial\}^s f}_{\Hb^{0;q-1,q+5\frac{\nu-1}{2},q}(\Mcomp)}+\norm{\chi \phi}_{H^1(\Mcomp)}
		\end{equation}
	\end{prop}
	\begin{proof}
		The proof is given in \cref{lin:sec:L2} after introducing all the necessary notation and additional computation in the preceding sections.
	\end{proof}
	We already note that the spaces
	\begin{nalign}
		\mathring{X}_{s,q}&=\{\phi\in\Hb^{0;q,q,q}(\Mcomp):\Ve \{1,\rho^{\nu}\partial\}^s\phi\in \Hb^{0;q,q,q}(\Mcomp)\}\\
		Y_{s,q}&=\{f\in \Hb^{0;q-1,q+5\frac{\nu-1}{2},q}(\Mcomp):\{1,\rho^{\nu}\partial\}^s f \in \Hb^{0;q-1,q+5\frac{\nu-1}{2},q}(\Mcomp)\}
	\end{nalign}
	satisfy the inclusions
	\begin{nalign}\label{lin:eq:inclusions}
		&\mathring{X}_{s,q}\subset \Hb^{0;q-s(\nu-1),q-s(\nu-1),q-s(\nu-1)}(\Mcomp)\\
		&\Hb^{0;q-1,q+5\frac{\nu-1}{2},q}(\Mcomp) \subset Y_{s,q}
	\end{nalign}
	
	We can use the energy estimate to deduce the solvability of the scattering problem for \cref{lin:eq:main_cutoffed}.
	\begin{theorem}\label{lin:thm:right_inverse}
		Fix $q_0$ as in \cref{lin:prop:energyEstimate} and $q\geq q_0,\,s\in[0,s(q)]$.
		There exists
		\begin{equation}
			\mathfrak{S}:Y_{s,q}\to \mathring{X}_{s,q},
		\end{equation}
		a bounded linear right inverse of $D_{\phi_A}P$, satisfying
		\begin{equation}\label{lin:eq:right_inverse}
			D_{\phi_A}P \mathfrak{S}f=f,\quad \norm{\Ve \{1,t^{2\nu-1}\partial\}^s\mathfrak{S}f}_{\Hb^{0;q+\frac{\nu-1}{2},q+\frac{\nu-1}{2},q}(\Mcomp)}\leq C_{s,q}\norm{\{1,t^{2\nu-1}\partial\}^s f}_{\Hb^{0;q+\frac{\nu-1}{2},q+5\frac{\nu-1}{2},q}(\Mcomp)}.
		\end{equation}
		
		In fact, there exists a right inverse $\mathfrak{S}_\infty$ and $s_1,q_1$ such that for all $s>s_1,q\geq q_1$ the estimate \cref{lin:eq:right_inverse} holds.
	\end{theorem}
	\begin{proof}
		\emph{Finite regularity:}
		
		The estimate \cref{lin:eq:energyEstimate} implies that the image of
		\begin{nalign}
			&\D_{\phi_A}P:X_{s,q}\to Y_{s,q}\\
			&X_{s,q}=\{\phi\in\mathring{X}_{s,q}:\D_{\phi_A}P\phi\in Y_{s,q}\}
		\end{nalign}
		 is closed and that the kernel is finite dimensional.
		In particular, we can split $X_{s;q}=W_{s,q}\oplus\ker \D_{\phi_A}P$, where $W_{s,q}$ is a closed subspace.

		The adjoint
		\begin{equation}
			(\D_{\phi_A}P)^*:Y^*_{s,q}\to X^*_{s,q}
		\end{equation} 
		is defined as the forward solution with low regularity, high growth towards the tip of the cone \emph{and} zero initial data on $\{t=2t_*\}$.
		Using well posedness of the forward solution of the Cauchy problem (for instance as given in \cite[Prop. 10.33]{hintz_introduction_2025}), it follows that $\ker (\D_{\phi_A}P)^*=\{0\}$.
		Hence, we conclude that $\D_{\phi_A}P$  is surjective.
		
		We conclude that $\D_{\phi_A}P|_{W_{s,q}}$ is a bijection and $(\D_{\phi_A}P|_{W_{s,q}})^{-1}$ satisfies \cref{lin:eq:right_inverse}.
		
		\emph{Smooth inverse:}
		Since $X_{s,q}\subset X_{s',q'}$ for any $q>q'$ and $s>s'$, it follows that
		\begin{equation}
			\ker \D_{\phi_A}P|_{X_{s,q}}\subset \ker\D_{\phi_A}P|_{X_{s',q'}}.
		\end{equation}
		Using that all these kernels are finite dimensional, we also get that there exists $s_1,q_1$ such that for all $q>q_1,s>s_1$
		it holds that
		\begin{equation}
			\ker \D_{\phi_A}P|_{X_{s,q}}= \ker\D_{\phi_A}P|_{X_{s_1,q_1}}.
		\end{equation}
		Let's consider a direct sum $X_{s_1,q_1}=\ker\D_{\phi_A}P|_{X_{s_1,q_1}}\oplus W_{s_1,q_1}$.
		Using the finite dimensionality of the kernel, we obtain that $W:=W_{s_1,q_1}\cap X_{s,q}$ is a closed subspace of $X_{s,q}$.
		Thus, defining the right inverse $\mathfrak{S}_{\infty}=(\D_{\phi_A}P|_{W_{q_1,s_1}})^{-1}$ yields the required smoothness property.
	\end{proof}

	Let us briefly explain the idea of this section.
	We focus on the case of a single soliton here, i.e. $A=\{0\}$, for sake of exposition.
	We observe that the naiv energy estimate, which is obtained by applying the divergence theorem with $J=\partial_t\cdot\T^{5\phi_A^4}[\phi]$ yields a non-coercive flux
	\begin{multline}\label{lin:eq:flux_energy}
		\int_{\Sigma} J=\int \dd x^3 (1-h'{}^2)(\partial_t\phi)^2+t^{-2\nu}(\partial_{\boldsymbol{x}}|_{t_*}\phi)^2-5\phi_A^4\phi^2\\
		=\int \dd x^3 (1-h'{}^2)(\partial_t\phi)^2+t^{-2\nu}(\partial_{\boldsymbol{x}}|_{t_*}\phi)^2-5t^{-2\nu}W^4(\boldsymbol{x})\phi^2+\mathcal{O}(t^{-\nu-1}\jpns{\boldsymbol{x}}^{-3})\phi^2\\
		=t^{3\nu}\int \dd \boldsymbol{x}^3  (1-h'{}^2)(\partial_t\phi)^2+t^{-2\nu}(\partial_{\boldsymbol{x}}|_{t_*}\phi)^2-5t^{-2\nu}W^4(\boldsymbol{x})\phi^2+\mathcal{O}(t^{-\nu-1}\jpns{\boldsymbol{x}}^{-3})\phi^2
	\end{multline}
	as well as a bulk term (with unfavourable sign for scattering construction)
	\begin{equation}\label{lin:eq:bulk_schematic}
		\int_{D}\div J=\int \dd x^3 \dd t \phi^2\partial_t\phi_A^4+f\partial_t\phi.
	\end{equation}
	The non-coercivity of $J$ to leading order in decay is caused by $\ker (\Delta+5W^4)$ and the eigenvalue $(\Delta+5W^4)Y=\kappa^2Y$ for $\kappa>0$.
	The main observation of this section is that by appending this trivial energy estimate with a study of the eigenfunctions and the kernel elements, it is possible to correct for the loss of coercivity of $J$.
	The bad bulk terms--for instance in \cref{lin:eq:bulk_schematic}--are treated via a Gronwall type argument, which requires taking the decay rate, $N$, sufficiently large.

	\subsection{Energy estimate}\label{lin:sec:energy}
	We use a $t$ weighted energy estimate that yields a coercive control of solutions to \cref{lin:eq:main_cutoffed} away from the location of the solitons.
	These are standard energy estimates adapted to moving bodies as discussed in \cite{kadar_construction_2024,angelopoulos_matching_2026}.
	We begin by recalling the existence of a function interpolating between $t_z$:
	\begin{lemma}[Time functions]\label{lin:lemma:tau}
		There exist smooth functions $s,s_*$ with the properties that $\dd s$ is timelike in $\M$ with uniform bound
		\begin{equation}\label{lin:eq:tau_uniform}
			\eta(\dd s,\dd s)<c<0,\quad \eta(\dd s_*,\dd s_*)\leq\begin{cases}
				c<0&\text{ in }\cap \{\abs{x_z}/t_z>\delta \}\\
				-c^{-1}t^{\nu-1} &\text{ else}
			\end{cases}\\
		\end{equation}
		together with the following properties
		\begin{equation}
			 s_*=\begin{cases}
				t& \text{in } \{1-\abs{x}/t<\delta/2\},\\
				\gamma_z^{-1}v_z & \text{in } \{\abs{x_z}/t_z<\delta/2\},
			\end{cases}\qquad
			 s=\begin{cases}
			t& \text{in } \{1-\abs{x}/t<\delta/2\},\\
				\gamma_z^{-1}t_z & \text{in } \{\abs{x_z}/t_z<\delta/2\}.
			\end{cases}
		\end{equation}
	\end{lemma}
	\begin{proof}
		We start with the construction of $ s$ following the proof of  \cite{angelopoulos_matching_2026}[Lemma 8.1] closely.
		Let us write
		\begin{equation}
			 s=\begin{cases}
				t&\text{in }\{\abs{x_z}/t_z>\delta\}\\
				(1-\abs{z}^2)t&\text{in }\{\abs{x_z}/t_z<\delta/2\}\\
				t-\chi(\abs{x_z}/t_z)z\cdot x&\text{in }\{\delta/2<\abs{x_z}/t_z<\delta\}
			\end{cases}
		\end{equation}
		Where the cutoff function $\chi$ is to be determined.
		Clearly in the exterior $\{\abs{x_z}/t_z>\delta\}$  and interior $\{\abs{x_z}/t_z<\delta/2\}$ regions, the bound \cref{lin:eq:tau_uniform} holds.
		It suffices to study the transition zone, where $\chi'\neq0$.
		Here, we rewrite $ s=\gamma_z(t_z+\chi(\abs{x_z}/t_z)z\cdot x_z)$ and compute
		\begin{equation}
			\gamma_z^{-2}\eta(\dd s,\dd s)\leq-1+\chi'{}^2\frac{(z\cdot x_z \abs{x_z})^2}{s_z^4}+\eta\Big(\chi z\cdot \dd x_z+\chi'\frac{z\cdot x_z\dd\abs{x_z}}{s_z}\Big).
		\end{equation}
		We can bound the last term by $\abs{z}^2(1+\sup_\rho \rho\chi'(\rho))$.
		Choosing a cutoff such that $\sup_\rho \rho\chi'(\rho)$ is sufficiently small yields that $\dd s$ is strictly timelike in the transition region.
		
		For $ s_*$, we start with $ s$ and modify it in the regions $\{\abs{x_z}/t_z<\delta/2\}$.
		We write
		\begin{equation}
			 s_*=\begin{cases}
				t&\text{in }\{\abs{x_z}/t_z>\delta/2\},\\
				\gamma^{-1}\left(t_z+\chi(\abs{x_z}/t_z)h(y_z)\lambda_z^{-1}(v_z)\right)& \text{in }\{\abs{x_z}/t_z<\delta/2\}.
			\end{cases}
		\end{equation}
		It is clear that $ s_*$ is timelike wherever it matches with $ s$, and we also have
		\begin{equation}\label{lin:eq:dv}
			\dd v_z=\dd t_z+h' \dd\abs{x_z-\zeta_z(t_z)}+(h'y_z-h)\frac{\dot{\lambda}_z}{\lambda_z^2}\dd v_z\implies(1+\A^{0,\nu-1,\nu-1}(\Mcomp))\dd v_z=\dd t_z+h' \dd\abs{x_z-\zeta_z(t_z)}.
		\end{equation}
		We can compute that the right hand side satisfies the required bound, and so does $\dd v_z$ as well.
		In the transition region, we can similarly compute
		\begin{equation}
			(1+\A^{0,\nu-1,\nu-1}(\Mcomp))\dd v_z=\dd t_z+h' \dd\abs{x_z-\zeta_z(t_z)}+\chi'\frac{h}{\lambda_z}\left(\frac{\dd \abs{x_z}}{t_z}-\frac{\abs{x_z}\dd t_z}{t_z^2}\right) 
		\end{equation}
		Setting $\zeta_z=0$ for a moment, we compute
		\begin{multline}
			\norm{\left(1-\chi' (\abs{x}/t)^2\right)\dd t+(\chi+\chi' \abs{x}/t)\dd \abs{x}}_\eta\\
			=-1+\chi^2+2\chi' \left(\abs{x}^2/t^2+\chi \abs{x}/t\right)+\chi'{}^2\abs{x}^2/t^2(1-\abs{x}^2/t^2)\leq0
		\end{multline}
		where we used that $\chi'\leq0$ and the previous smallness on $\sup_\rho \rho\chi'(\rho)$ to get the last inequality.
		For $\zeta_z\in\A^{\nu}(\interval)$, the extra terms are bounded by $\A^{\nu-1}$, as claimed.
	\end{proof}
	
	Let us write $\tilde{T}=\eta^{-1}(\dd  s,\cdot)$, $\tilde{T}_*=\eta^{-1}(\dd  s_*,\cdot)$ for the vector field corresponding $ s$, $ s_*$ respectively.
	Using $ s^q\tilde{T}$ as a multiplier we obtain control over $\phi$ up to terms localised to the solitons.
	\begin{lemma}[$T$ energy estimate]\label{lin:lemma:T-energy}
		Let $\phi$ be a solution of \cref{lin:eq:main_cutoffed}.
		There exists $q_0$ depending on $ s$ such that the following holds.
		For $q>q_0$ the currents $\tilde{J}^T_*= s_*^{-q}\tilde{T}\cdot\T^{5\phi_a^4}[\phi]$, $\tilde{J}^T= s^{-q}\tilde{T}\cdot\T^{5\phi_a^4}[\phi]$ satisfy the following divergence estimates 
		\begin{subequations}
			\begin{align}
				&\div \tilde{J}^T+\abs{f\tilde{T}\phi} s^{-q}\geq qc \abs{\partial\phi}^2 s^{-q-1}+q\phi^2 s^{-q}\A^{0,2\nu-5,-2\nu-1}(\Mcomp)\label{lin:eq:div_JT}\\
				&\div \tilde{J}^T_*+\abs{f\tilde{T}\phi} s^{-q}\geq qs_*^{-q-1}\T^{5W_z^4}(\tilde{T}_*\otimes\tilde{T})+\phi^2\A^{2\nu-5,-2\nu-1}+\A^{-q-1,\infty}\abs{\partial\phi}^2 \quad \text{in }\{\abs{x_z}/t_z<\delta\}\label{lin:eq:div_JTstar},
			\end{align}
		\end{subequations}
		where we also have
		\begin{equation}\label{lin:eq:Tstar_flux}
			\gamma_z\T^{5W_z^4}(\tilde{T}_*\otimes\tilde{T})=(1-h'{}^2)(T_z\phi)^2+\partial_{x_z}|_{v_z}\phi\cdot\partial_{x_z}|_{v_z}\phi-5W_z^4\phi^2+(\abs{\partial \phi}^2,\phi^2 W_z^4)\A^{\nu-1,\nu-1}(\Mcomp).
		\end{equation}
	\end{lemma}
	\begin{proof}
		We begin by noting the following standard property of the energy momentum tensor:
		\begin{equation}
			\T^0[\phi](X\otimes X)\sim_X \abs{\partial\phi}^2,\qquad \forall \eta(X,X)<0.
		\end{equation}
		
		\emph{Step 1: Exterior.} We begin by computing in the region $\D'=\{\abs{x}/t>\delta\}\cap\{1-\abs{x}/t>\delta\} $.
		We use  $\div \T^{5\phi_a^4}=f\nabla\phi+\phi^2\nabla(5\phi_a^4)$ to compute 
		\begin{equation}\label{lin:eq:div_energy}
			\div \tilde{J}^T= s^{-q}\left( \tilde{T}\cdot\div \T^{5\phi_a^4}+q  s^{-1} \T^{5\phi_a^4}\cdot(\tilde{T}\otimes\tilde{T})+  \pi^{\tilde{T}}\cdot \T^{5\phi_a^4}\right),
		\end{equation}
		where $\pi^{\tilde{T}}_{\mu\nu}=\nabla_{(\nu} \tilde{T}_{\mu)}$ is the deformation tensor of $\tilde{T}$.
		In the exterior, $\D'$, we may bound each term as follows
		\begin{nalign}
			&\abs{\tilde{T}\cdot \div\T^{5\phi_a^4}}=\abs{f\tilde{T}\phi}+5\abs{\phi^2\tilde{T}\phi_a^4}\leq \abs{f\tilde{T}\phi}+ \A^{0,2\nu-5}(\D')\phi^;2\\
			&\abs{\pi^{\tilde{T}}\cdot \T^{5\phi_a^4}}\leq \A^{\infty,-1}(\D')\left(\abs{\partial\phi}^2+\phi^2\A^{0,2\nu-4}(\D')\right);\\
			&\T^{5\phi_a^4}(\tilde{T}\otimes \tilde{T})\sim_{ s, s'} \abs{\partial\phi}^2+\A^{0,2\nu-4}\abs{\phi}^2.
		\end{nalign}
		Observe, that for $q$ sufficiently large depending on $ s$ the control provided by the last term dominates all other $\partial\phi$ terms.
		
		\emph{Step 2: Interior}
		We work in $\D'=\{\abs{x}/t<\delta\}$.
		We still use \cref{lin:eq:div_energy} to obtain 
		\begin{nalign}
		\abs{\tilde{T}\cdot \div\T^{5\phi_a^4}}\leq \abs{f\tilde{T}\phi}+ \A^{2\nu-5,-2\nu-1}(\D')\phi^2;\\
		\abs{\pi^{\tilde{T}}\cdot \T^{5\phi_a^4}}\leq \A^{-1,\infty}(\D')\left(\abs{\partial\phi}^2+\phi^2\A^{2\nu-4,0}(\D')\right);\\
		\T^{5\phi_a^4}(\tilde{T}\otimes \tilde{T})=\T^{5W_0^4}(\tilde{T}_*\otimes \tilde{T})+\phi^2\A^{2\nu-4,-\nu-1}(\D').
		\end{nalign}
		
		\emph{Step 3: $\tilde{T}_*$}
		The computation for $\tilde{T}_*$ follows analogously with the expression \cref{lin:eq:Tstar_flux} following from the computation in \cref{sec:fluxes} together with \cref{lin:eq:dv}, see \cref{app:eq:en_mom_scaled}.
	\end{proof}

	In \cref{lin:sec:L2} we use the divergence theorem for the current $\tilde{J}^T$ applied to functions $\Ve\phi\in\Hb^{0;0,(q-1)/2,(q-\nu)/2}(\Mcomp)$.
	We already now remark, that via Hardy inequality, see \cref{pre:lemma:Hardy}
	\begin{equation}
		\int_{\D_{ext,\delta}}\abs{\Ve\phi}^2\rho_-^{1-q}\mu_{\b}\gtrsim_q\int_{\D_{ext,\delta}}\abs{\phi}^2\rho_-^{1-q}\mu_{\b}
	\end{equation}
	The coercivity of $\div \tilde{J}^T$ already follows in the exterior region: fix $c>0$ and some $\mathfrak{f}\in \A^{0,2\nu-5,-2\nu-1}(\Mcomp)$, then for $\tau_0(\mathfrak{f})$ sufficiently small 
	\begin{equation}
		\int_{\D_{ext,\delta}\cap\{\tau<\tau_0\}} c \abs{\partial\phi}^2\tau^{-q-1}+\phi^2\tau^{-q}\mathfrak{f}\mu \gtrsim \int_{\D_{ext,\delta}\cap\{\tau<\tau_0\}}  \abs{\Ve\phi}^2\tau^{-q+1} \mu_{\b}.
	\end{equation}
	The rest of \cref{sec:lin} is about controlling $\phi$ also in $\D_z$ using $\tilde{J}^T_*$.
	
	\subsection{Localised appproximate solution}\label{lin:sec:localised_app}
	In this section, we localise the approximate solutions $\phi_{A}$ of \cref{an:prop:peeling}.
	In particular, we want to construct $\phi_z$ for $z\in A$ such that $\phi_z\in \A^{\nu/2-1,-\nu/2}(\Mcomp_z)$ and moreover it approximates $\phi_{A}$ sufficiently well in a neighbourhood of the $z$-th soliton, and solves $P$ to a sufficiently high order.
	Such approximate solutions will be used for the consturction of almost conserved currents that we use as projectors to obtain coercivity of the $\tilde{J}^T_*$ energy.
	
	\begin{lemma}\label{lin:lemma:localised_app}
		Let $\phi_{A}$ be an approximate solution as constructed in \cref{an:prop:peeling}. Then, for each $z\in A$ there exists an approximate solution $\phi_z=\phi_{z,rad}+\bar{\phi}_{z}\in  \A^{0,\nu/2-1,-\nu/2}(\Mcomp_z)$ such that
		\begin{nalign}
			&\phi_z-\phi_{A}\in \A^{\nu/2-1,5\nu/2-3}(\Mcomp_z\cap \D_z),&&P[\phi_z]\in \Aext^{5\nu/2-5,\nu/2-3}(\Mcomp_z)+\A^{\nu/2+1,5\nu/2-5,\nu/2-3}(\Mcomp_z),\\
			&\phi_{z,rad}\in \A^{0,\nu/2-1}(\Mcomp_\emptyset),&& \bar{\phi}_z-W_z\in\A^{0,5\nu/2-3,\nu/2-1}(\Mcomp_z)\\
			&\Box\phi_{z,rad}\in\A^{0,5(\nu/2-1)-2}(\Mcomp_\emptyset).
		\end{nalign}
	\end{lemma}
	\begin{proof}
		We consider the case $z=0$, the other following by symmetry.
		
		The approximate solution $\phi_{A}$ from \cref{map:def:approximate} is given as
		\begin{equation}
			\phi_{A}=\phi_{ext}+\sum_z W_z+\phi_{1,z}\in\A^{0,\nu/2-1,-\nu/2}(\Mcomp_A)
		\end{equation}
		where we localised $\phi_1$ of \cref{map:def:approximate} into an exterior part with $\phi_{ext}\in\A^{0,\nu/2-1}(\Mcomp_\emptyset)$ and localised parts $\supp\phi_{1,z}\subset \D_z$, $\phi_{1,z}\in \A^{3\nu/2-2,\nu/2-1}(\D_z)$.
		We replace $W_{z}$ for $z\neq 0$ with $\bar{\phi}_{z}$ satisfying $\Box\bar\phi_{z}=0$ and $W_{z}-\bar\phi_{z}\in\A^{\nu/2-1,5\nu/2-3}(\Mcomp_0\cap \D_0)$.
		Provided that such a choice of $\bar\phi_z$ exists, we write
		\begin{equation}
			\phi_{z=0}=\phi_{ext}+W_0+\phi_{1,z=0}+\sum_{z\neq0}\bar{\phi}_z=\phi_{A}+\underbrace{\sum_{z\neq0}-\phi_{1,z}+\bar{\phi}_z-W_z}_{\mathfrak{Err}},
		\end{equation}
		where $\mathfrak{Err}\in\A^{\nu/2-1,5\nu/2-3}(\Mcomp_0\cap\D_0)$.
		Thus we compute that in $\D_0$
		\begin{equation}
			P[\phi_z]=P[\phi_{A}]+(\phi_{A}+\mathfrak{Err})^5-\phi_{A}^5=\A^{\infty,\infty}(\D_0)+\A^{5\nu/2-5,\nu/2-3}(\D_0).
		\end{equation}
		In the exterior region, the error term comes from all the nonlinear terms.
		
		Let us write $g^{(1)}(t)=W_z(0,t)$ and  $g_i^{(2)}(t)=\partial_{x^i}W_z(0,t)$.
		We also consider the following two solutions to the linear wave equation
		\begin{nalign}
			\Box \phi^{(i)}&=0,\qquad \phi^{(1)}_z=r^{-1}\left(G_{(1)}(t+r)-G_{(1)}(\mathring{u})\right),\\
			\phi^{(2)}_{z,i}&=\hat{x}_i r^{-1}\left(G_{(2)}'(t+r)+G_{(2)}'(\mathring{u})-2r^{-1}\left(G_{(2)}(t+r)-G_{(2)}(\mathring{u})\right)\right).
		\end{nalign}
		We may evaluate
		\begin{equation}
			\phi^{(1)}_z|_{r=0}=2\partial_tG^{(1)}(t),\quad \partial_{x_i}\phi^{(2)}_{z,i}=\frac{1}{3}\partial_t^3G^{(2)}(t).
		\end{equation}
		Setting
		\begin{equation}
			2G^{(1)}=\int_0^t\dd t' g^{(1)}(t'),\qquad G^{(2)}=3\int_0^t\dd t' \frac{(t-t')^2}{2} g^{(2)}(t')
		\end{equation}
		yields the claimed $\bar{\phi}_z=\phi^{(1)}_z+\sum_i \phi^{(2)}_{z,i}$.
	\end{proof}
	
	We also note that solutions to $\Box$ also yield good approximate solutions with the same decay rates:
	\begin{lemma}\label{lin:lemma:localised_trivial}
		Let $\phi_{rad}\in\A^{0,\nu/2-1}(\Mcomp_{\emptyset})$ with $\Box\phi_{rad}\in\A^{0,5(\nu/2-1)-2}(\Mcomp_{\emptyset})$.
		Then $P[\phi_{rad}]\in\A^{0,5\nu/2-5}(\Mcomp_\emptyset)$.
	\end{lemma}
	
	\subsection{Conservation laws}\label{lin:sec:conservation}
	The aim of this section is to find current $J^\bullet$, such that their integrals along constant $\Sigma_{z,\tau}=\{v_z=\tau\}$ hypersurfaces gives a leading order projection to the kernel elements that ruin the coercivity of \cref{lin:eq:Tstar_flux} integrated along $\Sigma_{z,\tau}$.
	These currents are bilinear in $\phi$ the solution of \cref{lin:eq:main_cutoffed} and $\phi_z$ the localised approximate solution of \cref{lin:lemma:localised_app}.
	In particular, we define $3\abs{A}$ different currents corresponding to the localised size, center of mass and momentum around each soliton.
	As the bilinear forms are made out of $\phi_z$ instead of $\phi_A$, there will be a leading order (in terms of decay) error term around each of the solitons corresponding to the coupling $\mathfrak{M}$ appearing in \cref{mod:eq:scale_precise}.
	This will be another reason we need to take the decay rate $N$ of the right hand side of \cref{lin:eq:main_cutoffed} sufficiently large.
	
	In this section, we closely follow the conservation laws introduced in \cite{kadar_scattering_2024} and used in \cite{kadar_construction_2024}.
	Let us define the linearised energy momentum tensor 
	\begin{equation}
		\Tlin[\phi;\phi_A]=\dd \phi\otimes \dd\phi_A+\frac{\eta}{2}(\eta^{-1}(\dd\phi,\dd\phi_A)-\phi_A^5\phi)
	\end{equation} 
	and the corresponding currents
	\begin{subequations}\label{lin:eq:currents}
		\begin{align}
			J^{\Lambda,z}[\phi,\phi_A]&:=\partial_{t_z}\cdot \Tlin[\phi;\phi_A]\\
			J^{m,z}_i[\phi,\phi_A]&:=\partial_{x^i_z}\cdot  \Tlin[\phi;\phi_A]\\
			J^{c,z}_i[\phi,\phi_A]&:=(t_z\partial_{x^i_z}+x^i_z\partial_{t_z})\cdot  \Tlin[\phi;\phi_A],
		\end{align}
	\end{subequations}
	where the superscript stand for scaling, momentum and center of mass.
	We will always drop the subscript $i$ from $J^m,J^c$ whenever we treat them as vectorial quantities.
	We note that $\Tlin[\phi;\phi_A]$ is the same as the bilinear energy momentum tensor associated to $\T^{\phi_A^4}[\phi]$ and not to $\T^{5\phi_A^4}[\phi]$ as used in \cref{lin:lemma:T-energy}.

	In the following we compute their divergence and flux.
	\begin{lemma}[Divergence]\label{lin:lemma:div}
		Let $\phi_z$ be the localised approximate solutions of \cref{lin:lemma:div} and $J^\bullet$ the currents of \cref{lin:eq:currents}.
		Let $\phi$ be a solution of \cref{lin:eq:main_cutoffed}.
		We have the following divergence estimates:
		\begin{subequations}\label{lin:eq:divergences}
			\begin{align}
				\div J^{\Lambda,z}[\phi,\phi_z]&=f\A^{\nu/2-2,-\nu/2-1}(\Mcomp)+\partial_{t_z}\phi \A^{5\nu/2-5,\nu/2-3}(\Mcomp)+\phi\A^{5\nu/2-6,-3\nu/2-2}(\Mcomp)\label{lin:eq:divergencesLambda},\\
				\div J_i^{m,z}[\phi,\phi_z]&=f\A^{\nu/2-2,-3\nu/2}(\Mcomp)+\partial_{x^i}\phi \A^{5\nu/2-5,\nu/2-3}(\Mcomp)+\phi\A^{5\nu/2-6,-3\nu/2-2}(\Mcomp),\\
				\div J_i^{c,z}[\phi,\phi_z]&=f\A^{\nu/2-1,-3\nu/2+1}(\Mcomp)+t\partial_{x^i}\phi \A^{5\nu/2-5,\nu/2-3}(\Mcomp)+\phi\A^{5\nu/2-5,-3\nu/2-1}(\Mcomp),\\
				\div J^{c,z}[\phi,\phi_{z,rad}]&=f\A^{\nu/2-1,-3\nu/2+1}(\Mcomp)+t\partial_{x^i}\phi \A^{5\nu/2-5,\nu/2-3}(\Mcomp)+\phi\A^{5\nu/2-5,-3\nu/2-1}(\Mcomp),
			\end{align}
			where the index at $\C$ is always 0.
		\end{subequations}
	\end{lemma}

	\begin{proof}
		We perform all the computations at $0\in A$, the others following similarly.
		We start with the identity
		\begin{multline}
			\div J^\Lambda=\partial_t\phi_0(\Box+5\phi_0^4)\phi+\partial_t\phi P[\phi_0]=\partial_t\phi_0 \left(f+5(\phi_0^4-\phi_{A}^4)\phi\right)+\partial_t\phi \A^{5\nu/2-5,2\nu-3}(\Mcomp_0)\\
			=\A^{\nu/2-2,-\nu/2-1}(\Mcomp_0) f+\A^{5\nu/2-6,-\nu/2-3}(\Mcomp_0)\phi+\A^{5\nu/2-6,-3\nu/2-2}(\Mcomp)\phi+\partial_t\phi \A^{5\nu/2-5,\nu/2-3}(\Mcomp_0),
		\end{multline}
		where the bounds all follow from \cref{lin:lemma:localised_app}.
		
		For the others, we have similarly
		\begin{multline}
			\div J^m=\partial_{x^i}\phi_0(\Box+5\phi_0^4)\phi+\partial_{x^i}\phi P[\phi_0]\\
			=\A^{\nu/2-2,-3\nu/2}(\Mcomp_0)f+\A^{5\nu/2-6,-3\nu/2-2}(\Mcomp_0)\phi+\A^{5\nu/2-6,-3\nu/2-2}(\Mcomp)\phi+\partial_{x^i}\phi\A^{5\nu/2-5,\nu/2-3}(\Mcomp)
		\end{multline}
		and 
		\begin{multline}
				\div J^c=(t\partial_{x^i}+x^i\partial_t)\phi_0(\Box+5\phi_0^4)\phi+(t\partial_{x^i}+x^i\partial_t)\phi P[\phi_0]\\
				=\A^{\nu/2-1,-3\nu/2+1}(\Mcomp_0)f+\A^{5\nu/2-5,-3\nu/2-1}(\Mcomp_0)\phi+\A^{5\nu/2-5,-3\nu/2-1}(\Mcomp)\phi+(t\partial_{x^i}+x^i\partial_t)\phi\A^{5\nu/2-5,\nu/2-3}(\Mcomp)
		\end{multline}
		
	\end{proof}
	Integrating the divergence corresponding to $J^\Lambda$, we find that we can bound them using a coercive energy as
	\begin{multline}\label{lin:eq:JLambda_bound_rough}
		\int \phi \A^{5\nu/2-6,-3\nu/2-2}\lesssim\int \dd t\dd\boldsymbol{x}^3 t^{3\nu}t^{-3\nu/2-2}\phi\jpns{\boldsymbol{x}}^{-4}\lesssim\int \dd \tau \tau^{3\nu/2-2}\left(\int_{\Sigma_\tau} \dd\boldsymbol{x}^3 \phi^2\jpns{\boldsymbol{x}}^{-2}\right)^{1/2}\\
		\lesssim\int \dd \tau \tau^{\nu-2} \left(\int_{\Sigma_\tau} \dd \boldsymbol{x}^3 t^{\nu}\phi^2\jpns{\boldsymbol{x}}^{-2}\right)^{1/2}
	\end{multline}
	The importance of the numerology in this equation will be further explained below \cref{lin:lemma:flux}.
	
	Next, we compute their fluxes on constant $v$ hypersurfaces.
	\begin{lemma}[Fluxes]\label{lin:lemma:flux}
		Let $\phi_z$ be the localised approximate solutions of \cref{lin:lemma:div} and $J^\bullet$ the currents of \cref{lin:eq:currents}.
		Let $\phi|_\C=0$.
		Then, on $\Sigma=\{v_z=\tau\}$  using coordinates $\boldsymbol{x}$ as in \cref{not:eq:vy_def} it holds that
		
		\begin{subequations}\label{lin:eq:fluxes_bounds}
			\begin{align}
				\int_{\Sigma} J^\Lambda[\phi;\phi_z]\lesssim t^{\nu-1}\left(\int_\Sigma \jpns{\boldsymbol{x}}^{-1/2}T_z\cdot \T^0[\phi] \right)^{1/2}+t^{2(\nu-1)} \left(\int_\Sigma T_z\cdot \abs{\partial\phi}^2\dd x^3 \right)^{1/2}, \\
				\int_{\Sigma} J^m[\phi,\phi_z]\lesssim t^{\nu-1}\left(\int_\Sigma \jpns{\boldsymbol{x}}^{-1/2}T_z\cdot \T^0[\phi] \right)^{1/2}+t^{2(\nu-1)} \left(\int_\Sigma T_z\cdot \abs{\partial\phi}^2\dd x^3 \right)^{1/2},\\
				\int_{\Sigma} J^c[\phi,\phi_z]-J^c[\phi,\phi_{z,rad}]\lesssim t^{\nu}\left(\int_\Sigma \jpns{\boldsymbol{x}}^{-1/2}T_z\cdot \T^0[\phi] \right)^{1/2}+t^{2\nu-1} \left(\int_\Sigma T_z\cdot \abs{\partial\phi}^2\dd x^3 \right)^{1/2} .
			\end{align}
		\end{subequations}
		as well as
		\begin{subequations}\label{lin:eq:flux_contracted}
			\begin{align}
				\int_{\Sigma} J^\Lambda[\Lambda W(y)]\sim t^{\frac{3}{2}\nu-1},\qquad \int_{\Sigma} J^\Lambda[\partial_y W(y)]\in \A^{5\nu/2-2}(\interval),\\
				\int_{\Sigma} J^c[\partial_y W(y)]\sim t^{\frac{3}{2}\nu},\qquad \int_{\Sigma} J^c[\Lambda W(y)]\in \A^{\frac{5}{2}\nu-1}(\interval)\\
				\int_\Sigma J^{m}[\partial_y W(y)]\in \A^{\frac{5}{2}\nu-2}(\interval),\qquad\int_\Sigma J^{m}[\Lambda W(y)]\in \A^{\frac{5}{2}\nu-2}(\interval).
			\end{align}
		\end{subequations}
	\end{lemma}
	\begin{proof}
		Let us work around the soliton $0\in A$, for the rest following similarly.
		
		\emph{Scaling:}
		We recall the standard flux computation from \cref{app:en:energy:fg} with $\zeta=0$
		\begin{multline}
			\tau^{-3\nu}\int_{\Sigma} J^\Lambda[\phi]=\int_{\Sigma}  (1-h'^2)\partial_t\phi\partial_t\phi_0+\tau^{-2\nu}\nabla_{\boldsymbol{x}}|_\Sigma\phi_0\nabla_{{\boldsymbol{x}}}|_\Sigma\phi-\phi_0^5\phi\dd \boldsymbol{x}^3\\
			=\int_{\Sigma}  (1-h'^2)\partial_t\phi\partial_t\phi_0-\phi(\phi_0^5-\tau^{-2\nu}\Delta\phi_0)\dd \boldsymbol{x}^3=\int_{\Sigma}  (1-h'^2)\partial_t\phi\partial_t\phi_0-\phi L\partial_{v}\Lambda \phi_{A}+\phi\A^{0,5\nu/2-5,-2}(\Mcomp_0)\dd \boldsymbol{x}^3\\
			=\int_{\Sigma}  (1-h'^2)\partial_t\phi\partial_t\phi_0-\lambda^{1/2} \dot\lambda \phi L_1\Lambda W(\boldsymbol{x})+\phi\A^{0,5\nu/2-5,-2}(\Mcomp_0)\dd \boldsymbol{x}^3,
		\end{multline}
		where $L$ is the operator from \cref{not:eq:frozen_Box}.
		
		Let us recall that the divergence of the coercive energy is of the form 	\begin{equation}\label{lin:eq:coercive_flux}
			\gamma_z\T^{0}(\tilde{T}_*\otimes\tilde{T})=(1-h'{}^2)(T_z\phi)^2+\partial_{x_z}|_{v_z}\phi\cdot\partial_{x_z}|_{v_z}\phi.
		\end{equation}
		Using Hardy inequality, \cref{pre:lemma:Hardy}, the integrated version also controls $\tau^{-2\nu}\jpns{\boldsymbol{x}}^{-2}\phi^2$.
		We use these to bound each term in $J^\Lambda$ separately as follows
		\begin{nalign}
			&\begin{multlined}
				\tau^{3\nu}(1-h'{}^2)\partial_t\phi\partial_t\phi_0=\tau^{3\nu}(1-h'{}^2)\partial_t\phi\A^{0,\nu/2-2,-\nu/2-1}(\Mcomp_0)\\
				\leq \left(\tau^{3\nu}(1-h'{}^2)(\partial\phi)^2\right)^{1/2}\tau^{\nu-1}\A^{0,\nu-1,0}(\Mcomp_0);
			\end{multlined}\\
			&\begin{multlined}
				\tau^{3\nu}\phi \lambda^{1/2}\dot{\lambda}\phi L\Lambda W=\tau^{2\nu}\phi \A^{0,4(\nu-1)-\nu/2-1,-\nu/2-1}(\Mcomp)\\ \lesssim (\tau^{3\nu-2\nu}\jpns{\boldsymbol{x}}^{-2}\phi^2)^{1/2}\tau^{\nu-1}\A^{0,3(\nu-1),0}(\Mcomp);
			\end{multlined}\\
			&\tau^{3\nu}\phi \A^{5\nu/2-5,-2}(\Mcomp)\leq (\tau^{\nu}\jpns{\boldsymbol{x}}^{-2}\phi^2)^{1/2} \tau^{\nu-1}\A^{3\nu-3,3\nu/2-1}(\Mcomp).
		\end{nalign}
		Integrating these along $\Sigma$ and using \cref{lin:eq:coercive_flux} yields the first line of \cref{lin:eq:fluxes_bounds}. For instance we have
		\begin{equation}
			\int_\Sigma (\tau^{3\nu-2\nu}\jpns{\boldsymbol{x}}^{-2}\phi^2)^{1/2}\tau^{\nu-1}\A^{0,3(\nu-1),0}\lesssim \left(\int_\Sigma \tau^{3\nu-2\nu}\jpns{\boldsymbol{x}}^{-2.5}\phi^2\dd \boldsymbol{x}^3 \right)^{1/2}\left(\int_\Sigma \tau^{\nu-1}\A^{0,5.5(\nu-1),0} \dd \boldsymbol{x}^3\right),
		\end{equation}
		
		For general $\zeta\in \A^{\nu}$, we have an additional term that takes the form
		\begin{multline}
			\int_\Sigma \Tlin[\phi ,\phi_0](\dd t,\dd t)\dot{\zeta}\tau^{3\nu}\dd\boldsymbol{x}^3= \int_{\Sigma} \dot{\zeta}( \partial_t\phi\partial_t\phi_0+\tau^{-2\nu}\nabla_{\boldsymbol{x}}|_\Sigma\phi_0\nabla_{{\boldsymbol{x}}}|_\Sigma\phi-\phi_0^5\phi)\dd \boldsymbol{x}^3\\
			=\int_{\Sigma} \dot{\zeta}\partial_t\phi\partial_t\phi_0-\phi(\dot{\zeta}\phi_0^5-\tau^{-2\nu}\nabla_{\boldsymbol{x}}\dot{\zeta}\nabla_{\boldsymbol{x}}\phi_0)\dd \boldsymbol{x}^3
			\lesssim \tau^{2(\nu-1)}\left(\int_\Sigma \T^0[\phi]\right)^{1/2}
		\end{multline}
		For $J^c,J^m$ we will ignore the parts depending on $\dot{\zeta}$ as bounding them follows the same line of argument.
		
		Next, we compute the projection onto $\Lambda W$ and $\partial_y W$.
		Starting with $\Lambda W$, it holds that
		\begin{nalign}
			\int_\Sigma (1-h'{}^2)\partial_t\Lambda W(\boldsymbol{x})\partial \phi_0\dd \boldsymbol{x}^3\in \A^{-\nu/2-2}(\interval)\\
			\tau^{-3\nu/2-1}\int_\Sigma\Lambda W(\boldsymbol{x})L_1\Lambda W(\boldsymbol{x})\dd \boldsymbol{x}^3=\tau^{-3\nu/2-1}(\pi+\A^{2}(\interval))\\
			\int_\Sigma \Lambda W(\boldsymbol{x})\A^{0,5\nu/2-5,-2}(\Mcomp)\dd \boldsymbol{x}^3\in \A^{\min(\nu/2-3,2)}(\interval).
		\end{nalign}
		This already yields the result for $\Lambda W$.
		For $\partial_y W$ the only change is  to observe the leading order cancellation
		\begin{equation}
			\tau^{-3\nu/2-1}\int_\Sigma\partial_yW L_1\Lambda W\dd \boldsymbol{x}^3=0.
		\end{equation}
		
		\emph{Momentum:}
		Similarly, we can use \cref{app:eq:mom} to obtain
		\begin{multline}
			\tau^{-3\nu}\int_\Sigma J^m=\int_\Sigma \hat{x}(1-h'{}^2)h'T\phi T\phi_0+(1-h'{}^2)(T\phi_0\nabla|_{\Sigma}\phi+T\phi\nabla|_{\Sigma}\phi_0)+h'(\hat{x}\cdot\nabla|_{\Sigma}\phi_0\nabla|_\Sigma\phi\\
			+\hat{x}\cdot\nabla|_{\Sigma}\phi\nabla|_\Sigma\phi_0)
			-\hat{x}h'(\nabla|_{\Sigma}\phi\cdot\nabla|_{\Sigma}\phi_0-\phi\phi_0^5)\dd \boldsymbol{x}^3.
		\end{multline}
		The terms are all bounded the same way as for $J^\Lambda$.
		
		To compute the $\Lambda W$ contribution, we note that all the leading order terms vanish, due to $\Lambda W$ being supported on $\ell=0$ spherical harmonic, thus we only get for instance
		\begin{equation}
			\tau^{3\nu}\int_{\Sigma}(1-h'{}^2)T\phi_0\nabla|_\Sigma \Lambda W\dd\boldsymbol{x}^3=\tau^{2\nu}\int_{\Sigma}(1-h'{}^2)T\phi_0\nabla_{\boldsymbol{x}}|_\Sigma \Lambda W\dd\boldsymbol{x}^3=\A^{2\nu+\nu/2-2}(\interval).
		\end{equation}
		For $\partial_y W$ we observe the following explicit cancellation at top order with no summation over the index $i$
		\begin{multline}
			\int_{\Sigma\cap\{\abs{y}=c\}}\hat{x}_i\cdot\nabla|_{\Sigma}W\nabla|_\Sigma\partial_i W+\cancel{\hat{x}_i\cdot\nabla|_{\Sigma}\partial_i W\nabla|_\Sigma W}-\cancel{\hat{x}_i(\nabla|_{\Sigma}W\cdot \nabla|_\Sigma\partial_i W)}+\hat{x}_iW^5\partial_i W\dd\boldsymbol{x}^3\\
			=\frac{1}{4\pi}\int_\Sigma W'\frac{\Delta W}{3}+W^5\frac{W'}{3}\dd \boldsymbol{x}^3=0.
		\end{multline}

		\emph{Center of mass:}
		Finally, we use \cref{app:eq:com} to get
		\begin{multline}\label{lin:eq:proof_com_int}
			\tau^{-3\nu }\int_\Sigma J^c=\tau^{1-3\nu}\int_\Sigma J^m+\int_\Sigma \frac{1}{2}T\phi T\phi_0\hat{x}(\frac{h}{\lambda}h'+r)(1-h'{}^2)-(1-h'{}^2)\frac{h}{\lambda} (T\phi \nabla|_\Sigma\phi_0+T\phi_0 \nabla|_\Sigma\phi)\\
			-\frac{h}{\lambda}h'(\nabla|_{\Sigma}\phi \hat{x}\cdot\nabla|_{\Sigma}\phi_0+\nabla|_{\Sigma}\phi_0 \hat{x}\cdot\nabla|_{\Sigma}\phi)+(h'\frac{h}{\lambda}-r)\hat{x}\nabla|_{\Sigma}\phi\cdot \nabla|_{\Sigma}\phi_0+\phi_0^5\phi(x+\hat{x}\frac{h}{\lambda}h').
		\end{multline}
		A similar computation applies to $J^c[\phi,\phi_{0,rad}]$.
		Note that we cannot bound the first term in the second integral by the same energy quantity as the others, for instance
		\begin{nalign}
		\tau^{3\nu}x	\phi_0^5\phi=\phi\A^{0,3\nu/2+4(\nu-1),3\nu/2}=(\phi^2\tau^{\nu}\jpns{\boldsymbol{x}}^{-2})^{1/2}\tau^\nu\A^{0,3(\nu-1),0},\\
		\tau^{3\nu}\lambda^{-1}hh' \nabla|_{\Sigma}\phi \hat{x}\cdot\nabla|_{\Sigma}\phi_0=\nabla|_{\Sigma}\phi \A^{0,7\nu/2-1,5\nu/2}\leq (\tau^{3\nu}\abs{\partial\phi}^2)^{1/2}\tau^\nu\A^{0,\nu-1,0}.
		\end{nalign}
		When integrating this expression along $\Sigma$ and using Cauchy-Schwarz, we would need to bound
		\begin{equation}
			 \int_\Sigma (\tau^{3\nu}\abs{\partial\phi}^2)^{1/2}\A^{0,\nu-1,0}\dd\boldsymbol{x}^3\lesssim  \left(\int_\Sigma \tau^{3\nu}\abs{\partial\phi}^2\dd\boldsymbol{x}^3\right)^{1/2}\left(\int_\Sigma\boldsymbol{x}^{-2}\dd\boldsymbol{x}^3\right)^{1/2},
		\end{equation}
		with the second term producing a growing contribution as $\tau\to0$.
		It is clear that the difficulty comes from the slow decay of $\phi_0$ in the self similar region.
		Therefore, we split $\phi_0$ into $W_0$ and a \emph{radiation} term $\phi_{0,rad}$ as in \cref{lin:lemma:localised_app} solving $\Box\phi_{0,rad}\in\A^{0,5(\nu/2-1)-2}(\Mcomp_{\emptyset})$, with the first satisfying a cancellation and the second bounded by $J^{c,z}[\phi,\phi_{0,rad}]$.		

		Let us consider only the worst in $W_0$ part and compute
		\begin{multline}\label{lin:eq:proof_com_bound}
			-\tau^{3\nu}\int_\Sigma\frac{h}{\lambda}h'(\nabla_i|_{\Sigma}\phi \hat{x}\cdot\nabla|_{\Sigma}W_0+\nabla_i|_{\Sigma}W_0 \hat{x}\cdot\nabla|_{\Sigma}\phi)\\
			=\tau^{3\nu}\int_\Sigma \phi\lambda^{-1} \left(\nabla_{i}|_{\Sigma}(\hat{x}_jhh'\nabla_j|_{\Sigma} W_0)+\nabla_{j}|_{\Sigma}(\hat{x}_jhh'\nabla_i|_\Sigma W_0)\right)\\
			=\tau^{3\nu}\int_\Sigma \phi\lambda^{1} \left(\nabla_{\boldsymbol{x}_i}|_{\Sigma}(\boldsymbol{x}_j\nabla_{\boldsymbol{x}_j}|_{\Sigma} W_0)+\nabla_{\boldsymbol{x}_j}|_{\Sigma}(\boldsymbol{x}_j\nabla_{\boldsymbol{x}_i}|_\Sigma W_0)\right)
			+\int_\Sigma\phi \A^{0,3\nu/2+3(\nu-1),3\nu/2},
		\end{multline}
		where we used that $hh'=\abs{\boldsymbol{x}}+\R$ for $\abs{\boldsymbol{x}}$ sufficiently large.
		We compute for the unscaled ground state $W$ that
		\begin{multline}
			\partial_i \left(x_j\partial_j W\right)+\partial_j \left(x_j\partial_i W\right)=4\partial_i W+2x_j\partial_j\partial_i W=4\hat{x}_i W+2x_j\left(r^{-1}\delta_{ij}-\hat{x}_i\hat{x}_j\right)W'+2x_i\hat{x}_j\hat{x}_jW''\\
			=2\hat{x}_i(2W'+rW'')=0.
		\end{multline}
		The same computation for the rescaled $W_0$ implies that
		\begin{equation}
			\cref{lin:eq:proof_com_bound}\lesssim \int\phi\tau^{3\nu/2}\A^{0,3(\nu-1),0}.
		\end{equation}
		Thus follows the bound in \cref{lin:eq:fluxes_bounds}.
		
		To compute the explicit projections, we use \cref{lin:eq:proof_com_int}.
		For $\partial W$, we recall from \cite[(5.26),(5.27)]{kadar_construction_2024}
		\begin{equation}
			\int_{\R^3} h(x)h'(x)\left(\partial_i^2\partial_r W+\partial_i W\partial_r\partial_i W \right)+(hh'-r)\hat{x}_i\partial_j W\partial_{ji}W+W^5\partial_iW (r+hh')\hat{x}_i\neq0
		\end{equation}
		and as all other terms in \cref{lin:eq:proof_com_int} are lower order in $\tau$ this yields the result.
		For $\Lambda W$, we use that the leading order term vanishes due to spherical symmetry.		
	\end{proof}

	In \cref{lin:sec:coercivity} we show that \cref{lin:eq:fluxes_bounds} is not merely an upper bound, but we may use $\tau^{1-\nu}J^\bullet$ and the non-coercive $\tilde{J}^T_*$ to control $T\cdot\T^0[\phi]$ current.
	In particular, combining \cref{lin:eq:fluxes_bounds} with the divergence bound of \cref{lin:eq:JLambda_bound_rough} derived from \cref{lin:eq:divergences} we obtain that for $\phi$ decaying sufficiently rapidly towards $t=0$, it holds that
	\begin{equation}
		\int_{\Sigma_{\tau}} J^\Lambda[\phi]\lesssim \int 0^\tau \frac{\dd\tilde{\tau}}{\tilde{\tau}} \tilde{\tau}^{\nu-1}\left(\int_{\Sigma_{\tilde{\tau}}} T\cdot \T^0[\phi] \right)^{1/2}+\text{terms involving }f.
	\end{equation}
	Therefore, assuming the coercivity mentioned above, we see that the error term decays at the same order as $J^\Lambda$, thus possible to bound via Gronwall if the decay rate $N$ is sufficiently large.
	Similar argument applies to $J^c$. 
	
	\subsection{Unstable modes}\label{lin:sec:unstable}
	Next, we study the obstruction to the coercivity of $\tilde{J}^T_*$ coming from the nonzero eigenvalue $\norm{Y}_{L^2}=1$ of $(\Delta+5W^4)Y=\kappa^2Y$ with $\kappa>0$.
	It is well known that $Y$ decays exponentially $\abs{Y}\lesssim e^{-\kappa r}$, see \cite[Proposition 3.9]{duyckaerts_solutions_2016}.
	Fix a smooth cutoff function $\chi$ localising to $r\leq 1$ and let $\chi_R(\cdot)=\chi(\cdot/R)$.
	Recall the coordinates $\boldsymbol{x},\boldsymbol{t}$ from \cref{not:eq:box_nu}.
	We define 
	\begin{equation}\label{lin:eq:alpha_def}
		\alpha_z^\pm[\phi](\tau)=\int_{\Sigma^z_\tau}\chi_R(y_z) Y(\blambda_z\boldsymbol{x}_z)(\partial_{\boldsymbol{t}_z}\phi \pm \kappa \phi)\dd \boldsymbol{x}^3_z.
	\end{equation}
	We  compute the time derivatives as follows:
	\begin{lemma}\label{lin:lem:alpha_eq}
		Assume that $R<R_2$, where $R_2$ is the transition region in \cref{not:eq:v_def}.
		For $\phi$ solving \cref{lin:eq:main_cutoffed} it holds that
		\begin{equation}
			\tau^\nu \partial_\tau \alpha_z^\pm[\phi](\tau)=\pm \blambda_z\kappa \alpha_z^\pm[\phi](\tau)+\int_{\Sigma_\tau^z}\phi [\Delta_y,\chi_R ]Y+\A^{\infty,\infty,\nu-1}(\phi,\partial_{\mathbf{t}}\phi)+Yt^{2\nu}f\dd \boldsymbol{x}^3_z
		\end{equation}
	\end{lemma}
	\begin{proof}
		We prove this for $z=0$, the others following similarly.
		We note that
		\begin{equation}
			\tau^\nu\partial_\tau \alpha=\partial_{\mathbf{t}}\alpha=\int_{\Sigma_\tau}\chi_R Y (\partial_{\mathbf{t}}^2\pm \kappa\partial_{\mathbf{t}})\phi\dd y^3
		\end{equation}
		We use \cref{not:eq:partial_nu}, \cref{map:def:approximate}  to express $\Box$ and $\phi_A^4$ respectively, and an integration by parts yields the result.
	\end{proof}
	
	Let us use the globalising coordinate choice $\boldsymbol{\tau}=\tau^{1-\nu}/(\nu-1)$ setting $\tau^\nu\partial_\tau=-\partial_{\boldsymbol{\tau}}$.
	We use this lemma to bound $\alpha^\pm$ in forward or backward in time.
	\begin{lemma}\label{lin:lem:alpha_bound}
		Fix $\tau_1<\tau_2$ and corresponding $\mathbf{\tau}_1>\mathbf{\tau}_2$.
		For $R,R_2,\phi$ as in \cref{lin:lem:alpha_eq} it holds that
		\begin{nalign}\label{lin:eq:alpha_bound}
			\alpha_z^+(\tau_2)\lesssim e^{ \blambda_z(\mathbf{\tau}_2-\mathbf{\tau}_1)}\alpha^+(\tau_1)+  \int_{\tau_1}^{\tau_2}\dd\tau e^{\blambda_z(\boldsymbol{\tau}_2-\boldsymbol{\tau})} \left(\tau^{-\nu/2}e^{-R\blambda_z}\left(\int_{\Sigma^z_\tau} T_z\cdot \T^0[\phi] \right)^{1/2}+\left(\int_{\Sigma_\tau^z} \chi t^{4\nu}f^2\right)\right)\\
			\alpha_z^-(\tau_1)\lesssim e^{ \blambda_z(\mathbf{\tau}_2-\mathbf{\tau}_1)}\alpha^-(\tau_2)+\int_{\tau_1}^{\tau_2}\dd\tau e^{\blambda_z(\boldsymbol{\tau}_2-\boldsymbol{\tau})} \left(\tau^{-\nu/2}e^{-R\blambda_z}\left(\int_{\Sigma_\tau^z} T_z\cdot \T^0[\phi] \right)^{1/2}+\left(\int_{\Sigma^z_\tau} \chi t^{4\nu}f^2\right)\right)
		\end{nalign}
	\end{lemma}

	\subsection{Coercivity}\label{lin:sec:coercivity}
	In this section, we use the control provided by $J^c,J^\Lambda,\alpha,J^{T}_*$ to estimate the coercive quantity $E$.
	The results in this section are well known for the energy critical wave equation, and we will mostly recall these statements.
	
	\begin{lemma}[\cite{duyckaerts_solutions_2016}]\label{lin:lem:coerc_basic}
		Let  $f_\nu\in C^\infty_c$ for $\nu\in\{0,1,2,3\}$ be such that
		\begin{equation}\label{lin:eq:diagonal_projections}
			(f_0,\Lambda W)=1,\quad (f_i,\partial_{x_j}W)=\delta_{ij} ,\quad (f_0,\partial W)=(f_j,\Lambda W)=(E_\nu,Y)=0.
		\end{equation} 
		For $R_1$ sufficiently large, there exists $C_Y,C^{-1}$ sufficiently large such that $\forall \phi \in\dot{H}^1$ it holds that
		\begin{equation}\label{lin:eq:coercivity_basic}
			(-\phi,(\Delta+V)\phi)+C_Y(\chi_{<R_1}Y,\phi)^2+\sum_\nu(f_\nu,\phi)^2\geq C \norm{\phi}_{\dot{H}^1(\R^3)}.
		\end{equation}
	\end{lemma}
	\begin{proof}
		This is a small modification of the result \cite[Proposition 3.6]{duyckaerts_solutions_2016} already contained in \cite[Lemma  5.5]{kadar_construction_2024}.
	\end{proof}
	
	We also need  a localised version of the above coercivity statement as follows:
	\begin{lemma}\label{lin:lem:coerc_loc}
		Let $f_\nu\in (\jpns{r}^{-1/4}\dot{H}^{1})^*$ satisfy \cref{lin:eq:diagonal_projections} and $R_1$ be as in \cref{lin:lem:coerc_basic}.
		There exists $C'>0$ such that for all $R_3$ sufficiently large it holds that
		\begin{equation}\label{lin:eq:coercivity_localised}
			\int_{r<R_3}\abs{\nabla \phi}^2-V\phi^2\dd x^3+\frac{C'}{R_3}(-\phi,(\Delta+V)\phi)+C_Y(\chi_{<R_1}Y,\phi)^2+\sum_\nu(f_\nu,\phi)^2\geq C \int_{r<R_3}\abs{\nabla \phi}^2\dd x^3
		\end{equation}
	\end{lemma}
	\begin{proof}
			\emph{Step 1:}
			Let us begin by assuming that $C_Y(\chi_{<R_1}Y,\phi)^2+\sum_\nu(f_\nu,\phi)^2=0$.
			Indeed, deriving \cref{lin:eq:coercivity_localised} follows by projecting out these modes.
			
			\emph{Step 2:}
			We write $B_{R_3}=\{r<R_3\}$ for $R_3\gg1$.
			Let us split $\phi=\phi_\Delta+\phi_0$ with $\phi_0|_{r>R_3}=0$ and
			\begin{equation}
				\Delta\phi_\Delta|_{r<R_3}=0,\qquad \phi_\Delta|_{\partial B_{R_3}}=\phi|_{\partial B_{R_3}}.
			\end{equation}
			We notice that 
			\begin{nalign}
				\int_{B_{R_3}}\abs{\nabla\phi}^2-V\phi^2\dd x^3=\int_{B_{R_3}}\abs{\nabla\phi_\Delta}^2+\abs{\nabla\phi_0}^2-V\phi_0^2-V\phi_0\phi_\Delta-V\phi_\Delta^2\dd x^3\\
				\gtrsim (-\phi_0,(\Delta+V)\phi_0)-\int_{B_{R_3}}V\phi_0\phi_\Delta+V\phi_\Delta^2\dd x^3.
			\end{nalign}
			Using \cref{lin:eq:coercivity_basic} for $\phi_0$, we may use the first term on the right hand side to obtain coercive control of $\phi_0$ as
			\begin{equation}\label{lin:eq:coercive_0}
				(-\phi_0,(\Delta+V)\phi_0)+c_Y(\chi_{<R_1}Y,\phi_\Delta)^2+\sum_\nu(f_\nu,\phi_\Delta)^2\geq C \norm{\phi_0}_{\dot{H}^1(\R^3)}.
			\end{equation}
			Next, we bound the $\phi_\Delta$ projections.
			
			\emph{Step 3:}
			From Hardy and trace inequalities we control the data of the Dirichlet problem by the exterior norm:
			\begin{equation}\label{lin:eq:hardyTrace_ext}
				\norm{\phi_\Delta}_{\dot{H}^{1/2}(\partial B_{R_3})}+R_3^{-1/2}\norm{\phi_\Delta}_{L^2(\partial B_{R_3})}\lesssim \norm{\phi_\Delta}_{\dot{H}^1(B_{R_3}^c)},
			\end{equation}
			where the scaling follows by considering $\phi_{\Delta,R_3}(y)=\phi_\Delta(R_3y)$ and applying the unit scale variants.
			Let us recall the following standard interior and near boundary estimates for the Dirichlet problem (obtained via scaling)
			\begin{nalign}\label{lin:eq:dirichlet_bound}
				\norm{\phi_\Delta}_{L^\infty(B_{R_3/2})}+\norm{R_3\partial\phi_\Delta}_{L^\infty(B_{R_3/2})}\
				&\lesssim R_3^{-1}\norm{\phi_\Delta}_{L^2(\partial B_{R_3})};\\
				\norm{\phi_\Delta}_{L^2(B_{R_3})}+R_3\norm{\partial\phi_\Delta}_{L^2(B_{R_3})}&\lesssim R_3^{1/2}\norm{\phi_\Delta}_{L^2(\partial B_{R_3})}.
			\end{nalign}
			We use these  bounds to estimate the $V\phi_\Delta^2$ error term as
			\begin{multline}
				\int V\phi_\Delta^2=\int_{r<R_3/2} V\phi_\Delta^2+\int_{R_3/2<r<R_3} V\phi_\Delta^2+\int_{r>R_3} V\phi_\Delta^2\\
				\lesssim \norm{\phi_\Delta}_{L^\infty(B_{R_3/2})}^2+R_3^{-2}\int_{R_3/2<r<R_3}\phi^2r^{-2}+R_3^{-2}\int_{R_3<r}\phi^2r^{-2}\lesssim R_3^{-2}\norm{\phi_\Delta}_{\dot{H}^{1}(B^c_{R_3})}^2
			\end{multline}
			Similarly we may bound the projections of $\phi_\Delta$ as
			\begin{nalign}\label{lin:eq:proj_delta}
				(\chi_{<R_1}Y,\phi_\Delta)&\lesssim R_3^{-1}\norm{\phi_\Delta}_{\dot{H}^1(B_{R_3}^c)}^2,\\
				\int_{r\leq R_3} (\abs{\phi_\Delta}+\jpns{r}\abs{\partial\phi_\delta}) r^{-2-\frac{3}{4}}&\lesssim  R_3^{-1/4}\norm{\phi_\Delta}_{\dot{H}^1(B_{R_3}^c)},\\
				\int_{r> R_3} (\abs{\phi_\Delta}+\jpns{r}\abs{\partial\phi_\delta}) r^{-2-\frac{3}{4}} &\lesssim  R_3^{-1/4}\norm{\phi_\Delta}_{\dot{H}^1(B_{R_3}^c)},\\
				(f_\nu,\phi_\Delta)&\lesssim \int (\abs{\phi_\Delta}+\jpns{r}\abs{\partial\phi}) r^{-2-\frac{3}{4}}\lesssim R_3^{-1/4}\norm{\phi_\Delta}_{\dot{H}^1(B_{R_3}^c)}.
			\end{nalign}
			
			\emph{Step 4:} 
			We may combine \cref{lin:eq:coercive_0,lin:eq:proj_delta} to obtain for some constant $C'$ sufficiently large
			\begin{equation}
				(-\phi_0,(\Delta+V)\phi_0)+C'R_3^{-1/2}\norm{\phi_\Delta}_{\dot{H}^1(B_{R_3}^c)}^2\geq C \norm{\phi_0}_{\dot{H}^1(\R^3)}^2.
			\end{equation}
			Hence, we deduce that
			\begin{equation}
				\int_{B_{R_3}} \abs{\nabla\phi}^2-V\phi^2+C'R_3^{-1/2}\int \abs{\nabla\phi}^2\geq C \norm{\phi_0}_{\dot{H}^1(\R^3)}^2-\int V\phi_0\phi_\Delta.
			\end{equation}
			We can bound the cross term
			\begin{equation}
				\int V\phi_0\phi_\Delta\leq \epsilon \int \phi_0^2\jpns{r}^{-2}+\epsilon^{-1}\int \phi_\Delta^2\jpns{r}^{-6}\lesssim \epsilon\norm{\phi_\Delta}_{\dot{H}^1}^2+\epsilon^{-1} R_3^{-1}\norm{\phi_\Delta}_{\dot{H}^1(B^c_{R_3})}^2
			\end{equation}
			and thus conclude the result.
	\end{proof}
	
	For our case, it will be convenient to recast this estimate with respect to the fluxes that also have $\partial_t\phi$ terms:
	\begin{prop}\label{lin:prop:coercivity}
		Let $J^c,J^\Lambda$ be as in \cref{lin:eq:currents} and $\alpha$ as in \cref{lin:eq:alpha_def}.
		For $\tau$ sufficiently small  there exists $C_Y,C,C'>0$ such that
		\begin{multline}\label{lin:eq:coercivity_used}
			\tau^{\nu}C_Y \abs{\alpha^\pm[\phi]}^2+\tau^{-2\nu}\left(\int_{\Sigma_\tau^z}J^c[\phi]\right)^2+\tau^{2-2\nu}\left(\int_{\Sigma_\tau^z}J^\Lambda[\phi]\right)^2+\int_{\Sigma_\tau^z\cap \D_z} T_z\cdot\T^{5\phi_A}[\phi]+\tau^{\frac{\nu-1}{4}} \int_{\Sigma^z_\tau} T_z\cdot\T^{0}[\phi]\\
			\geq C\int_{\Sigma_z\cap \D_z} T_z\cdot\T^{0}[\phi]
		\end{multline}
	\end{prop}
	\begin{proof}
		We prove this only for $z=0$.
		We work on $\Sigma_\tau^0$ and use $\boldsymbol{x},\boldsymbol{t}$ coordinates.
		Working in $\D_0$, we observe that near its boundary $x\sim t$, therefore $\boldsymbol{x}\sim t^{1-\nu}$, and this is exactly the value of $R_3$ that we are taking from \cref{lin:lem:coerc_loc}.
		
		We may bound the term of $J^c,J^\Lambda$ that have a time derivative with the corresponding term from $T_z\cdot\T^{5\phi_A}[\phi]$ for which the time derivative is coercive, see \cref{lin:eq:Tstar_flux}.
		The estimate \cref{lin:eq:coercivity_used} follows from \cref{lin:lem:coerc_loc} together with the bounds in \cref{lin:lemma:flux} once one writes all the terms with respect to $\boldsymbol{x}$ coordinates.
	\end{proof}

	\subsection{$L^2$ estimate }\label{lin:sec:L2}
	We are ready to prove \cref{lin:prop:energyEstimate} for $k=0$.
	Let's take $\tilde{J}_q=\tilde{J}^T_{q,*}+c\tilde{J}^T_q$ for $c>0$ sufficiently small to be determined, and $q$ as in \cref{lin:lemma:T-energy}.
	Using the divergence theorem for $\tilde{J}_q$ we already obtain
	\begin{multline}\label{lin:eq:1}
		\int \div \tilde{J}_q +Ct^{-q}\abs{f}\abs{\partial\phi}\dd\mu +\norm{\chi\phi }_{L^2} \gtrsim cq\int t^{-q-1}\abs{\partial\phi}^2+t^{-q-1}\phi^2\A^{0,2\nu-4,-2\nu}\dd\mu +\int t^{-q-1}\phi^2\A^{0,2\nu-4,-\nu-1}\dd\mu\\
		+q\int t^{-q-1}\abs{\partial\phi}^2 \A^{0,\nu-1,\nu-1}\dd\mu 
		+\sum_z \int\phi^2 t^{-q-1}\A^{0,2\nu-4,-2\nu}\dd\mu
		\\+\sum_{z\in A}q\int_{\D_z}t^{-q-1}\left((1-h'{}^2)(T_z\phi)^2+\abs{\partial_x|_{v_z}\phi}^2-5W_z\phi^2\right)\dd\mu,
	\end{multline}
	where we kept only the dependence on the constant $c$ explicit.
	Next, we apply \cref{lin:eq:coercivity_used} to obtain that the last term in this expression combined with the first yields coercive energy control locally of size \emph{independent of $c$}.
	For $t$ (depending on $c$) sufficiently small it holds that 
	\begin{multline}\label{lin:eq:2}
		\int_\tau\dd\tau \tau^{-1-q}\left(t^{\nu}C_Y \abs{\alpha^\pm[\phi]}^2+t^{-2\nu}\left(\int_{\Sigma_\tau}J^c[\phi]\right)^2+t^{2-2\nu}\left(\int_{\Sigma_\tau}J^\Lambda[\phi]\right)^2\right)\\
		+\int_{\D_0}t^{-1-q}\left((1-h'{}^2)(T_z\phi)^2+\abs{\partial_x|_{v_z}\phi}^2-5W_z\phi^2\right)\dd\mu+\frac{cq}{10\abs{A}}\int t^{-1-q}\abs{\partial\phi}^2\dd\mu \\
		\gtrsim \int_{\D_0}t^{-1-q}\left((1-h'{}^2)(T_z\phi)^2+\abs{\partial_x|_{v_z}\phi}^2\right),
	\end{multline}
	and similar estimate apply in each $\D_z$
	We obtain control of $\phi^2$ norm via Hardy inequality applied to each region as
	\begin{nalign}\label{lin:eq:3}
		\frac{cq}{10\abs{A}}\int t^{-1-q}\abs{\partial\phi}^2\dd\mu+\int_{\D_0}t^{-1-q}\left((1-h'{}^2)(T_z\phi)^2+\abs{\partial_x|_{v_z}\phi}^2\right)\dd\mu\\
		\gtrsim \int_{\D_0} t^{-1-q-2\nu}\jpns{\boldsymbol{x}}^{-1/10-2}\phi^2\dd\mu,
	\end{nalign}
	where we bound $\phi^2$ in a tubular neighbourhood $\abs{x_z}/t_z\in(\delta/8,\delta/4)$ using the first term and propagate this to the interior. 
	The $c$ constant does not appear on the right, as it is replaced by the $\jpns{x}^{-1/10}$ for sufficiently small $t$.

	Therefore, we may combine \cref{lin:eq:1,lin:eq:2,lin:eq:3} to get for $c,t,q^{-1}$ sufficiently small that
	\begin{multline}\label{lin:eq:4}
				\int \div J_q+Ct^{-q}\abs{f}\abs{\partial\phi}\dd\mu +\norm{\chi\phi }_{L^2}^2 +\int_\tau\dd\tau \tau^{-1-q}\left(t^{\nu}C_Y \abs{\alpha^\pm[\phi]}^2+t^{-2\nu}\left(\int_{\Sigma_\tau}J^c[\phi]\right)^2+t^{2-2\nu}\left(\int_{\Sigma_\tau}J^\Lambda[\phi]\right)^2\right)\\
				\gtrsim cq\int t^{-q-1}\abs{\partial\phi}^2+t^{-q-1}\phi^2\A^{0,2\nu-4,-2\nu}\dd\mu+\int t^{-q-1}\phi^2\A^{0,2\nu-4,-\nu-1}\dd\mu\\
				+q\int t^{-q-1}\abs{\partial\phi}^2 \A^{0,\nu-1,\nu-1}\dd\mu
				+\sum_z \int\phi^2 t^{-q-1}\A^{0,2\nu-4,-2\nu}\dd\mu
				\\+\sum_{z\in A}q\int_{\D_z}t^{-q-1}\left((1-h'{}^2)(T_z\phi)^2+\abs{\partial_x|_{v_z}\phi}^2+\jpns{\boldsymbol{x}_z}^{-2-1/10}\phi^2\right)\dd\mu\\
				\gtrsim cq\int t^{-q-1}\abs{\partial\phi}^2\dd\mu+\sum_{z\in A}q\int_{\D_z}t^{-q-1}\left((1-h'{}^2)(T_z\phi)^2+\abs{\partial_x|_{v_z}\phi}^2+t^{-2\nu}\jpns{\boldsymbol{x}_z}^{-2-1/10}\phi^2\right)\dd\mu.
	\end{multline}
	
	We proceed to bound the error terms arising from the projections.
	We begin with $J^\Lambda$.
	Let us write $\D_{0,\tau}=\D_0\cap \{t<\tau\}$.
	Using \cref{lin:eq:divergencesLambda} we obtain that
	\begin{equation}
		\int \dd \tau t^{1-q-2\nu}\left(\int_{\Sigma_\tau}J^\Lambda[\phi]\right)^2\lesssim\int \dd\tau t^{1-q-2\nu}\left(\int_{\D_{0,\tau}} \abs{f} t^{-\nu/2-1}\jpns{\boldsymbol{x}}^{-1}+\abs{\partial\phi}t^{\nu/2-3}\jpns{\boldsymbol{x}}^{-2} +\phi t^{-3\nu/2-2}\jpns{\boldsymbol{x}}^{-4}\right)^2,
	\end{equation}
	where we used the following bounds on each term
	\begin{subequations}
		\begin{align}
				\begin{multlined}
				\int\dd\tau t^{1-q-2\nu}\left(\int_{\D_{0,\tau}}\dd t\dd x^3 \phi t^{-3\nu/2-2}\jpns{\boldsymbol{x}}^{-4}\right)^2\lesssim \int\dd\tau t^{1-q-2\nu}\left(\int_{\D_{0,\tau}}\dd t\dd\boldsymbol{x}^3 \phi t^{3\nu/2-2}\jpns{\boldsymbol{x}}^{-4}\right)^2\\
				\lesssim\int\dd\tau t^{1-q-2\nu}\left(\int_{\D_{0,\tau}}\dd t\dd\boldsymbol{x}^3 \phi^2 t^{3\nu-3}\jpns{\boldsymbol{x}}^{-2}\right)\lesssim \int_{\D_0}\dd t\dd\boldsymbol{x}^3 \phi^2 t^{\nu-1-q}\jpns{\boldsymbol{x}}^{-2},
			\end{multlined}\\
			\int\dd\tau t^{1-q-2\nu}\left(\int_{\D_{0,\tau}} \abs{f} t^{-\nu/2-1}\jpns{\boldsymbol{x}}^{-1}\right)^2\lesssim\int_{\D_0} \dd t\dd\boldsymbol{x}^3 \abs{f}t^{3\nu-q}\jpns{\boldsymbol{x}}^2, \label{lin:eq:divf_Lambda}\\
			\int\dd\tau t^{1-q-2\nu}\left(\int_{\D_{0,\tau}} \abs{\partial\phi}t^{\nu/2-3}\jpns{\boldsymbol{x}}^{-2} \right)^2\lesssim \int_{\D_0} \dd t\dd \boldsymbol{x}^3\abs{\partial\phi}^2t^{\nu-1-q+4\nu-2}.
		\end{align}
	\end{subequations}
	Therefore, choosing $q$ sufficiently large, we may drop $J^\Lambda$ from the left hand side of \cref{lin:eq:4} at the cost of introducing an $f$ dependence as in \cref{lin:eq:divf_Lambda}.
	
	Observing the  structure of the divergence for $J^c,\tau J^m$ in \cref{lin:eq:divergences} we see that a similar bound also holds for each.
	
	Finally, we also bound the projections $\alpha^\pm$.
	Using \cref{lin:lem:alpha_bound} we obtain that
	\begin{equation}
		\int\dd\tau \tau^{-1-q+\nu}\abs{\alpha^+[\phi]}^2\lesssim  e^{-R_1\kappa}\int_{\D_0} t^{-1-q}\T[\phi](T^0\otimes T)+\int_{\D_0} t^{-1-q+2\nu}f^2 \chi_{<R_2}(y_0).
	\end{equation}
	For $\alpha^-$, we need to introduce an additional control in the region $t\in(t_*1/2,t_*)$ to control for the boundary term on the right hand side of \cref{lin:eq:alpha_bound}
		\begin{equation}
		\int\dd\tau  \tau^{-1-q+\nu}\abs{\alpha^-}^2\lesssim e^{-R_1\kappa}\int_{\D_0} t^{-1-q}\T[\phi](T^0\otimes T)+\int_{\D_0} t^{-1-q+2\nu}f^2+\norm{\chi\phi }_{L^2}^2.
	\end{equation}
	Taking $q$ again sufficiently large, we can again absorb these contributions at the expense of additional terms dependent on $f^2$.
	Therefore upon an application of Cauchy-Schwartz to bound the $f\partial\phi$ term as well we conclude
	\begin{equation}
		\int \div J_q+C f^2 t^{-2\nu-q}\jpns{\boldsymbol{x}}^2\dd\mu+\norm{\chi\phi }_{L^2}^2
		\gtrsim c\int t^{-q-1}\abs{\partial\phi}^2+t^{-2\nu}\jpns{\boldsymbol{x}_z}^{-2}\phi^2.
	\end{equation}
	Relabelling $\frac{q-\nu}{2}\mapsto q$, this yields the $k=0$ form of \cref{lin:prop:energyEstimate} away from $\C$ using that
	\begin{nalign}
		&\int_{\D_0} \dd x^3\dd t \phi^2 t^{-1-q-2\nu}\jpns{\boldsymbol{x}}^{-2}\phi^2=\int_{\D_0}\mu_{\b} \phi^2 t^{-1-q+\nu}\jpns{\boldsymbol{x}} u\phi^2\sim \norm{\phi }_{\Hb^{0;\frac{q-\nu}{2}+(\nu-1)/2,\frac{q-\nu}{2}}(\D_z)}\\
		&\int_{\D_0}\mu f^2 t^{-2\nu-q}\jpns{x}^2\sim\norm{f}_{\Hb^{0;\frac{4\nu+q-5}{2},\frac{q-\nu}{2}}}.
	\end{nalign}
	Indeed, note that we ignored the $\frac{\nu-1}{2}$ improvement near $i_-$.
	To improve this around the lightcone $\C$, we apply a cutoff function localising to $\D_{ext}=\{1-\abs{x}/t<\delta\}$ and use \cref{model:lem:i} to improve the weight
	\begin{equation}
		\Hb^{0;\frac{-1}{2},\frac{q-\nu}{2}+(\nu-1)/2,\frac{q-\nu}{2}}(\Mcomp) \mapsto \Hb^{0;\frac{q-\nu}{2}+(\nu-1)/2+1/2,\frac{q-\nu}{2}+(\nu-1)/2,\frac{q-\nu}{2}}(\Mcomp).
	\end{equation}

	Let us use the relabelling $\frac{q-\nu}{2}\mapsto q$ from now on.
	To obtain the higher order estimates, we simply commute with unit size partial derivatives
	\begin{equation}
		\Box \partial\phi=-5\partial\phi_a^4\phi+\partial f.
	\end{equation}
	We already know that 
	\begin{equation}
		\partial\phi\in \Hb^{0;q-1/2,q-1,q-\nu}(\D)
		\implies \partial\phi_a^4\phi\in \Hb^{0;q-1/2,q+4(\frac{\nu-1}{2}-1)-1,q-3\nu}(\D)
	\end{equation}
	Therefore, we may apply \cref{model:lem:i}  for $q-3\nu>-1/2$ to control
	\begin{equation}
		\norm{\Ve\partial\phi}_{\Hb^{0;q-\nu+1,q-\nu,q-\nu}}\sim \norm{\Ve \rho^{\nu}\partial\phi}_{\Hb^{0;q+1/2,q,q}}.
	\end{equation}
	The result follows by induction, where the upper bound on $s$ comes from being able to apply $\cref{model:lem:i}$, and so the right hand side must satisfy $q-(\nu-1)s-3\nu>-1/2$.
	
	\section{Nonlinear solutions}\label{sec:nonlin}
	In this section, we apply the linear estimates proved in \cref{sec:lin} for the nonlinear problem.
	Let $\phi_A$ be an approximate solution with rate $\nu$ and error $\A^{\infty,\infty,\infty}(\Mcomp)$, and define
	\begin{equation}
		\mathcal{N}_{\phi_A}[\phi]=P[\phi_A+\phi]-P[\phi_A]-D_{\phi_A}P\phi.
	\end{equation}
	We begin by bounding the nonlinearities via the linear estimate \cref{lin:eq:right_inverse}.
	For this estimate, we are extremely lossy with the necessary decay rate and use embedding that are far from sharp.
	This is merely for presentational purpose as these are sufficient for closing the nonlinear iteration.
	\begin{lemma}\label{non:lemma:nonlin}
		There exists $q_1(\nu)$ such that for all $q>q_1(\nu)$ the following holds.
		Let $s=10$ and $\phi_1,\phi_2$ be a smooth function in $\Mcomp$ with
		\begin{equation}\label{non:eq:phi}
			\norm{\Ve \{1,\rho^{\nu}\partial\}^s\phi_i}_{\Hb^{0;q,q,q}(\Mcomp)}\leq\epsilon_i\leq 1\qquad \text{for }i\in\{1,2\}.
		\end{equation}
		Then, it holds that
		\begin{subequations}
			\begin{align}
				\norm{t^{-1}\{1,\rho^{\nu}\partial\}^s \mathcal{N}_{\phi_A}[\phi_1]}_{\Hb^{0;q-1,q+5\frac{\nu-1}{2},q}(\Mcomp)}\lesssim_s \epsilon_1^2\label{non:eq:N1}\\
				\norm{t^{-1}\{1,\rho^{\nu}\partial\}^s (\mathcal{N}_{\phi_A}[\phi_1])-\mathcal{N}_{\phi_A}[\phi_1+\phi_2]}_{\Hb^{0;q-1,q+5\frac{\nu-1}{2},q}(\Mcomp)}\lesssim_s \epsilon_1\epsilon_2\label{non:eq:N2}
			\end{align}
		\end{subequations}
	\end{lemma}
	\begin{proof}
		We use that \cref{non:eq:phi} implies the upper bound
		\begin{equation}
			\norm{ \phi}_{\Hb^{s;q-s(\nu-1),q-s(\nu-1),q-s(\nu-1)}(\Mcomp)}\lesssim_{s,q}\epsilon
		\end{equation}
		and furthermore it suffices to bound $\norm{t^{-1}\mathcal{N}_{\phi_A}[\phi]}_{\Hb^{0;q-1,q+5\frac{\nu-1}{2},q}(\Mcomp)}$.
		The rest follows from \cref{map:lem:nonlin_L2}.
		We show \cref{non:eq:N1} from \cref{map:eq:Hb_bound}, \cref{non:eq:N2} following similarly from \cref{map:eq:Hb_contraction}.
		
		Observing that $\phi_A\in\A^{0,\nu/2-1,-\nu/2}(\Mcomp)$ and taking $q-10(\nu-1)>-\nu/2$ we also have $\phi\in\Hb^{s;0,\nu/2-1,-\nu/2}(\Mcomp)$.
		Then it follows that
		\begin{equation}
			\norm{t^{-1}\phi^2\phi_A^3}_{\Hb^{0;q-1,q+5\frac{\nu-1}{2},q}(\Mcomp)}\lesssim\norm{\phi^2}_{\Hb^{0;q-1,q+8\frac{\nu-1}{2}+1,q+1+3\nu/2}(\Mcomp)}\lesssim \norm{\phi}_{\Hb^{0;(q-1,q+8\frac{\nu-1}{2}+1,q+1+3\nu/2)/2}(\Mcomp)}^2.
		\end{equation}
		Taking $q$ sufficiently large so that $q+1+3\nu/2<2\left(q-10(\nu-1)\right)$, $q+8\frac{\nu-1}{2}+1<2\left(q+\frac{\nu-1}{2}-10(\nu-1)\right)$ and $q-1<2\left(q+\frac{\nu-1}{2}\right)$ implies that the right hand side is bounded by $\epsilon^2$.
		The other nonlinear terms follow similarly.
	\end{proof}
	
	\begin{corollary}\label{nonlin:prop:contraction}
		Let $\mathfrak{S},q_0$ be a smooth right inverse and minimum decay rate from \cref{lin:thm:right_inverse}. 
		Pick $q_1$ as in \cref{non:lemma:nonlin}
		Pick $q>\max(q_0,q_1)$ and pick $s=10$.
		Then, there exists $T_f,\epsilon>0$ sufficiently small, depending on $\mathfrak{S}$, such that on $B_\epsilon=\{\phi \in X_{s,q}:\norm{f}_{X_{s,q}}\leq \epsilon\}$
		\begin{equation}\label{nonlin:eq:contraction}
			\mathfrak{S}\circ \mathcal{N}_{\phi_A}:B_\epsilon\to B_{\epsilon/2}
		\end{equation}
		is a contraction.
		
	\end{corollary}
	
	\begin{proof}
		From \cref{map:lem:nonlin_L2,lin:thm:right_inverse} we already have
		\begin{equation}
			X_{s,q}\xrightarrow[\cref{non:eq:N1}]{\mathcal{N}_{\phi_A}}Y_{s,q+1}\xrightarrow[\cref{lin:thm:right_inverse}]{\mathfrak{S}}X_{s,q+1}.
		\end{equation}
		Let us also note that the first map is a purely nonlinear, thus for $\epsilon$ small enough the composition $\mathfrak{S}\circ\mathcal{N}_{\phi_A}$ satisfies the mapping \cref{nonlin:eq:contraction}.
		
		Similarly, using \cref{non:eq:N2}, the contraction property follows for $\epsilon$ sufficiently small.
	\end{proof}

	Finally, we can conclude the main result
	\begin{proof}[Proof of \cref{in:thm:correction}]
		Let $q,s,\epsilon,T_f,B_\epsilon$ be as in \cref{nonlin:prop:contraction}.
		
		The approximate solution $\phi_A$ constructed in \cref{an:prop:peeling} satisfy $P[\phi_A]\in\Hb^{\infty; \infty,\infty,\infty}(\Mcomp)\subset Y_{\infty,\infty}$.
		Therefore, for $T_f>0$ possibly small than before, it holds that $\mathfrak{S}(P[\phi_A])\in B_{\epsilon/2}$ as in \cref{nonlin:eq:contraction}.
		Using the contraction mapping theorem, there exists $\phi\in B_\epsilon$ such that
		\begin{equation}
			\mathfrak{S}\circ(\mathcal{N}_{\phi_A}[\phi]+P[\phi_A])=-\phi.
		\end{equation}
		This implies that $P[\phi_A+\phi]=0$ and proves the existence of a solution $\phi\in X_{s,q}$.
		
		Since $\mathfrak{S}$ is a smooth right inverse, it follows that $\phi\in X_{\infty,\infty}\subset \Hb^{\infty;\infty,\infty,\infty}(\Mcomp)$.
	\end{proof}
	
	\pagebreak
	\appendix
	
	\section{Flux computations}\label{sec:fluxes}
	
	We record some standard flux computations closely following \cite{kadar_construction_2024}:
	
	\paragraph{Energy}
	We compute the $T$ energy through a hypersurface $\tilde{\Sigma}=\{t+h(x)=c\}$ for some $h\in\mathcal{C}^\infty$ with $\norm{\nabla h}\leq1$. Let us use coordinates $x,s=t+h(x)-c$ in a neighbourhood of $\tilde{\Sigma}$. We remark the following computations
	\begin{subequations}
		\begin{align}			
			\partial_i|_t=\partial_i|_s-h_iT,\quad ds=dt+\partial_ih\cdot dx_i\\
			\T^w[\phi](\dd t,\dd s)=(1-\abs{\nabla h}^2)(T\phi)^2+\abs{\partial|_{\tilde{\Sigma}}\phi}^2-w\phi^2\label{app:eq:en_mom}.
		\end{align}
	\end{subequations}
	The induced measure on $\tilde{\Sigma}$ with the above coordinates is $\det(\delta_{ij}-\partial_ih\partial_jh)^{1/2}=(1-\abs{\partial h}^2)^{1/2}=\abs{\dd s}$. 
	Thus, we find the induced energy to be
	\begin{equation}
		\begin{gathered}
			\int_{\tilde{\Sigma}_\tau}\T^w[\phi]\bigg(\frac{ds}{\abs{ds}},dt\bigg)=\int_{\tilde{\Sigma}_\tau}\frac{1}{2}\big((T\phi)^2+\partial_i|_t\phi\cdot\partial_i|_t\phi-w\phi^2\big)+\partial_ih\partial_i|_{t}\phi T\phi\dd x^3\\
			=\int_{\tilde{\Sigma}_\tau}\frac{1}{2}\Big((1-\partial h\cdot\partial h)(T\phi)^2+\partial_i|_s\phi\cdot\partial_i|_s\phi-w\phi^2\Big)\dd x^3.
		\end{gathered}
	\end{equation}
	
	We can similarly calculate the bilinear energy content, with the difference that we need to replace $\T^w[\phi](ds,dt)$ with
	\begin{equation}
		\begin{gathered}
			\T^w[\phi,\psi](dt,ds)=(1-\abs{\nabla h}^2)T\phi T\psi+\partial_i|_s\phi\cdot\partial_i|_s\psi-w\phi\psi.
		\end{gathered}
	\end{equation}
	yielding
	\begin{equation}\label{app:en:energy:fg}
		\begin{gathered}
			\int_{\tilde{\Sigma}_\tau}\T^w[\phi,\psi]\bigg(\frac{ds}{\abs{ds}},dt\bigg)
			=\int_{\R^3}\frac{1}{2}\Big((1-\partial h\cdot\partial h)(T\phi)(T\psi)+\partial_i|_s\phi\cdot\partial_i|_s\psi-w\phi\psi\Big)\dd x^3
		\end{gathered}
	\end{equation}
	For a rescaled hypersurface $\tilde{\Sigma}=\{t+\frac{h(\lambda x)}{\lambda}=\lambda\}$ with $\lambda\in\R$, we similarly have 
	\begin{equation}\label{app:eq:en_mom_scaled}
		\T^w[\phi][\dd t,\dd t+\lambda^{-1}h(\lambda x)]=(1-\abs{\nabla_{\lambda x}h}^2)(T\phi)^2+\abs{\partial_x|_{\tilde{\Sigma}}\phi}^2-w\phi^2.
	\end{equation}

	On hypersurfaces that move, as recorded for instance in $v$ in \cref{not:eq:vy_def} for $\zeta\neq0$ we have an additional term.
	For some $\lambda\in\R^+$ write $s=t+\frac{h(\lambda(x-\zeta(t)))}{\lambda}$ so that
	$\dd s=\dd t+h' \dd x-h'\dot{\zeta}\dd t$.
	The extra contribution to \cref{app:eq:en_mom} is
	\begin{equation}
		\T^w[\phi](\dd t,\dot{\zeta}\dd t)=\dot{\zeta}\T^w[\phi](\dd t,\dd t).
	\end{equation}
	
	\paragraph{Momentum}
	Let's fix a hypersurface $\tilde{\Sigma}=\{t+h(x)=c\}$ for a spherically symmetric $h$ and write $s=t+h(x)-c$.
	We compute the flux given by contraction with the $\partial_x$ Killing vector.
	For this, we use that $\partial_x=\partial_x|_{\tilde{\Sigma}}-\hat{x}h'T$ and $\eta^{-1}(\dd s,\cdot)=-(1-h'^2)T-h'\hat{x}\cdot \partial_x|_{\tilde{\Sigma}}$, where we used $\eta^{-1}$ as a raising operator.
	Therefore, we get
	\begin{nalign}\label{app:eq:mom}
		&\begin{multlined}
			\T^w[\phi](\dd x,\dd s)=(\partial_x|_{\tilde{\Sigma}}-\hat{x}h'T)\phi \big(-h'\partial_x|_{\tilde{\Sigma}}-(1-h'^2)T\big)\phi\\
			-\frac{1}{2}h'\hat{x}\Big(\partial_x|_{\tilde{\Sigma}}\phi\cdot \partial_x|_{\tilde{\Sigma}}\phi-2h'T\phi \partial_x|_{\tilde{\Sigma}}\phi-(1-h'^2)(T\phi)^2-w\phi^2\Big)
		\end{multlined}
		\\&=\hat{x}(1-h'^2)h'(T\phi)^2+2(1-h'^2)(T\phi)\partial_x|_{\tilde{\Sigma}}\phi+2h'\hat{x}\cdot\partial_x|_{\tilde{\Sigma}}\phi \partial_x|_{\tilde{\Sigma}}\phi
		-\hat{x}h'\big(\partial_x|_{\tilde{\Sigma}}\phi\cdot \partial_x|_{\tilde{\Sigma}}\phi-w\phi^2\big)
	\end{nalign}
	We can use the polarization identity to find the corresponding bilinear functional.
	
	\paragraph{Center of mass}
	Finally, we may combine the previous two computation, to find the flux corresponding to the Lorentz boost. 
	\begin{nalign}\label{app:eq:com}
		-\T^w[\phi]\bigg(\dd s, h\dd x-x\dd t\bigg)=\frac{1}{2}(T\phi)^2\bigg(\hat{x}h(1-h'^2)h'+x(1-h'^2)\bigg)-(1-h'^2)h(T\phi)X\phi\\
		-hh' \hat{x}\cdot\partial_x|_{\tilde{\Sigma}}\phi \partial_x|_{\tilde{\Sigma}}\phi+\frac{1}{2}(h'h-r)\hat{x}\abs{\partial_x|_{\tilde{\Sigma}}\phi}^2+\frac{1}{2}\phi^2\hat{x}w\big(r+hh'\big)
	\end{nalign}

	\pagebreak
	\printbibliography

@misc{angelopoulos_matching_2026,
  title = {Matching Conditions for Scattering Solutions of Scalar Wave Equations on Extremal {{Reissner-Nordstr\"om}} Spacetimes},
  author = {Angelopoulos, Yannis and Kadar, Istvan},
  year = 2026,
  month = feb,
  number = {arXiv:2602.14712},
  eprint = {2602.14712},
  publisher = {arXiv},
  doi = {10.48550/arXiv.2602.14712},
  archiveprefix = {arXiv}
}

@article{biernat_hyperboloidal_2021,
  title = {Hyperboloidal {{Similarity Coordinates}} and a {{Globally Stable Blowup Profile}} for {{Supercritical Wave Maps}}},
  author = {Biernat, Pawe{\l} and Donninger, Roland and Sch{\"o}rkhuber, Birgit},
  year = 2021,
  month = oct,
  journal = {International Mathematics Research Notices},
  volume = {2021},
  number = {21},
  eprint = {1707.09812},
  pages = {16530--16591},
  issn = {1073-7928},
  doi = {10.1093/imrn/rnz286},
  archiveprefix = {arXiv}
}

@misc{collot_singularity_2024,
  title = {Singularity Formed by the Collision of Two Collapsing Solitons in Interaction for the {{2D Keller-Segel}} System},
  author = {Collot, Charles and Ghoul, Tej-Eddine and Masmoudi, Nader and Nguyen, Van Tien},
  year = 2024,
  month = sep,
  number = {arXiv:2409.05363},
  eprint = {2409.05363},
  publisher = {arXiv},
  doi = {10.48550/arXiv.2409.05363},
  archiveprefix = {arXiv}
}

@article{collot_type_2018,
  title = {Type {{II}} Blow up Manifolds for the Energy Supercritical Semilinear Wave Equation},
  author = {Collot, Charles},
  year = 2018,
  month = mar,
  journal = {Memoirs of the American Mathematical Society},
  volume = {252},
  eprint = {1407.4525},
  doi = {10.1090/memo/1205},
  archiveprefix = {arXiv}
}

@article{cote_construction_2013,
  title = {Construction of a {{Multisoliton Blowup Solution}} to the {{Semilinear Wave Equation}} in {{One Space Dimension}}},
  author = {C{\^o}te, Rapha{\"e}l and Zaag, Hatem},
  year = 2013,
  journal = {Communications on Pure and Applied Mathematics},
  volume = {66},
  number = {10},
  eprint = {1110.2512},
  pages = {1541--1581},
  issn = {1097-0312},
  doi = {10.1002/cpa.21452},
  archiveprefix = {arXiv},
  copyright = {Copyright \copyright{} 2011 Wiley Periodicals, Inc.}
}

@article{donninger_nonscattering_2013,
  title = {Nonscattering Solutions and Blowup at Infinity for the Critical Wave Equation},
  author = {Donninger, Roland and Krieger, Joachim},
  year = 2013,
  month = sep,
  journal = {Mathematische Annalen},
  volume = {357},
  number = {1},
  pages = {89--163},
  issn = {1432-1807},
  doi = {10.1007/s00208-013-0898-1}
}

@article{duyckaerts_classification_2013,
  title = {Classification of the Radial Solutions of the Focusing, Energy-Critical Wave Equation},
  author = {Duyckaerts, Thomas and Kenig, Carlos and Merle, Frank},
  year = 2013,
  journal = {Cambridge Journal of Mathematics},
  volume = {1},
  number = {1},
  pages = {75--144},
  publisher = {International Press of Boston},
  issn = {2168-0949},
  doi = {10.4310/CJM.2013.v1.n1.a3}
}

@article{duyckaerts_dynamics_2008,
  title = {Dynamics of {{Threshold Solutions}} for {{Energy-Critical Wave Equation}}},
  author = {Duyckaerts, Thomas and Merle, Frank},
  year = 2008,
  month = jan,
  journal = {International Mathematics Research Papers},
  volume = {2008},
  pages = {rpn002},
  issn = {1687-3017},
  doi = {10.1093/imrp/rpn002}
}

@article{duyckaerts_soliton_2022,
  title = {Soliton Resolution for Critical Co-Rotational Wave Maps and Radial Cubic Wave Equation},
  author = {Duyckaerts, Thomas and Kenig, Carlos and Martel, Yvan and Merle, Frank},
  year = 2022,
  month = apr,
  journal = {Communications in Mathematical Physics},
  volume = {391},
  number = {2},
  eprint = {2103.01293},
  pages = {779--871},
  issn = {0010-3616, 1432-0916},
  doi = {10.1007/s00220-022-04330-z},
  archiveprefix = {arXiv}
}

@article{duyckaerts_soliton_2023,
  title = {Soliton Resolution for the Radial Critical Wave Equation in All Odd Space Dimensions},
  author = {Duyckaerts, Thomas and Kenig, Carlos and Merle, Frank},
  year = 2023,
  month = mar,
  journal = {Acta Mathematica},
  volume = {230},
  number = {1},
  pages = {1--92},
  publisher = {International Press of Boston},
  issn = {1871-2509},
  doi = {10.4310/ACTA.2023.v230.n1.a1}
}

@article{duyckaerts_solutions_2016,
  title = {Solutions of the Focusing Nonradial Critical Wave Equation with the Compactness Property},
  author = {Duyckaerts, Thomas and Kenig, Carlos E. and Merle, Frank},
  year = 2016,
  journal = {Annali della Scuola Normale Superiore di Pisa. Classe di scienze},
  volume = {15},
  number = {1},
  pages = {731--808},
  publisher = {Classe di Scienze},
  issn = {0391-173X},
  chapter = {Annali della Scuola Normale Superiore di Pisa. Classe di scienze}
}

@article{hillairet_smooth_2012,
  title = {Smooth Type {{II}} Blow-up Solutions to the Four-Dimensional Energy-Critical Wave Equation},
  author = {Hillairet, Matthieu and Rapha{\"e}l, Pierre},
  year = 2012,
  month = jan,
  journal = {Analysis \& PDE},
  volume = {5},
  number = {4},
  eprint = {1010.1768},
  pages = {777--829},
  issn = {2157-5045, 1948-206X},
  doi = {10.2140/apde.2012.5.777},
  archiveprefix = {arXiv}
}

@book{hintz_introduction_2025,
  ids = {hintz_introduction_2019},
  title = {An {{Introduction}} to {{Microlocal Analysis}}},
  author = {Hintz, Peter},
  year = 2025,
  month = oct,
  publisher = {Springer Nature Switzerland},
  isbn = {978-3-031-90705-0}
}

@misc{hintz_lectures_2023,
  title = {Lectures on Geometric Singular Analysis with Applications to Elliptic and Hyperbolic {{PDE}}},
  author = {Hintz, Peter},
  year = 2023
}

@article{hintz_underdetermined-elliptic_2025,
  title = {Underdetermined-Elliptic {{PDE}} on Asymptotically {{Euclidean}} Manifolds, and Generalizations},
  author = {Hintz, Peter},
  year = 2025,
  journal = {Mathematical Research Letters},
  volume = {32},
  number = {3},
  pages = {789--831},
  issn = {10732780, 1945001X},
  doi = {10.4310/MRL.250729160443}
}

@article{holzegel_boundedness_2014,
  title = {Boundedness and Growth for the Massive Wave Equation on Asymptotically Anti-de {{Sitter}} Black Holes},
  author = {Holzegel, Gustav H. and Warnick, Claude M.},
  year = 2014,
  month = feb,
  journal = {Journal of Functional Analysis},
  volume = {266},
  number = {4},
  eprint = {1209.3308},
  pages = {2436--2485},
  issn = {00221236},
  doi = {10.1016/j.jfa.2013.10.019},
  archiveprefix = {arXiv}
}

@misc{jendrej_concentric_2025,
  title = {Concentric Bubbles Concentrating in Finite Time for the Energy Critical Wave Maps Equation},
  author = {Jendrej, Jacek and Krieger, Joachim},
  year = 2025,
  month = jan,
  number = {arXiv:2501.08396},
  eprint = {2501.08396},
  publisher = {arXiv},
  doi = {10.48550/arXiv.2501.08396},
  archiveprefix = {arXiv}
}

@article{jendrej_construction_2020,
  title = {Construction of Multi-Bubble Solutions for the Energy-Critical Wave Equation in Dimension 5},
  author = {Jendrej, Jacek and Martel, Yvan},
  year = 2020,
  month = jul,
  journal = {Journal de Math\'ematiques Pures et Appliqu\'ees},
  volume = {139},
  pages = {317--355},
  issn = {0021-7824},
  doi = {10.1016/j.matpur.2020.02.007}
}

@article{jendrej_dynamics_2022,
  title = {Dynamics of Bubbling Wave Maps with Prescribed Radiation},
  author = {Jendrej, Jacek and Laurie, Andrew and Rodriguez, Casey},
  year = 2022,
  month = aug,
  journal = {Annales scientifiques de l'\'Ecole Normale Sup\'erieure},
  volume = {55},
  number = {4},
  pages = {1135--1198},
  issn = {00129593, 18732151},
  doi = {10.24033/asens.2514}
}

@article{jendrej_soliton_2025,
  ids = {jendrej_soliton_2022},
  title = {Soliton Resolution for Energy-Critical Wave Maps in the Equivariant Case},
  author = {Jendrej, Jacek and Lawrie, Andrew},
  year = 2025,
  month = jul,
  journal = {Journal of the American Mathematical Society},
  volume = {38},
  number = {3},
  pages = {783--875},
  issn = {0894-0347, 1088-6834},
  doi = {10.1090/jams/1012}
}

@article{jensen_spectral_1979,
  title = {Spectral Properties of {{Schr\"odinger}} Operators and Time-Decay of the Wave Functions},
  author = {Jensen, Arne and Kato, Tosio},
  year = 1979,
  month = sep,
  journal = {Duke Mathematical Journal},
  volume = {46},
  number = {3},
  pages = {583--611},
  publisher = {Duke University Press},
  issn = {0012-7094, 1547-7398},
  doi = {10.1215/S0012-7094-79-04631-3}
}

@misc{jeong_classification_2026,
  title = {Classification of Single-Bubble Blow-up Solutions for {{Calogero--Moser}} Derivative Nonlinear {{Schr\"odinger}} Equation},
  author = {Jeong, Uihyeon and Kim, Kihyun and Kim, Taegyu and Kwon, Soonsik},
  year = 2026,
  month = jan,
  number = {arXiv:2601.07410},
  eprint = {2601.07410},
  publisher = {arXiv},
  doi = {10.48550/arXiv.2601.07410},
  archiveprefix = {arXiv}
}

@misc{kadar_construction_2024,
  title = {Construction of Multi-Soliton Solutions for the Energy Critical Wave Equation in Dimension 3},
  author = {Kadar, Istvan},
  year = 2024,
  month = sep,
  eprint = {2409.05267},
  publisher = {arXiv},
  doi = {10.48550/arXiv.2409.05267},
  archiveprefix = {arXiv}
}

@misc{kadar_note_2026,
  title = {A Note on Exterior Stability of Finite Time Singularity Formation for Nonlinear Wave Equations},
  author = {Kadar, Istvan and Kehrberger, Lionor},
  year = 2026,
  month = feb,
  number = {arXiv:2602.03963},
  eprint = {2602.03963},
  publisher = {arXiv},
  doi = {10.48550/arXiv.2602.03963},
  archiveprefix = {arXiv}
}

@misc{kadar_scattering_2024,
  title = {A Scattering Theory Construction of Dynamical Solitons in 3d},
  author = {Kadar, Istvan},
  year = 2024,
  month = mar,
  eprint = {2403.13891},
  doi = {10.48550/arXiv.2403.13891},
  archiveprefix = {arXiv}
}

@article{kadar_scattering_2025,
  title = {Scattering, {{Polyhomogeneity}} and {{Asymptotics}} for {{Quasilinear Wave Equations From Past}} to {{Future Null Infinity}}},
  author = {Kadar, Istvan and Kehrberger, Lionor},
  year = 2025,
  month = jan,
  eprint = {2501.09814},
  doi = {10.48550/arXiv.2501.09814},
  archiveprefix = {arXiv}
}

@article{kim_rigidity_2025,
  ids = {kim_rigidity_2024},
  title = {Rigidity of Smooth Finite-Time Blow-up for Equivariant Self-Dual {{Chern}}--{{Simons}}--{{Schr\"odinger}} Equation},
  author = {Kim, Kihyun},
  year = 2025,
  month = jan,
  journal = {Journal of the European Mathematical Society},
  issn = {1435-9855},
  doi = {10.4171/jems/1569}
}

@misc{kim_rigidity_2026,
  title = {Rigidity Results in Multi-Bubble Dynamics for Non-Radial Energy-Critical Heat Equation},
  author = {Kim, Kihyun and Merle, Frank},
  year = 2026,
  month = jan,
  number = {arXiv:2601.12517},
  eprint = {2601.12517},
  publisher = {arXiv},
  doi = {10.48550/arXiv.2601.12517},
  archiveprefix = {arXiv}
}

@article{krieger_full_2014,
  title = {Full Range of Blow up Exponents for the Quintic Wave Equation in Three Dimensions},
  author = {Krieger, Joachim and Schlag, Wilhelm},
  year = 2014,
  journal = {Journal de Math\'ematiques Pures et Appliqu\'ees},
  volume = {101},
  number = {6},
  eprint = {1212.3795},
  pages = {873--900},
  issn = {0021-7824},
  doi = {10.1016/j.matpur.2013.10.008},
  archiveprefix = {arXiv}
}

@article{krieger_renormalization_2008,
  title = {Renormalization and Blow up for Charge One Equivariant Critical Wave Maps},
  author = {Krieger, J. and Schlag, W. and Tataru, D.},
  year = 2008,
  month = mar,
  journal = {Inventiones mathematicae},
  volume = {171},
  number = {3},
  pages = {543--615},
  issn = {1432-1297},
  doi = {10.1007/s00222-007-0089-3}
}

@article{krieger_slow_2009,
  title = {Slow Blow-up Solutions for the {{H1}}({{R3}}) Critical Focusing Semilinear Wave Equation},
  author = {Krieger, Joachim and Schlag, Wilhelm and Tataru, Daniel},
  year = 2009,
  month = mar,
  journal = {Duke Mathematical Journal},
  volume = {147},
  number = {1},
  pages = {1--53},
  publisher = {Duke University Press},
  issn = {0012-7094, 1547-7398},
  doi = {10.1215/00127094-2009-005}
}

@article{martel_blow_2015,
  ids = {martel_blow_2012},
  title = {Blow up for the Critical {{gKdV}} Equation {{III}}: Exotic Regimes},
  shorttitle = {Blow up for the Critical {{gKdV}} Equation {{III}}},
  author = {Martel, Yvan and Merle, Frank and Rapha{\"e}l, Pierre},
  year = 2015,
  month = jun,
  journal = {Annali Scienze},
  pages = {575--631},
  issn = {2036-2145},
  doi = {10.2422/2036-2145.201209_004},
  copyright = {Copyright (c) 2015 Annali della Scuola Normale Superiore di Pisa. Classe di Scienze}
}

@article{martel_construction_2016,
  title = {Construction of Multi-Solitons for the Energy-Critical Wave Equation in Dimension 5},
  author = {Martel, Yvan and Merle, Frank},
  year = 2016,
  month = dec,
  journal = {Archive for Rational Mechanics and Analysis},
  volume = {222},
  number = {3},
  eprint = {1504.01595},
  pages = {1113--1160},
  issn = {0003-9527, 1432-0673},
  doi = {10.1007/s00205-016-1018-7},
  archiveprefix = {arXiv}
}

@article{martel_inelasticity_2018,
  title = {Inelasticity of Soliton Collisions for the {{5D}} Energy Critical Wave Equation},
  author = {Martel, Yvan and Merle, Frank},
  year = 2018,
  month = dec,
  journal = {Inventiones mathematicae},
  volume = {214},
  number = {3},
  eprint = {1708.09712},
  pages = {1267--1363},
  issn = {0020-9910, 1432-1297},
  doi = {10.1007/s00222-018-0822-0},
  archiveprefix = {arXiv}
}

@article{martel_strongly_2018,
  ids = {martel_strongly_2015},
  title = {{Strongly interacting blow up bubbles for the mass critical nonlinear Schr\"odinger equation}},
  author = {Martel, Yvan and Rapha{\"e}l, Pierre},
  year = 2018,
  journal = {Annales scientifiques de l'\'Ecole Normale Sup\'erieure},
  volume = {51},
  number = {3},
  pages = {701--737},
  issn = {1873-2151},
  doi = {10.24033/asens.2364}
}

@article{merle_blow_2011-2,
  ids = {merle_blow_2011-1},
  title = {Blow up Dynamics for Smooth Equivariant Solutions to the Energy Critical {{Schr\"odinger}} Map},
  author = {Merle, Frank and Rapha{\"e}l, Pierre and Rodnianski, Igor},
  year = 2011,
  month = mar,
  journal = {Comptes Rendus Mathematique},
  volume = {349},
  number = {5},
  pages = {279--283},
  issn = {1631-073X},
  doi = {10.1016/j.crma.2011.01.026}
}

@article{merle_determination_2003,
  title = {Determination of the {{Blow-Up Rate}} for the {{Semilinear Wave Equation}}},
  author = {Merle, Frank and Zaag, Hatem},
  year = 2003,
  journal = {American Journal of Mathematics},
  volume = {125},
  number = {5},
  eprint = {25099211},
  eprinttype = {jstor},
  pages = {1147--1164},
  publisher = {Johns Hopkins University Press},
  issn = {0002-9327}
}

@article{merle_type_2015,
  title = {Type {{II}} Blow up for the Energy Supercritical {{NLS}}},
  author = {Merle, Frank and Rapha{\"e}l, Pierre and Rodnianski, Igor},
  year = 2015,
  journal = {Cambridge Journal of Mathematics},
  volume = {3},
  number = {4},
  eprint = {1407.1415},
  pages = {439--617},
  issn = {21680930},
  doi = {10.4310/cjm.2015.v3.n4.a1},
  archiveprefix = {arXiv}
}

@article{merle_universality_2004,
  title = {On Universality of Blow-up Profile for {{L}} 2 Critical Nonlinear {{Schrodinger}} Equation},
  author = {Merle, Frank and Rapha{\"e}l, Pierre},
  year = 2004,
  month = jun,
  journal = {Inventiones Mathematicae},
  volume = {156},
  number = {3},
  pages = {565--672},
  issn = {0020-9910, 1432-1297},
  doi = {10.1007/s00222-003-0346-z}
}

@article{raphael_quantized_2014,
  title = {Quantized Slow Blow up Dynamics for the Corotational Energy Critical Harmonic Heat Flow},
  author = {Rapha{\"e}l, Pierre and Schweyer, Remi},
  year = 2014,
  month = dec,
  journal = {Analysis \& PDE},
  volume = {7},
  number = {8},
  eprint = {1301.1859},
  pages = {1713--1805},
  issn = {1948-206X, 2157-5045},
  doi = {10.2140/apde.2014.7.1713},
  archiveprefix = {arXiv}
}

@article{raphael_stable_2012,
  title = {Stable Blow up Dynamics for the Critical Co-Rotational Wave Maps and Equivariant {{Yang-Mills}} Problems},
  author = {Rapha{\"e}l, Pierre and Rodnianski, Igor},
  year = 2012,
  month = jun,
  journal = {Publications math\'ematiques de l'IH\'ES},
  volume = {115},
  number = {1},
  eprint = {0911.0692},
  pages = {1--122},
  issn = {1618-1913},
  doi = {10.1007/s10240-011-0037-z},
  archiveprefix = {arXiv}
}

@article{rodnianski_formation_2010,
  title = {On the Formation of Singularities in the Critical {{O}}(3) {$\sigma$}-Model},
  author = {Rodnianski, Igor and Sterbenz, Jacob},
  year = 2010,
  journal = {Annals of Mathematics},
  volume = {172},
  number = {1},
  eprint = {20752269},
  eprinttype = {jstor},
  pages = {187--242},
  issn = {0003486X},
  arxiv = {math/0605023}
}

@article{shatah_weak_1988,
  title = {Weak Solutions and Development of Singularities of the {{SU}}(2) {$\sigma$}-Model},
  author = {Shatah, Jalal},
  year = 1988,
  month = jun,
  journal = {Communications on Pure and Applied Mathematics},
  volume = {41},
  number = {4},
  pages = {459--469},
  issn = {00103640, 10970312},
  doi = {10.1002/cpa.3160410405}
}

@article{yuan_multi-solitons_2019,
  title = {On Multi-Solitons for the Energy-Critical Wave Equation in Dimension 5},
  author = {Yuan, Xu},
  year = 2019,
  month = dec,
  journal = {Nonlinearity},
  volume = {32},
  number = {12},
  eprint = {1809.05414},
  pages = {5017--5048},
  issn = {0951-7715, 1361-6544},
  doi = {10.1088/1361-6544/ab46ec},
  archiveprefix = {arXiv}
}
\end{document}